\documentclass{amsart}

\makeatletter
\@namedef{subjclassname@2020}{%
  \textup{2020} Mathematics Subject Classification}

\makeatother 

\usepackage{amsthm,amssymb,amsfonts,latexsym,mathtools,thmtools,amsmath,lscape}
\usepackage[T1]{fontenc}
\usepackage{tikz-cd} 
\usepackage{enumitem} 
\usepackage{longtable}
\usepackage{mathrsfs} 
\usepackage{array}
\usepackage{multirow}
\usepackage{hyperref} 
\usepackage{hyperref}
\usepackage[all]{xy}
\usepackage{booktabs}
\hypersetup{
    colorlinks=true,
    linkcolor=blue,
    filecolor=blue,      
    urlcolor=blue,
    linktocpage=true
}

\newtheorem{theorem}{Theorem}[section]
\newtheorem{lemma}[theorem]{Lemma}
\newtheorem{proposition}[theorem]{Proposition}

\theoremstyle{definition}
\newtheorem{definition}[theorem]{Definition}

\newtheorem{remark}[theorem]{Remark}

\newtheorem{example}[theorem]{Example}

\theoremstyle{remark}

\numberwithin{equation}{section}

\begin{document}

\title[Smooth geometry of skew PBW extensions I]{Smooth geometry of skew PBW extensions over commutative polynomial rings I}


\author{Andr\'es Rubiano}
\address{Universidad Nacional de Colombia - Sede Bogot\'a}
\curraddr{Campus Universitario}
\email{arubianos@unal.edu.co}
\address{Universidad ECCI}
\curraddr{Campus Universitario}
\email{arubianos@ecci.edu.co}
\thanks{}


\author{Armando Reyes}
\address{Universidad Nacional de Colombia - Sede Bogot\'a}
\curraddr{Campus Universitario}
\email{mareyesv@unal.edu.co}

\thanks{}



\date{}

\dedicatory{Dedicated to Professor Oswaldo Lezama on the Occasion of His 68th Birthday}

\begin{abstract} 

In this paper, we investigate the differential smoothness of skew PBW extensions over commutative polynomial rings on one and two indeterminates.

\end{abstract}

\maketitle


\section{Introduction}

Ore \cite{Ore1931, Ore1933} introduced a kind of noncommutative polynomial rings which has become one of most basic and useful constructions in ring theory and noncommutative algebra. For an associative and unital ring $R$, an endomorphism $\sigma$ of $R$ and a $\sigma$-derivation $\delta$ of $R$, the {\em Ore extension} or {\em skew polynomial ring} of $R$ is obtained by adding a single generator $x$ to $R$ subject to the relation $xr = \sigma(r) x + \delta(r)$ for all $r\in R$. This Ore extension of $R$ is denoted by $R[x; \sigma, \delta]$. As one can appreciate in the literature, a lot of papers and books have been published concerning ring-theoretical, homological, geometrical properties and applications of these extensions (e.g. \cite{BrownGoodearl2002, BuesoTorrecillasVerschoren2003, Fajardoetal2020, Fajardoetal2024, GoodearlLetzter1994, GoodearlWarfield2004, McConnellRobson2001, Li2002, Rosenberg1995, SeilerBook2010} and references therein). 

On the other hand, Bell and Goodearl \cite{BellGoodearl1988} defined the {\em Poincar\'e-Birkhoff-Witt} ({\em PBW} for short) {\em extensions} with the aim of cover several families of generalized operator rings as the enveloping algebra of a finite-dimensional Lie algebra, Weyl algebras, differential operators over Lie algebras, the twisted or smash product differential operator rings and universal enveloping rings \cite[Section 5]{BellGoodearl1988}. Different properties of PBW extensions have been studied by some researchers \cite{AkalanMarubayashi2016, Giesbrechtetal2002, Giesbrechtetal2014, Goodearl1990, MarubayashiZhang1996, Matczuk1988, SeilerBook2010, Tumwesigyeetal2020}.

With the aim of generalizing Ore extensions of injective type (that is, $R[x; \sigma, \delta]$ with $\sigma$ an injective map) and PBW extensions, Gallego and Lezama \cite{GallegoLezama2010} introduced the notion of {\em skew PBW} ({\em SPBW}) {\em extension}. Over the years several authors have shown that SPBW extensions also generalize families of noncommutative algebras such as {\em 3-dimensional skew polynomial algebras} introduced by Bell and Smith \cite{BellSmith1990}, {\em diffusion algebras} defined by Isaev et al. \cite{IsaevPyatovRittenberg2001, PyatovTwarock2002}, {\em ambiskew polynomial rings} introduced by Jordan \cite{Jordan2000, JordanWells1996}, {\em solvable polynomial rings} introduced by Kandri-Rody and Weispfenning  \cite{KandryWeispfenning1990}, {\em almost normalizing extensions} defined by McConnell and Robson \cite{McConnellRobson2001}, and {\em skew bi-quadratic algebras with PBW basis} introduced by Bavula \cite{Bavula2023}. As expected, there are different relations between SPBW extensions and other noncommutative algebras having PBW bases defined in the literature (e.g. \cite{Apel1988, Bavula2020, BuesoTorrecillasVerschoren2003, GomezTorrecillas2014, Levandovskyy2005, Li2002, McConnellRobson2001, Rosenberg1995, SeilerBook2010}). 

In this paper we are interested in the notion of {\em differential smoothnness} of algebras defined by Brzezi{\'n}ski and Sitarz \cite{BrzezinskiSitarz2017}. Before saying some words about it, we recall key aspects of {\em connections} and {\em differential calculi} in noncommutative geometry.

The {\em theory of connections} in noncommutative geometry is well-known (for more details, see the beautiful treatments presented by Connes \cite{Connes1994} or Giachetta et al. \cite{Giachettaetal2005}). Briefly, one considers a differential graded algebra $\Omega A = \bigoplus\limits_{n = 0} \Omega^{n} A$ over a $\Bbbk$-algebra $A = \Omega^{0} A$ with $\Bbbk$ a field, and then defines a {\em connection} in a left $A$-module $M$ as a linear map $\nabla^{0} : M \to \Omega^{1} A \otimes_A M$ that satisfies the Leibniz's rule $\nabla^0(am) = da \otimes_A m + a\nabla^{0}(m)$ for all $m\in M$ and $a \in A$. As it can be seen, this is a noncommutative definition obtained by a replacement of commutative algebras of functions on a manifold $X$, and their modules of sections of a vector bundle over $X$ (in the classical definition of a connection), by noncommutative algebras and their general one-sided modules. Just as Brzezi{\'n}ski said, \textquotedblleft this captures very well the classical context in which connections appear and brings it successfully to the realm of noncommutative geometry\textquotedblright\ \cite[p. 557]{Brzezinski2008}.

Brzezi{\'n}ski in his paper noted that, on the algebraic side, this definition of connection seems to be only a half of a more general picture. In the first place, a noncommutative connection is defined by using the tensor functor, and as is well-known, this functor has a right adjoint, the {\em hom-functor}, so it is natural to ask whether it is possible to introduce connection-like objects defined with the use of the hom-functor. In the second place, the vector space dual to $M$ is a right $A$-module and a left connection in the above sense does not induce a right connection on the dual of $M$, so having in mind the adjointness properties between tensor and hom functors, the induced map necessarily involves the hom-functor. 

Motivated by all these facts, Brzezi{\'n}ski \cite{Brzezinski2008} showed that there is a natural and potentially rich theory of connnection-like objects defined as maps on the spaces of morphisms of modules. Due to the role of spaces of homomorphisms, these objects are termed {\em hom-connections} (also are called {\em divergences} due to that if $A$ is an algebra of functions on the Euclidean space $\mathbb{R}^n$ and $\Omega^{1}(A)$ is the standard module of one-form, then we obtain the classical divergence of the elementary vector calculus \cite[p. 892]{Brzezinski2011}). As a matter of fact, he proved that hom-connections arise naturally from (strong) connections in {\em noncommutative principal bundles}, and that every left connection on a bimodule (in the sense of Cuntz and Quillen \cite{CuntzQuillen1995}) gives rise to a hom-connection. 
Brzezi{\'n}ski also studied the induction procedure of hom-connections via differentiable bimodules (and hence, via maps of differential graded algebras), and proved that any hom-connection can be extended to {\em higher forms}. He introduced the notion of {\em curvature} and showed that a consecutive application of hom-connections can be expressed in terms of the curvature, which leads to a chain complex associated to a {\em flat} (i.e. curvature-zero) {\em hom-connection} (this chain complex and its homology can be considered as dual complements of the cochain complex associated to a connection and the twisted cohomology, which is crucial in the theory of noncommutative differential fibrations \cite{BeggsBrzezinski2005}).

Two years later, Brzezi{\'n}ski et al. \cite{BrzezinskiElKaoutitLomp2010} presented a construction of {\em differential calculi} which admits hom-connections. This construction is based on the use of {\em twisted multi-derivations}, where the constructed first-order calculus $\Omega^{1}(A)$ is free as a left and right $A$-module; $\Omega^{1}(A)$ should be understood as a module of sections on the cotangent bundle over a manifold represented by $A$, and hence their construction corresponds to parallelizable manifolds or to an algebra of functions on a local chart. One year later, Brzezi{\'n}ski asserted that \textquotedblleft one should expect $\Omega^1(A)$ to be a finitely generated and projective module over $A$ (thus corresponding to sections of a non-trivial vector bundle by the Serre-Swan theorem)\textquotedblright\ \cite[p. 885]{Brzezinski2011}. In his paper, he extended the construction in \cite{BrzezinskiElKaoutitLomp2010} to finitely generated and projective modules.

Related to differential calculi, we have the {\em smoothness of algebras}. Briefly, and as Brzezi\'nski and Lomp said \cite[Section 1]{BrzezinskiLomp2018}, the study of this smoothness goes back at least to Grothendieck's EGA \cite{Grothendieck1964}. The concept of a {\em formally smooth commutative} ({\em topological}) {\em algebra} introduced by him was extended to the noncommutative setting by Schelter \cite{Schelter1986}. An algebra is {\em formally smooth} if and only if the kernel of the multiplication map is projective as a bimodule. This notion arose as a replacement of a far too general definition based on the finiteness of the global dimension; Cuntz and Quillen \cite{CuntzQuillen1995} called these algebras {\em quasi-free}. Precisely, the notion of smoothness based on the finiteness of this dimension was refined by Stafford and Zhang \cite{StaffordZhang1994}, where a Noetherian algebra is said to be {\em smooth} provided that it has a finite global dimension equal to the homological dimension of all its simple modules. In the homological setting, Van den Bergh \cite{VandenBergh1998} called an algebra {\em homologically smooth} if it admits a finite resolution by finitely generated projective bimodules. The characterization of this kind of smoothness for the noncommutative pillow, the quantum teardrops, and quantum homogeneous spaces was made by Brzezi{\'n}ski \cite{Brzezinski2008, Brzezinski2014} and Kr\"ahmer \cite{Krahmer2012}, respectively.

Brzezi{\'n}ski and Sitarz \cite{BrzezinskiSitarz2017} defined other notion of smoothness of algebras, termed {\em differential smoothness} due to the use of differential graded algebras of a specified dimension that admits a noncommutative version of the Hodge star isomorphism, which considers the existence of a top form in a differential calculus over an algebra together with a string version of the Poincar\'e duality realized as an isomorphism between complexes of differential and integral forms. This new notion of smoothness is different and more constructive than the homological smoothness mentioned above. \textquotedblleft The idea behind the {\em differential smoothness} of algebras is rooted in the observation that a classical smooth orientable manifold, in addition to de Rham complex of differential forms, admits also the complex of {\em integral forms} isomorphic to the de Rham complex \cite[Section 4.5]{Manin1997}. The de Rham differential can be understood as a special left connection, while the boundary operator in the complex of integral forms is an example of a {\em right connection}\textquotedblright\ \cite[p. 413]{BrzezinskiSitarz2017}.

Several authors (e.g. \cite{Brzezinski2015, Brzezinski2016, BrzezinskiElKaoutitLomp2010, BrzezinskiLomp2018, BrzezinskiSitarz2017, DuboisVioletteKernerMadore1990, Karacuha2015, KaracuhaLomp2014, ReyesSarmiento2022}) have characterized the differential smoothness of algebras such as the quantum two - and three - spheres, disc, plane, the noncommutative torus, the coordinate algebras of the quantum group $SU_q(2)$, the noncommutative pillow algebra, the quantum cone algebras, the quantum polynomial algebras, Hopf algebra domains of Gelfand-Kirillov dimension two that are not PI, families of Ore extensions, some 3-dimensional skew polynomial algebras, diffusion algebras in three generators, and noncommutative coordinate algebras of deformations of several examples of classical orbifolds such as the pillow orbifold, singular cones and lens spaces. An interesting fact is that some of these algebras are also homologically smooth in the Van den Bergh's sense.

Considering the active research on differential smoothness of noncommutative algebras, and having in mind that ring-theoretical and geometrical properties of SPBW extensions (and hence of PBW extensions) have been investigated by different authors \cite{AbdiTalebi2023, Artamonov2015, GomezTorrecillas2014, Fajardoetal2020, Hamidizadehetal2020, HashemiKhalilnezhadAlhevaz2017, HigueraReyes2023, LezamaGomez2019, NinoRamirezReyes2020, ReyesRodriguez2021, ReyesSuarez2020, SuarezReyesSuarez2023, Tumwesigyeetal2020}, our purpose in this paper is to investigate this smoothness for the SPBW extensions over the commutative polynomial rings $\Bbbk[t]$ and $\Bbbk[t_1, t_2]$ (in a sequel paper \cite{RubianoReyes2024DSSPBWKt3n} we study the differential smoothness in the case of commutative polynomial rings generated on three and more indeterminates). Since these extensions are more general than 3-dimensional skew polynomial algebras \cite{BellSmith1990}, diffusion algebras \cite{IsaevPyatovRittenberg2001}, and skew bi-quadratic algebras \cite{Bavula2023} (see also double Ore extensions \cite{ZhangZhang2008, ZhangZhang2009}), and that the differential smoothness of all these families of algebras has been investigated in \cite{ReyesSarmiento2022, RubianoReyes2024DSBiquadraticAlgebras, RubianoReyes2024DSDoubleOreExtensions}, this paper is a sequel of the research of the smooth geometry of SPBW extensions from Brzezi{\'n}ski and Sitarz's point of view. In this way, we contribute to the study of the noncommutative geometry (algebraic and differential) of SPBW extensions that has been carried out by Lezama \cite{Lezama2020, Lezama2021, LezamaGomez2019, LezamaLatorre2017} and other people \cite{ChaconReyes2024, NinoReyes2024, ReyesSuarez2016, SuarezReyesSuarez2023}. 
 
The article is organized as follows. In Section \ref{sec2} we recall the definitions and preliminaries on SPBW extensions and differential smoothness of algebras in order to set up notation and render this paper self-contained. Next, Section \ref{DICSPBWKt} contains the first original results on the paper. We extend Brzezi{\'n}ski's ideas developed for skew polynomial rings of the commutative polynomial ring $\Bbbk[t]$ \cite{Brzezinski2015} (Example \ref{Brzezinski2015DSOEbiquadratic}) to the setting of SPBW extensions over $\Bbbk[t]$. Due to the length of the non-trivial computations, first we take as toy models the SPBW extensions generated by two and three indeterminates (Sections \ref{SPBWTMTwoI} and \ref{SPBWTMThreeI}, respectively), while the general case is presented in Section \ref{SPBWTMGeneralCase}. Theorems \ref{smoothPBW2}, \ref{smoothPBW3} and \ref{smoothPBWn} are the key results that establish sufficient conditions to assert that a SPBW extension over $\Bbbk[t]$ is differentially smooth. In Section \ref{DICSPBWKt1t2}, we study the differential smoothness of SPBW extensions over $\Bbbk[t_1, t_2]$. As it can be seen, the computations are highly non-trivial. Just as we did in Section \ref{DICSPBWKt}, we divide our treatment in the case of two, three and $n$ indeterminates (Sections \ref{SPBWTMTwoI1}, \ref{SPBWTMTwoI3}, and \ref{SPBWTMTwoIn}, respectively). The important results in this section are Theorems \ref{smoothPBW22}, \ref{smoothPBW23} and \ref{smoothPBW2n}. Finally, in Section \ref{FutureworkDSSPBW} we say a few words about a future work related to the sequel paper.

Throughout the paper, the word ring means an associative ring with identity not necessarily commutative. $\mathbb{N}$ denotes the set of natural numbers including zero, $K$ and $\Bbbk$ denote a commutative ring with identity and a field, respectively. ${\rm Aut}(R)$ denotes the set of automorphisms of the ring $R$.

\section{Definitions and preliminaries}\label{sec2}

\subsection{Skew Poincar\'e-Birkhoff-Witt extensions}\label{SPBWdefinitionspreliminaries}

\begin{definition}[{\cite[Definition 1]{GallegoLezama2010}}]\label{defpbwextension} 
Let $R$ and $A$ be rings. We say that $A$ is a {\it SPBW extension} over $R$ if the following conditions hold:
\begin{itemize}
\item [(i)] $R$ is a subring of $A$ sharing the same identity element.

\item [(ii)] There exist elements $x_1, \ldots, x_n \in A\ \backslash\ R$ such that $A$ is a left free $R$-module with basis given by the set $\text{Mon}(A):= \{ x^{\alpha} = x_{1}^{\alpha_1} \cdots x_{n}^{\alpha_n}\mid \alpha =(\alpha_1, \ldots, \alpha_n) \in \mathbb{N}^n\}$.

\item[(iii)] For each $1 \leq i \leq n$ and any $r \in R\  \backslash\ \{0\}$, there exists an element $c_{i,r} \in R \ \backslash\ \{0\}$ such that $x_ir - c_{i,r}x_i \in R$.

\item[\rm (iv)]For $1\leq i, j\leq n$, there exists an element $d_{i,j}\in R\ \backslash\ \{0\}$ such that
\[
x_jx_i-d_{i,j}x_ix_j\in R+Rx_1+\cdots +Rx_n,
\]

i.e., there exist elements $r_0^{(i,j)}, r_1^{(i,j)}, \dotsc, r_n^{(i,j)} \in R$ with
\begin{center}\label{relSPBW}
$x_jx_i - d_{i,j}x_ix_j = r_0^{(i,j)} + \sum_{k=1}^{n} r_k^{(i,j)}x_k$.    
\end{center}
\end{itemize}
\end{definition}

We use freely the notation $A = \sigma(R)\langle x_1,\dotsc, x_n\rangle$ to denote a SPBW extension $A$ over a ring $R$ in the indeterminates $x_1, \dotsc, x_n$. $R$ is called the {\em ring of coefficients} of the extension $A$. 

Since $\text{Mon}(A)$ is a left $R$-basis of $A$, the elements $c_{i,r}$ and $d_{i,j}$ in Definition \ref{defpbwextension} are unique. Every element $f \in A\  \backslash\ \{0\}$ has a unique representation as $f = \sum_{i=0}^tr_iX_i$, with $r_i \in R\ \backslash\ \{0\}$ and $X_i \in \text{Mon}(A)$ for $0 \leq i \leq t$ with $X_0=1$. When necessary, we use the notation $f = \sum_{i=0}^tr_iY_i$. For $X=x^{\alpha}\in\text{Mon}(A)$, ${\rm exp}(X):=\alpha$ and $\deg(X):=|\alpha|$. Let $\deg(f):=\max\{\deg(X_i)\}_{i=1}^t$ \cite[Remark 2 and Definition 6]{GallegoLezama2010}. 

If $A = \sigma(R)\langle x_1,\dotsc, x_n\rangle$ is a SPBW extension over $R$, then for each $1 \leq i \leq n$, there exist an injective endomorphism $\sigma_i : R \to R$ and a $\sigma_i$-derivation $\delta_i: R \to R$ such that $x_ir = \sigma_i(r)x_i + \delta_i(r)$, for each $r\in R$ \cite[Proposition 3]{GallegoLezama2010}. We use the notation $\Sigma:=\{\sigma_1,\dots,\sigma_n\}$ and $\Delta:=\{\delta_1,\dots,\delta_n\}$, and say that the pair $(\Sigma, \Delta)$ is a \textit{system of endomorphisms and $\Sigma$-derivations} of $R$ with respect to $A$. For $\alpha = (\alpha_1, \dots , \alpha_n) \in \mathbb{N}^n$, $\sigma^{\alpha}:= \sigma_1^{\alpha_1}\circ \cdots \circ \sigma_n^{\alpha_n}$, $\delta^{\alpha} := \delta_1^{\alpha_1} \circ \cdots \circ \delta_n^{\alpha_n}$, where $\circ$ denotes the classical composition of functions. 

\begin{definition}[{\cite[Definition 4]{GallegoLezama2010}}, {\cite[Definition 2.3 (ii)]{LezamaAcostaReyes2015}}]\label{quasicommutativebijective}
Consider a SPBW extension $A = \sigma(R)\langle x_1,\dotsc, x_n\rangle$ over $R$.
\begin{itemize}
\item [(i)] $A$ is called {\em quasi-commutative} if the conditions (iii) - (iv) in Definition (\ref{defpbwextension}) are replaced by the following:
\begin{itemize}
\item For every $1 \leq i \leq n$ and $r \in R \ \backslash\ \{0\}$ there exists $c_{i,j} \in R\ \backslash\ \{0\}$ such that $x_ir = c_{i,r}x_i$.
        
\item For every $1 \leq i,j \leq n$, there exists $d_{i,j} \in R\ \backslash\ \{0\}$ such that $x_jx_i = d_{i,j}x_ix_j$.
\end{itemize}
    
\item [(ii)] $A$ is {\it bijective} if $\sigma_i$ is bijective, for every $1 \leq i \leq n$, and $d_{i,j}$ is invertible, for any $1 \leq i < j \leq n$.
    
\item [(iii)] If $\sigma_i$ is the identity map of $R$ for each $i = 1, \dotsc, n$, then we say that $A$ is of {\em derivation type}. Similarly, if $\delta_i$ is zero, for every $i$, then $A$ is called of {\em endomorphism type}.
\item [(iv)] $A$ is said to be {\em semi-commutative} if it is quasi-commutative and $x_ir = rx_i$, for each $i$ and every $r\in R$.
\end{itemize}
\end{definition}

Next, we consider some interesting families examples of SPBW extensions.

\begin{example}\label{ExamplesSPBWextensionsDS}
\begin{enumerate}
\item [\rm (i)] SPBW extensions of endomorphism type over a ring are more general than iterated Ore extensions of endomorphism type of the same ring. Let us illustrate the situation with two and three indeterminates.
		
For the iterated Ore extension of endomorphism type $R[x;\sigma_x][y;\sigma_y]$, if $r\in R$ then we have the following relations: $xr = \sigma_x(r)x$, $yr = \sigma_y(r)y$, and $yx = \sigma_y(x)y$. Now, if we have $\sigma(R)\langle x, y\rangle$ a SPBW extension of endomorphism type over $R$, then for any $r\in R$, Definition \ref{defpbwextension} establishes that $xr=\sigma_1(r)x$, $yr=\sigma_2(r)y$, and $yx = d_{1,2}xy + r_0 + r_1x + r_2y$, for some elements $d_{1,2}, r_0, r_1$ and $r_2$ belong to $R$. 
	
If we have the iterated Ore extension $R[x;\sigma_x][y;\sigma_y][z;\sigma_z]$, then for any $r\in R$, $xr = \sigma_x(r)x$, $yr = \sigma_y(r)y$, $zr = \sigma_z(r)z$, $yx = \sigma_y(x)y$, $zx = \sigma_z(x)z$, $zy = \sigma_z(y)z$. For the SPBW extension of endomorphism type $\sigma(R)\langle x, y, z\rangle$, $xr=\sigma_1(r)x$, $yr=\sigma_2(r)y$, $zr = \sigma_3(r)z$, $yx = d_{1,2}xy + r_0 + r_1x + r_2y + r_3z$, $zx = d_{1,3}xz + r_0' + r_1'x + r_2'y + r_3'z$, and $zy = d_{2,3}yz + r_0'' + r_1''x + r_2''y + r_3''z$, for some elements $d_{1,2}, d_{1,3}, d_{2,3}, r_0, r_0', r_0'', r_1, r_1', r_1'', r_2, r_2', r_2'', r_3$, $r_3', r_3''$ of $R$. As the number of indeterminates increases, the differences between both algebraic structures are more remarkable.

\item [\rm (ii)] From Definition \ref{defpbwextension} (iv), it is clear that SPBW extensions are more general than iterated skew polynomial rings. For example, universal enveloping algebras of finite dimensional Lie algebras and some 3-dimensional skew polynomial algebras in the sense of Bell and Smith \cite{BellSmith1990} cannot be expressed as iterated skew polynomial rings but are SPBW extensions. Quasi-commutative SPBW extensions are isomorphic to iterated Ore extensions of endomorphism type \cite[ Theorem 2.3]{LezamaReyes2014}. 

\item [\rm (iii)] PBW extensions introduced by Bell and Goodearl \cite{BellGoodearl1988} are particular examples of SPBW extensions. More exactly, the first objects satisfy the relation $x_ir = rx_i + \delta_i(r)$ for every $i = 1,\dotsc, n$ and each $r\in R$, and the elements $d_{ij}$ in Definition \ref{defpbwextension}  (iv) are equal to the identity of $R$. As examples of PBW extensions, we mention the following:  the enveloping algebra of a finite-dimensional Lie algebra; any differential operator ring $R[\theta_1,\dotsc, \theta_1;\delta_1,\dotsc, \delta_n]$ formed from commuting derivations $\delta_1,\dotsc, \delta_n$; differential operators introduced by Rinehart; twisted or smash product differential operator rings, and others \cite[p. 27]{BellGoodearl1988}. 

\item [\rm (iv)] {\em 3-dimensional skew polynomial algebras} were defined by Bell and Smith \cite{BellSmith1990}. Briefly, a $3$-{\em dimensional algebra} $A$ is a $\Bbbk$-algebra generated by the indeterminates $x, y, z$ subject to the relations 
\[
yz-\alpha zy = \lambda, \quad zx - \beta xz = \mu, \quad {\rm and} \quad xy - \gamma yx = \nu, 
\]

where $\lambda,\mu,\nu \in \Bbbk x+\Bbbk y+\Bbbk z+\Bbbk$, and $\alpha, \beta, \gamma \in \Bbbk^{*}$. $A$ is called a \textit{3-dimensional skew polynomial $\Bbbk$-algebra} if the set $\left\{x^iy^jz^k\mid i,j,k\geq 0\right\}$ forms a $\Bbbk$-basis of the algebra. Up to isomorphism, there are fifteen 3-dimensional skew polynomial $\Bbbk$-algebras \cite{Rosenberg1995}, Theorem C4.3.1] (see also \cite{RedmanPhDThesis1996, Redman1999, ReyesSuarez20173D}).

\item [\rm (v)] {\em Diffusion algebras} were introduced from the physicist point of view by Isaev et al. \cite{IsaevPyatovRittenberg2001} as quadratic algebras that appear as algebras of operators that model the stochastic flow of motion of particles in a one dimensional discrete lattice, while Pyatov and Twarock \cite{PyatovTwarock2002} presented a construction formalism for these algebras and to use the latter to prove the results in \cite{IsaevPyatovRittenberg2001}: \textquotedblleft Diffusion algebras play a key role in the understanding of one-dimensional stochastic processes. In the case of $N$ species of particles with only nearest-neighbor interactions with exclusion on a one-dimensional lattice, diffusion algebras are useful tools in finding expressions for the probability distribution of the stationary state of these processes. Following the idea of matrix product states, the latter are given in terms of monomials built from the generators of a quadratic algebra\textquotedblright\ \cite[p. 3268]{PyatovTwarock2002}.

Following Pyatov and Twarock's notation and let $\alpha, \beta$ be two elements belonging to the set $I_N := \{1, \dotsc, n\}$ with $\alpha < \beta$. Consider quadratic relations of the form
\begin{equation}\label{PyatovTwarock2002(1)}
    g_{\alpha \beta} D_{\alpha} D_{\beta} - g_{\beta \alpha} D_{\beta} D_{\alpha} = x_{\beta} D_{\alpha} - x_{\alpha} D_{\beta},
\end{equation}

with $g_{\alpha \beta} \in \mathbb{R} \ \backslash \ \{0\}, \ g_{\beta \alpha} \in \mathbb{R}$, and $x_{\alpha}, x_{\beta} \in \mathbb{C}$.

From \cite[Definition 1.1]{PyatovTwarock2002}, an algebra with set of generators given by $\left\{ D_{\alpha} \mid \alpha \in I_N\right\}$ and relations of type {\rm (}\ref{PyatovTwarock2002(1)}{\rm )} is called {\em diffusion algebra}, if it admits a linear PBW-basis of ordered monomials of the form
\begin{equation}\label{PyatovTwarock2002(2)}
    D_{\alpha_1}^{k_1} D_{\alpha_2}^{k_2} \dotsb D_{\alpha_n}^{k_n},\quad {\rm with} \quad k_j \in \mathbb{N} \quad {\rm and} \quad \alpha_1 > \alpha_2 > \dotsb > \alpha_n.
\end{equation}

Due to physical reasons only relations with positive coefficients $g_{\alpha \beta} \in \mathbb{R}_{>0}$ and $g_{\beta \alpha} \in \mathbb{R}_{\ge 0}$ ($\alpha < \beta$) are relevant because they are interpreted as hopping rates in stochastic models \cite[p. 3268]{PyatovTwarock2002}.

\item [\rm (vi)] Let $n\ge 2$ be a natural number. A family $M = (m_{ij})_{i > j}$ of elements $m_{ij}$ belonging to $R$ ($1\le j < i \le n$) is called a {\em lower triangular half-matrix} with coefficients in $R$. The set of all such matrices is denoted by $L_n(R)$.

Bavula \cite[Section 1]{Bavula2023} defined for $\sigma = (\sigma_1, \dotsc, \sigma_n)$ an $n$-tuple of commuting endomorphisms of $R$, $\delta = (\delta_1, \dotsc, \delta_n)$ an $n$-tuple of $\sigma$-endomorphisms of $R$ (that is, $\delta_i$ is a $\sigma_i$-derivation of $R$ for $i=1,\dotsc, n$), $Q = (q_{ij})\in L_n(Z(R))$, $\mathbb{A}:= (a_{ij, k})$ where $a_{ij, k}\in R$, $1\le j < i \le n$ and $k = 1,\dotsc, n$, and $\mathbb{B}:= (b_{ij})\in L_n(R)$, the {\em skew bi-quadratic algebra} ({\em SBQA}) $A = R[x_1,\dotsc, x_n;\sigma, \delta, Q, \mathbb{A}, \mathbb{B}]$ as a ring generated by the ring $R$ and elements $x_1, \dotsc, x_n$ subject to the defining relations
\begin{align}
    x_ir = &\ \sigma_i(r)x_i + \delta_i(r),\quad {\rm for}\ i = 1, \dotsc, n,\ {\rm and\ every}\ r\in R, \label{Bavula2023(1)} \\
    x_ix_j - q_{ij}x_jx_i = &\ \sum_{k=1}^{n} a_{ij, k}x_k + b_{ij},\quad {\rm for\ all}\ j < i.\label{Bavula2023(2)}
\end{align}

If $\sigma_i = {\rm id}_R$ and $\delta_i = 0$ for $i = 1,\dotsc, n$, the ring $A$ is called the {\em bi-quadratic algebra} ({\em BQA}) and is denoted by $A = R[x_1, \dotsc, x_n; Q, \mathbb{A}, \mathbb{B}]$. $A$ has {\em PBW basis} if $A = \bigoplus\limits_{\alpha \in \mathbb{N}^{n}} Rx^{\alpha}$ where $x^{\alpha} = x_1^{\alpha_1}\dotsb x_n^{\alpha_n}$.

It is clear from the definition that bi-quadratic algebras having PBW basis are particular examples of SPBW extensions.

\end{enumerate}
\end{example}

\subsection{Differential smoothness}\label{DefinitionsandpreliminariesDSA}

We follow Brzezi\'nski and Sitarz's presentation on differential smoothness carried out in \cite[Section 2]{BrzezinskiSitarz2017} (c.f. \cite{Brzezinski2008, Brzezinski2014}).

\begin{definition}[{\cite[Section 2.1]{BrzezinskiSitarz2017}}]
\begin{enumerate}
    \item [\rm (i)] A {\em differential graded algebra} is a non-negatively graded algebra $\Omega$ with the product denoted by $\wedge$ together with a degree-one linear map $d:\Omega^{\bullet} \to \Omega^{\bullet +1}$ that satisfies the graded Leibniz's rule and is such that $d \circ d = 0$. 
    
    \item [\rm (ii)] A differential graded algebra $(\Omega, d)$ is a {\em calculus over an algebra} $A$ if $\Omega^0 A = A$ and $\Omega^n A = A\ dA \wedge dA \wedge \dotsb \wedge dA$ ($dA$ appears $n$-times) for all $n\in \mathbb{N}$ (this last is called the {\em density condition}). We write $(\Omega A, d)$ with $\Omega A = \bigoplus_{n\in \mathbb{N}} \Omega^{n}A$. By using the Leibniz's rule, it follows that $\Omega^n A = dA \wedge dA \wedge \dotsb \wedge dA\ A$. A differential calculus $\Omega A$ is said to be {\em connected} if ${\rm ker}(d\mid_{\Omega^0 A}) = \Bbbk$.
    
    \item [\rm (iii)] A calculus $(\Omega A, d)$ is said to have {\em dimension} $n$ if $\Omega^n A\neq 0$ and $\Omega^m A = 0$ for all $m > n$. An $n$-dimensional calculus $\Omega A$ {\em admits a volume form} if $\Omega^n A$ is isomorphic to $A$ as a left and right $A$-module. 
\end{enumerate}
\end{definition}

The existence of a right $A$-module isomorphism means that there is a free generator, say $\omega$, of $\Omega^n A$ (as a right $A$-module), i.e. $\omega \in \Omega^n A$, such that all elements of $\Omega^n A$ can be uniquely expressed as $\omega a$ with $a \in A$. If $\omega$ is also a free generator of $\Omega^n A$ as a left $A$-module, this is said to be a {\em volume form} on $\Omega A$.

The right $A$-module isomorphism $\Omega^n A \to A$ corresponding to a volume form $\omega$ is denoted by $\pi_{\omega}$, i.e.
\begin{equation}\label{BrzezinskiSitarz2017(2.1)}
\pi_{\omega} (\omega a) = a, \quad {\rm for\ all}\ a\in A.
\end{equation}

By using that $\Omega^n A$ is also isomorphic to $A$ as a left $A$-module, any free generator $\omega $ induces an algebra endomorphism $\nu_{\omega}$ of $A$ by the formula
\begin{equation}\label{BrzezinskiSitarz2017(2.2)}
    a \omega = \omega \nu_{\omega} (a).
\end{equation}

Note that if $\omega$ is a volume form, then $\nu_{\omega}$ is an algebra automorphism.

Now, we proceed to recall the key ingredients of the {\em integral calculus} on $A$ as dual to its differential calculus. For more details, see Brzezinski et al. \cite{Brzezinski2008, BrzezinskiElKaoutitLomp2010}.

Let $(\Omega A, d)$ be a differential calculus on $A$. The space of $n$-forms $\Omega^n A$ is an $A$-bimodule. Consider $\mathcal{I}_{n}A$ the right dual of $\Omega^{n}A$, the space of all right $A$-linear maps $\Omega^{n}A\rightarrow A$, that is, $\mathcal{I}_{n}A := {\rm Hom}_{A}(\Omega^{n}(A),A)$. Notice that each of the $\mathcal{I}_{n}A$ is an $A$-bimodule with the actions
\begin{align*}
    (a\cdot\phi\cdot b)(\omega)=a\phi(b\omega),\quad {\rm for\ all}\ \phi \in \mathcal{I}_{n}A,\ \omega \in \Omega^{n}A\ {\rm and}\ a,b \in A.
\end{align*}

The direct sum of all the $\mathcal{I}_{n}A$, that is, $\mathcal{I}A = \bigoplus\limits_{n} \mathcal{I}_n A$, is a right $\Omega A$-module with action given by
\begin{align}\label{BrzezinskiSitarz2017(2.3)}
    (\phi\cdot\omega)(\omega')=\phi(\omega\wedge\omega'),\quad {\rm for\ all}\ \phi\in\mathcal{I}_{n + m}A, \ \omega\in \Omega^{n}A \ {\rm and} \ \omega' \in \Omega^{m}A.
\end{align}

\begin{definition}[{\cite[Definition 2.1]{Brzezinski2008}}]
A {\em divergence} (also called {\em hom-connection}) on $A$ is a linear map $\nabla: \mathcal{I}_1 A \to A$ such that
\begin{equation}\label{BrzezinskiSitarz2017(2.4)}
    \nabla(\phi \cdot a) = \nabla(\phi) a + \phi(da), \quad {\rm for\ all}\ \phi \in \mathcal{I}_1 A \ {\rm and} \ a \in A.
\end{equation}  
\end{definition}

Note that a divergence can be extended to the whole of $\mathcal{I}A$, 
\[
\nabla_n: \mathcal{I}_{n+1} A \to \mathcal{I}_{n} A,
\]

by considering
\begin{equation}\label{BrzezinskiSitarz2017(2.5)}
\nabla_n(\phi)(\omega) = \nabla(\phi \cdot \omega) + (-1)^{n+1} \phi(d \omega), \quad {\rm for\ all}\ \phi \in \mathcal{I}_{n+1}(A)\ {\rm and} \ \omega \in \Omega^n A.
\end{equation}

By putting together (\ref{BrzezinskiSitarz2017(2.4)}) and (\ref{BrzezinskiSitarz2017(2.5)}), we get the Leibniz's rule 
\begin{equation}
    \nabla_n(\phi \cdot \omega) = \nabla_{m + n}(\phi) \cdot \omega + (-1)^{m + n} \phi \cdot d\omega,
\end{equation}

for all elements $\phi \in \mathcal{I}_{m + n + 1} A$ and $\omega \in \Omega^m A$ \cite[Lemma 3.2]{Brzezinski2008}. In the case $n = 0$, if ${\rm Hom}_A(A, M)$ is canonically identified with $M$, then $\nabla_0$ reduces to the classical Leibniz's rule.

\begin{definition}[{\cite[Definition 3.4]{Brzezinski2008}}]
The right $A$-module map 
$$
F = \nabla_0 \circ \nabla_1: {\rm Hom}_A(\Omega^{2} A, M) \to M
$$ is called a {\em curvature} of a hom-connection $(M, \nabla_0)$. $(M, \nabla_0)$ is said to be {\em flat} if its curvature is the zero map, that is, if $\nabla \circ \nabla_1 = 0$. This condition implies that $\nabla_n \circ \nabla_{n+1} = 0$ for all $n\in \mathbb{N}$.
\end{definition}

$\mathcal{I} A$ together with the $\nabla_n$ form a chain complex called the {\em complex of integral forms} over $A$. The cokernel map of $\nabla$, that is, $\Lambda: A \to {\rm Coker} \nabla = A / {\rm Im} \nabla$ is said to be the {\em integral on $A$ associated to} $\mathcal{I}A$.

Given a left $A$-module $X$ with action $a\cdot x$, for all $a\in A,\ x \in X$, and an algebra automorphism $\nu$ of $A$, the notation $^{\nu}X$ stands for $X$ with the $A$-module structure twisted by $\nu$, i.e. with the $A$-action $a\otimes x \mapsto \nu(a)\cdot x $.

The following definition of an \textit{integrable differential calculus} seeks to portray a version of Hodge star isomorphisms between the complex of differential forms of a differentiable manifold and a complex of dual modules of it \cite[p. 112]{Brzezinski2015}. 

\begin{definition}[{\cite[Definition 2.1]{BrzezinskiSitarz2017}}]
An $n$-dimensional differential calculus $(\Omega A, d)$ is said to be {\em integrable} if $(\Omega A, d)$ admits a complex of integral forms $(\mathcal{I}A, \nabla)$ for which there exist an algebra automorphism $\nu$ of $A$ and $A$-bimodule isomorphisms \linebreak $\Theta_k: \Omega^{k} A \to ^{\nu} \mathcal{I}_{n-k}A$, $k = 0, \dotsc, n$, rendering commmutative the following diagram:
\[
{\large{
\begin{tikzcd}
A \arrow{r}{d} \arrow{d}{\Theta_0} & \Omega^{1} A \arrow{d}{\Theta_1} \arrow{r}{d} & \Omega^2 A  \arrow{d}{\Theta_2} \arrow{r}{d} & \dotsb \arrow{r}{d} & \Omega^{n-1} A \arrow{d}{\Theta_{n-1}} \arrow{r}{d} & \Omega^n A  \arrow{d}{\Theta_n} \\ ^{\nu} \mathcal{I}_n A \arrow[swap]{r}{\nabla_{n-1}} & ^{\nu} \mathcal{I}_{n-1} A \arrow[swap]{r}{\nabla_{n-2}} & ^{\nu} \mathcal{I}_{n-2} A \arrow[swap]{r}{\nabla_{n-3}} & \dotsb \arrow[swap]{r}{\nabla_{1}} & ^{\nu} \mathcal{I}_{1} A \arrow[swap]{r}{\nabla} & ^{\nu} A
\end{tikzcd}
}}
\]

The $n$-form $\omega:= \Theta_n^{-1}(1)\in \Omega^n A$ is called an {\em integrating volume form}. 
\end{definition}

The algebra of complex matrices $M_n(\mathbb{C})$ with the $n$-dimensional calculus generated by derivations presented by Dubois-Violette et al. \cite{DuboisViolette1988, DuboisVioletteKernerMadore1990}, the quantum group $SU_q(2)$ with the three-dimensional left covariant calculus developed by Woronowicz \cite{Woronowicz1987} and the quantum standard sphere with the restriction of the above calculus, are examples of algebras admitting integrable calculi. For more details on the subject, see Brzezi\'nski et al. \cite{BrzezinskiElKaoutitLomp2010}. 

The following proposition shows that the integrability of a differential calculus can be defined without explicit reference to integral forms. This allows us to guarantee the integrability by considering the existence of finitely generator elements that allow to determine left and right components of any homogeneous element of $\Omega(A)$.

\begin{proposition}[{\cite[Theorem 2.2]{BrzezinskiSitarz2017}}]\label{integrableequiva} 
Let $(\Omega A, d)$ be an $n$-dimensional differential calculus over an algebra $A$. The following assertions are equivalent:
\begin{enumerate}
    \item [\rm (1)] $(\Omega A, d)$ is an integrable differential calculus.
    
    \item [\rm (2)] There exists an algebra automorphism $\nu$ of $A$ and $A$-bimodule isomorphisms $\Theta_k : \Omega^k A \rightarrow \ ^{\nu}\mathcal{I}_{n-k}A$, $k =0, \ldots, n$, such that, for all $\omega'\in \Omega^k A$ and $\omega''\in \Omega^mA$,
    \begin{align*}
        \Theta_{k+m}(\omega'\wedge\omega'')=(-1)^{(n-1)m}\Theta_k(\omega')\cdot\omega''.
    \end{align*}
    
    \item [\rm (3)] There exists an algebra automorphism $\nu$ of $A$ and an $A$-bimodule map $\vartheta:\Omega^nA\rightarrow\ ^{\nu}A$ such that all left multiplication maps
    \begin{align*}
    \ell_{\vartheta}^{k}:\Omega^k A &\ \rightarrow \mathcal{I}_{n-k}A, \\
    \omega' &\ \mapsto \vartheta\cdot\omega', \quad k = 0, 1, \dotsc, n,
    \end{align*}
    where the actions $\cdot$ are defined by {\rm (}\ref{BrzezinskiSitarz2017(2.3)}{\rm )}, are bijective.
    
    \item [\rm (4)] $(\Omega A, d)$ has a volume form $\omega$ such that all left multiplication maps
    \begin{align*}
        \ell_{\pi_{\omega}}^{k}:\Omega^k A &\ \rightarrow \mathcal{I}_{n-k}A, \\
        \omega' &\ \mapsto \pi_{\omega} \cdot \omega', \quad k=0,1, \dotsc, n-1,
    \end{align*}
    
    where $\pi_{\omega}$ is defined by {\rm (}\ref{BrzezinskiSitarz2017(2.1)}{\rm )}, are bijective.
\end{enumerate}
\end{proposition}

A volume form $\omega\in \Omega^nA$ is an {\em integrating form} if and only if it satisfies Proposition \ref{integrableequiva} (4) \cite[Remark 2.3]{BrzezinskiSitarz2017}.

The most interesting cases of differential calculi are those where $\Omega^k A$ are finitely generated and projective right or left (or both) $A$-modules \cite{Brzezinski2011}.

\begin{proposition}\label{BrzezinskiSitarz2017Lemmas2.6and2.7}
\begin{enumerate}
\item [\rm (1)] \cite[Lemma 2.6]{BrzezinskiSitarz2017} Consider $(\Omega A, d)$ an integrable and $n$-dimensional calculus over $A$ with integrating form $\omega$. Then $\Omega^{k} A$ is a finitely generated projective right $A$-module if there exist a finite number of forms $\omega_i \in \Omega^{k} A$ and $\overline{\omega}_i \in \Omega^{n-k} A$ such that, for all $\omega' \in \Omega^{k} A$, we have that 
\begin{equation*}
\omega' = \sum_{i} \omega_i \pi_{\omega} (\overline{\omega}_i \wedge \omega').
\end{equation*}

\item [\rm (2)] \cite[Lemma 2.7]{BrzezinskiSitarz2017} Let $(\Omega A, d)$ be an $n$-dimensional calculus over $A$ admitting a volume form $\omega$. Assume that for all $k = 1, \ldots, n-1$, there exists a finite number of forms $\omega_{i}^{k},\overline{\omega}_{i}^{k} \in \Omega^{k}(A)$ such that for all $\omega'\in \Omega^kA$, we have that
\begin{equation*}
\omega'=\displaystyle\sum_i\omega_{i}^{k}\pi_\omega(\overline{\omega}_{i}^{n-k}\wedge\omega')=\displaystyle\sum_i\nu_{\omega}^{-1}(\pi_\omega(\omega'\wedge\omega_{i}^{n-k}))\overline{\omega}_{i}^{k},
\end{equation*}

where $\pi_{\omega}$ and $\nu_{\omega}$ are defined by {\rm (}\ref{BrzezinskiSitarz2017(2.1)}{\rm )} and {\rm (}\ref{BrzezinskiSitarz2017(2.2)}{\rm )}, respectively. Then $\omega$ is an integral form and all the $\Omega^{k}A$ are finitely generated and projective as left and right $A$-modules.
\end{enumerate}
\end{proposition}

Brzezi\'nski and Sitarz \cite[p. 421]{BrzezinskiSitarz2017} asserted that to connect the integrability of the differential graded algebra $(\Omega A, d)$ with the algebra $A$, it is necessary to relate the dimension of the differential calculus $\Omega A$ with that of $A$, and since we are dealing with algebras that are deformations of coordinate algebras of affine varieties, the {\em Gelfand-Kirillov dimension} introduced by Gelfand and Kirillov \cite{GelfandKirillov1966, GelfandKirillov1966b} seems to be the best suited. Briefly, given an affine $\Bbbk$-algebra $A$, the {\em Gelfand-Kirillov dimension of} $A$, denoted by ${\rm GKdim}(A)$, is given by
\[
{\rm GKdim}(A) := \underset{n\to \infty}{\rm lim\ sup} \frac{{\rm log}({\rm dim}\ V^{n})}{{\rm log}\ n},
\]

where $V$ is a finite-dimensional subspace of $A$ that generates $A$ as an algebra. This definition is independent of choice of $V$. If $A$ is not affine, then its Gelfand-Kirillov dimension is defined to be the supremum of the Gelfand-Kirillov dimensions of all affine subalgebras of $A$. An affine domain of Gelfand-Kirillov dimension zero is precisely a division ring that is finite-dimensional over its center. In the case of an affine domain of Gelfand-Kirillov dimension one over $\Bbbk$, this is precisely a finite module over its center, and thus polynomial identity. In some sense, this dimensions measures the deviation of the algebra $A$ from finite dimensionality. For more details about this dimension, see the excellent treatment developed by Krause and Lenagan \cite{KrauseLenagan2000}.

After preliminaries above, we arrive to the key notion of this paper.

\begin{definition}[{\cite[Definition 2.4]{BrzezinskiSitarz2017}}]\label{BrzezinskiSitarz2017Definition2.4}
An affine algebra $A$ with integer Gelfand-Kirillov dimension $n$ is said to be {\em differentially smooth} if it admits an $n$-dimensional connected integrable differential calculus $(\Omega A, d)$.
\end{definition}

From Definition \ref{BrzezinskiSitarz2017Definition2.4} a differentially smooth algebra comes equipped with a well-behaved differential structure and with the precise concept of integration \cite[p. 2414]{BrzezinskiLomp2018}.

\begin{example}\label{Brzezinski2015DSOEbiquadratic}
\begin{enumerate}
    \item [\rm (i)] The polynomial algebra $\Bbbk[x_1, \dotsc, x_n]$ has Gelfand-Kirillov dimension $n$ and the usual exterior algebra is an $n$-dimensional integrable calculus, whence $\Bbbk[x_1, \dotsc, x_n]$ is differentially smooth.

    \item [\rm (ii)] If $\sigma$ is an endomorphism of $R$, then a map $\delta : R \rightarrow R$ is called a {\em $\sigma$-derivation} on $R$, if it is additive and satisfies that $\displaystyle \delta(r s) = \sigma(r )\delta(s)+\delta(r )s$, for every $r,s \in R$ (strictly speaking, this is the definition of {\em left} $\sigma$-{\em derivation}). The pair $(\sigma, \delta)$ is called a {\em quasi-derivation on} $R$ \cite[Definition 3.1]{BuesoTorrecillasVerschoren2003}. According to Ore \cite{Ore1931, Ore1933}, the {\em skew polynomial ring} of $R$ is defined as the ring $R[x;\sigma,\delta]$ generated by $R$ and an indeterminate $x$ subject to the relation $xr := \sigma(r)x + \delta(r)$, for every $r \in R$, such that $R[x;\sigma,\delta]$ is a free left $R$-module with basis $\left \{x^k\ | \ k \in \mathbb{N} \right \}$. In the literature, $R[x;\sigma, \delta]$ is called a {\em skew polynomial ring over} $R$ {\em of mixed type}. If $\sigma$ is an injective map of $R$, then we call it an {\em Ore extension of injective type}, while if $\sigma$ is the identity of $R$, then we write $R[x;\delta]$ and call it a {\em ring of derivation type}. On the other hand, if $\delta$ is the zero map, then we write $R[x;\sigma]$ which is known as a {\em ring of endomorphism type}. 

Brzezi{\'n}ski \cite{Brzezinski2015} characterized the differential smoothness of skew polynomial rings of the form $\Bbbk[t][x; \sigma_{q, r}, \delta_{p(t)}]$ where $\sigma_{q, r}(t) = qt + r$, with $q, r \in \Bbbk,\ q\neq 0$, and the $\sigma_{q, r}-$derivation $\delta_{p(t)}$ is defined as
\begin{equation}\label{deltap}
\delta_{p(t)} (f(t)) = \frac{f(\sigma_{q, r}(t)) - f(t)}{\sigma_{q, r}(t) - t} p(t),
\end{equation}

for an element $p(t) \in \Bbbk[t]$. $\delta_{p(t)}(f(t))$ is a suitable limit when $q = 1$ and $r = 0$, that is, when $\sigma_{q, r}$ is the identity map of $\Bbbk[t]$.

For the maps
\begin{equation}\label{Brzezinski2015(3.4)}
\nu_t(t) = t,\quad \nu_t(x) = qx + p'(t)\quad {\rm and}\quad \nu_x(t) = \sigma_{q, r}^{-1}(t),\quad \nu_x(x) = x,
\end{equation}

where $p'(t)$ is the classical $t$-derivative of $p(t)$, Brzezi{\'n}ski \cite[Lemma 3.1]{Brzezinski2015} showed that all of them simultaneously extend to algebra automorphisms $\nu_t$ and $\nu_x$ of $\Bbbk[t][x; \sigma_{q, r}, \delta_{p(t)}]$ only in the following three cases:
    \begin{enumerate}
        \item [\rm (a)] $q = 1, r = 0$ with no restriction on $p(t)$;
        
        \item [\rm (b)] $q = 1, r\neq 0$ and $p(t) = c$, $c\in \Bbbk$;
        
        \item [\rm (c)] $q\neq 1, p(t) = c\left( t + \frac{r}{q-1} \right)$, $c\in \Bbbk$ with no restriction on $r$.
    \end{enumerate}
    
In any of the cases {\rm (a) - (c)} we have that $\nu_x \circ \nu_t = \nu_t \circ \nu_x$. If the Ore extension $\Bbbk[t][x; \sigma_{q, r}, \delta_{p(t)}]$ satisfies one of these three conditions, Brzezi{\'n}ski proved that it is differentially smooth \cite[Proposition 3.3]{Brzezinski2015}.

From Brzezi{\'n}ski's result we get that the algebras
\begin{itemize}
 \item The {\em polynomial algebra} $\Bbbk[x_1, x_2]$;
        
        \item The {\em Weyl algebra} $A_1(\Bbbk) = \Bbbk\{x_1, x_2\} / \langle x_1x_2 - x_2x_1 - 1\rangle$;
        
        \item The {\em universal enveloping algebra of the Lie algebra} $\mathfrak{n}_2 = \langle x_1, x_2\mid [x_2, x_1] = x_1\rangle$, that is, $U(\mathfrak{n}_2) = \Bbbk\{x_1, x_2\} / \langle x_2x_1 - x_1x_2 - x_1\rangle$, and
        
        \item The {\em quantum plane} ({\em Manin's plane}) $\mathcal{O}_q(\Bbbk) = \Bbbk \{x_1, x_2\} / \langle x_2 x_1 - qx_1 x_2\rangle$, where $q\in \Bbbk\ \backslash\ \{0,1\}$, 
\end{itemize}

are differentially smooth.

\item [\rm (iii)] For the 3-dimensional skew polynomial algebras and diffusion algebras (Example \ref{ExamplesSPBWextensionsDS} (iv) and (v)), its differential smoothness was studied by the second author in \cite{ReyesSarmiento2022}.
\end{enumerate}
\end{example}

\begin{remark}
There are examples of algebras that are not differentially smooth. Consider the commutative algebra $A = \mathbb{C}[x, y] / \langle xy \rangle$. A proof by contradiction shows that for this algebra there are no one-dimensional connected integrable calculi over $A$, so it cannot be differentially smooth \cite[Example 2.5]{BrzezinskiSitarz2017}.
\end{remark}

\section{Differential smoothness of SPBW extensions over \texorpdfstring{$\Bbbk[t]$}{Lg}}\label{DICSPBWKt}

In this section, we investigate the differential smoothness of bijective SPBW extensions over the commutative polynomial ring $\Bbbk[t]$.

\subsection{SPBW extensions in two indeterminates}\label{SPBWTMTwoI}

Consider a SPBW extension of the form $\sigma(\Bbbk[t]) \langle x_1, x_2\rangle$. From Definition \ref{defpbwextension}, we get the relations 
\begin{align*}
    x_1r(t) = &\ \sigma_1(r(t)) x_1 + \delta_1(r(t)),\quad x_2 r(t) =  \sigma_2(r(t)) x_2 + \delta_2(r(t)), \quad {\rm and} \\
    x_2 x_1 = &\ c_{1, 2}(t)x_1 x_2 + q_{1,2}^{(0)}(t) + q_{1,2}^{(1)}(t)x_1 + q_{1,2}^{(2)}(t)x_2, 
\end{align*}

where $r(t), c_{1, 2}(t), q_{1, 2}^{(0)}(t), q_{1, 2}^{(1)}(t), q_{1, 2}^{(2)}(t)$ belong to $\Bbbk[t]$ with $c_{1,2}(t)$ non-zero.

Let $\sigma_1(t) = a_1t+b_1$ and $\sigma_2(t)=a_2t+b_2$ be automorphisms of $\Bbbk[t]$ (this is precisely the form of the elements of ${\rm Aut}(\Bbbk[x])$ \cite{ShestakovUmirbaev2003, VandenEssen2000}) with the corresponding $\sigma_i$-derivations ($i =  1, 2$) expressed as in (\ref{deltap}), that is,
\[
 \delta_1(f(t)) = \frac{f\left(\sigma_{1}(t)\right)-f(t)}{\sigma_{1}(t)-t}p_1(t), \quad {\rm and} \quad \delta_2(f(t)) = \frac{f\left(\sigma_{2}(t)\right)-f(t)}{\sigma_{2}(t)-t}p_2(t),
\]

where $p_1(t), p_2(t)$ are fixed elements of $\Bbbk[t]$. Thus, the relations between the indeterminates $t, x_1$ and $x_2$ can be expressed as
\begin{align}
    x_1t = &\ a_1tx_1 + b_1x_1+ p_1(t), \quad x_2t = a_2tx_2 + b_2x_2+ p_2(t), \quad {\rm and} \label{relPBW3.1} \\
    x_2 x_1 = &\ c_{1, 2}(t)x_1 x_2 + q_{1,2}^{(0)}(t) + q_{1,2}^{(1)}(t)x_1 + q_{1,2}^{(2)}(t)x_2. \label{relPBW3.3}
\end{align}

\begin{proposition}\label{commrel}
    From Equation {\rm (}\ref{relPBW3.1}{\rm )}, we obtain the commutation relations
    \begin{align*}
        x_1t^n &\ =(a_1t+b_1)^nx_1+p_1(t)\displaystyle\sum_{l=0}^{n-1}(a_1t+b_1)^lt^{n-1-l},\quad {\rm and} \\
        x_2t^n &\ =(a_2t+b_2)^nx_2+p_2(t)\displaystyle\sum_{l=0}^{n-1}(a_2t+b_2)^lt^{n-1-l}. 
    \end{align*}
\end{proposition}
\begin{proof} We proceed by induction on $n$. For $n=1$, the assertion is clear. Suppose that the relation holds for $n=k$. Since
\begin{align*}
    x_1t^{k+1}&\ =(x_1t^k)t =\left((a_1t+b_1)^kx_1+p_1(t)\displaystyle\sum_{l=0}^{k-1}(a_1t+b_1)^lt^{k-1-l}\right)t \\
    &\ = (a_1t+b_1)^kx_1t+p_1(t)\displaystyle\sum_{l=0}^{k-1}(a_1t+b_1)^lt^{(k+1)-1-l} \\
    &\ = (a_1t+b_1)^k((a_1t+b_1)x_1+p_1(t))+p_1(t)\displaystyle\sum_{l=0}^{k-1}(a_1t+b_1)^lt^{(k+1)-1-l} \\
    &\ = (a_1t+b_1)^{k+1}x_1+\left((a_1t+b_1)^kp_1(t)+p_1(t)\displaystyle\sum_{l=0}^{k-1}(a_1t+b_1)^lt^{(k+1)-1-l}\right) \\
    &\ = (a_1t+b_1)^{k+1}x_1+\sum_{l=0}^{k}(a_1t+b_1)^lt^{(k+1)-1-l},
\end{align*}

the assertion follows. The proof of the second relation is similar.
\end{proof}

\begin{proposition}\label{autoSkewPBW} 
Let
\begin{align}
   \nu_t(t) = &\ t, & \nu_t(x_1) = &\ a_1x_1 + p_1'(t), & \nu_t(x_2) = &\ a_2x_2+p_2'(t), \label{Autoskew1} \\ 
    \nu_{x_1}(t) = &\ \sigma_{1}^{-1}(t), & \nu_{x_1}(x_1) = &\ x_1, &  \nu_{x_1}(x_2) = &\ c_{1,2}x_2+q_{1,2}^{(1)}, \label{Autoskew2} \\
    \nu_{x_2}(t) = &\ \sigma_{2}^{-1}(t), & \nu_{x_2}(x_1) = &\ c_{1,2}^{-1}x_1-c_{1,2}^{-1}q_{1,2}^{(2)}, & \nu_{x_2}(x_2) = &\ x_2, \label{Autoskew3}
\end{align}

where $p_1'(t)$ and $p'_2(t)$ are the $t$-derivatives of $p_1(t)$ and $p_2(t)$, respectively, and $c_{1,2}, q_{1,2}^{(0)}, q_{1,2}^{(1)}, q_{1,2}^{(2)} \in \Bbbk$, with $c_{1,2}$ non-zero. Then:
\begin{enumerate}
\item [\rm (1)] Leibniz's rule holds in the cases listed in Table \ref{Leibnizsruletwoindeterminates(1)}. The maps defined by {\rm (}\ref{Autoskew1}{\rm )}, {\rm (}\ref{Autoskew2}{\rm )} and {\rm (}\ref{Autoskew3}{\rm )} simultaneously extend to algebra automorphisms $\nu_t, \nu _{x_1}, \nu_{x_2}$ of $\sigma(\Bbbk[t])\langle x_1, x_2\rangle$ only in cases {\rm (a)}, {\rm (b)}, {\rm (c)}, {\rm (d)}, {\rm (e)}, {\rm (g)} and {\rm (i)}.
\end{enumerate}

\begin{table}[h]
\caption{Leibniz's rule}
\label{Leibnizsruletwoindeterminates(1)}
\begin{center}
\resizebox{12.5cm}{!}{
\setlength\extrarowheight{6pt}
\begin{tabular}{ |c|c|c|c| } 
\hline
{\rm Case} & {\rm Possibilities for} $a_1$, $b_1$, $a_2$, $b_2$ & {\em Polynomials} $p_1(t)$ {\em and} $p_2(t)$ & {\em Restrictions} \\
\hline
\multirow{2}{*}{{\rm (a)}} & \multirow{2}{*}{$a_1=1$, $b_1=0$, $a_2=1$, $b_2=0$} & $p_1(t)=p_2(t)=0$ & $q_{1,2}^{(0)}=q_{1,2}^{(1)}=q_{1,2}^{(2)}=0$,  $c_{1,2}\in\Bbbk^{\ast}$ \\ \cline{3-4}
&  & $p_1(t), p_2(t)\in\Bbbk[t]$ & $q_{1,2}^{(1)}=q_{1,2}^{(2)}=0$, $c_{1,2}=1$, $q_{1,2}^{(0)}\in\Bbbk$ \\ 
\hline
\multirow{2}{*}{{\rm (b)}} & \multirow{2}{*}{$a_1=1$, $b_1=0$, $a_2=1$, $b_2\not =0$} & $p_1(t)=p_1$, $p_2(t)=p_2$, $p_1,p_2\in\Bbbk$ & $c_{1,2}=1$, $q_{1,2}^{(1)}=q_{1,2}^{(2)}=0$, $q_{1,2}^{(0)}\in\Bbbk$ \\ \cline{3-4}
&  & $p_1(t)=p_2(t)=0$ & $q_{1,2}^{(0)}=q_{1,2}^{(1)}=q_{1,2}^{(2)}=0$,  $c_{1,2}\in\Bbbk^{\ast}$ \\ 
\hline
\multirow{2}{*}{{\rm (c)}} & \multirow{2}{*}{$a_1=1$, $b_1\not=0$, $a_2=1$, $b_2=0$} & $p_1(t)=p_1$, $p_2(t)=p_2$, $p_1,p_2\in\Bbbk$ & $c_{1,2}=1$, $q_{1,2}^{(1)}=q_{1,2}^{(2)}=0$, $q_{1,2}^{(0)}\in\Bbbk$ \\ \cline{3-4}
&  & $p_1(t)=p_2(t)=0$ & $q_{1,2}^{(0)}=q_{1,2}^{(1)}=q_{1,2}^{(2)}=0$,  $c_{1,2}\in\Bbbk^{\ast}$ \\  
\hline
\multirow{2}{*}{{\rm (d)}} & \multirow{2}{*}{$a_1=1$, $b_1\not=0$, $a_2=1$, $b_2\not =0$} & $p_1(t)=p_1$, $p_2(t)=p_2$, $p_1,p_2\in\Bbbk$ & $c_{1,2}=1$, $q_{1,2}^{(1)}=q_{1,2}^{(2)}=0$, $q_{1,2}^{(0)}\in\Bbbk$ \\ \cline{3-4}
&  & $p_1(t)=p_2(t)=0$ & $q_{1,2}^{(0)}=q_{1,2}^{(1)}=q_{1,2}^{(2)}=0$,  $c_{1,2}\in\Bbbk^{\ast}$ \\   
\hline
\multirow{3}{*}{{\rm (e)}} & \multirow{3}{*}{$a_1=1$, $b_1=0$, $a_2\not=1$, $b_2\in\Bbbk$} & $p_1(t)=0$, $p_2(t)=p_2\left(t+\frac{b_2}{a_2-1}\right)$, $p_2\in\Bbbk$ & $c_{1,2}=1$, $q_{1,2}^{(0)}=q_{1,2}^{(1)}=q_{1,2}^{(2)}=0$ \\ \cline{3-4}
&  & $p_1(t)=p_1$ $p_2(t)=0$,$p_1\in\Bbbk$  & $q_{1,2}^{(0)}=q_{1,2}^{(1)}=q_{1,2}^{(2)}=0$,  $c_{1,2}=a_{2}^{-1}$ \\ \cline{3-4}
&  & $p_1(t)=p_2(t)=0$ & $q_{1,2}^{(0)}=q_{1,2}^{(1)}=q_{1,2}^{(2)}=0$,  $c_{1,2}\not\in\{1,a_2^{-1}\}$ \\ 
\hline
{{\rm (f)}} & $a_1=1$, $b_1\not=0$, $a_2\not=1$ & $p_1(t)=p_1$,  $p_2(t)=p_2\left(t+\frac{b_2}{a_2-1}\right), p_1,p_2\in\Bbbk$ & There is not solution for all relations \\ 
\hline
\multirow{3}{*}{{\rm (g)}} & \multirow{3}{*}{$a_1\not=1$,  $a_2=1$, $b_2=0$} & $p_1(t)=p_1\left(t+\frac{b_1}{a_1-1}\right)$, $p_2(t)=0$, $p_1\in\Bbbk$ & $c_{1,2}=1$, $q_{1,2}^{(0)}=q_{1,2}^{(1)}=q_{1,2}^{(2)}=0$ \\ \cline{3-4}
&  & $p_1(t)=0$ $p_2(t)=p_2$,$p_2\in\Bbbk$  & $q_{1,2}^{(0)}=q_{1,2}^{(1)}=q_{1,2}^{(2)}=0$,  $c_{1,2}=a_{1}^{-1}$ \\ \cline{3-4}
&  & $p_1(t)=p_2(t)=0$ & $q_{1,2}^{(0)}=q_{1,2}^{(1)}=q_{1,2}^{(2)}=0$,  $c_{1,2}\not\in\{1,a_1^{-1}\}$ \\ 
\hline
{{\rm (h)}} & $a_1\not=1$, $b_1\in\Bbbk$,  $a_2=1$, $b_2\not=0$ & $p_1(t)=p_1\left(t+\frac{b_1}{a_1-1}\right)$,  $p_2(t)=p_2, p_1,p_2\in\Bbbk$ & There is not solution for all relations \\ 
\hline
{{\rm (i)}} & $a_1\not=1$,  $a_2\not=1$, $b_1=0$ $b_2=0$ & $p_1(t)=p_1t$,  $p_2(t)=p_2t$, $p_1,p_2\in\Bbbk$ & $c_{1,2}=1$, $q_{1,2}^{(0)}=q_{1,2}^{(1)}=q_{1,2}^{(2)}=0$ \\ 
\hline
\end{tabular}
}
\end{center}
\end{table}

\begin{enumerate}
\item [\rm (2)] In cases {\rm (a)}, {\rm (b)}, {\rm (c)}, {\rm (d)}, {\rm (e)}, {\rm (g)} and {\rm (i)}, we get that 
    \begin{equation}\label{Eq1skew}
   \nu_t \circ \nu_{x_1} = \nu_{x_1} \circ \nu_t, \quad \nu_t \circ \nu_{x_2} = \nu_{x_2} \circ \nu_t, \quad \nu_{x_2} \circ \nu_{x_1} = \nu_{x_1} \circ \nu_{x_2}.
    \end{equation}
\end{enumerate}
\end{proposition}
\begin{proof}
For the first assertion, the map $\nu_t$ can be extended to an algebra homomorphism if and only if the definitions of $\nu_t(t)$, $\nu_t(x_1)$ and $\nu_t(x_2)$ respect relations {\rm (}\ref{relPBW3.1}{\rm )}, and {\rm (}\ref{relPBW3.3}{\rm )}, i.e.
\begin{align*}
   \nu_t(x_1)\nu_t(t)-\nu_t(a_1t+b_1)\nu_t(x_1) = &\ \nu_t(p_1(t)), \\
   \nu_t(x_2)\nu_t(t)-\nu_t(a_2t+b_2)\nu_t(x_2) = &\ \nu_t(p_2(t)), \quad {\rm and} \\
   \nu_t(x_2)\nu_t(x_1)-c_{1,2}\nu_t(x_1)\nu_t(x_2) = &\ q_{1,2}^{(0)}+q_{1,2}^{(1)}\nu_t(x_1)+q_{1,2}^{(2)}\nu_t(x_2).
\end{align*}

In this way, we obtain the equations 
\begin{align}
((a_1-1)t+b_1)p_1'(t) = &\ (a_1-1)p_1(t), \notag \\
((a_2-1)t+b_2)p_2'(t) = &\ (a_2-1)p_2(t), \label{firstsecondrel}
\end{align}

and
\begin{align}
   &\ (a_1a_2-1)q_{1,2}^{(0)} + (a_2-1)a_1q_{1,2}^{(1)}x_1 + a_1(p_2'(t)x_1-c_{1,2}x_1p_2'(t)) + (a_1-1)a_2q_{1,2}^{(2)}x_2 \notag \\ 
   &\ + a_2(x_2p_1'(t)-c_{1,2}p_1'(t)x_2) + (1-c_{1,2})p_1'(t)p_2'(t)-q_{1,2}^{(1)}p_1'(t)-q_{1,2}^{(2)}p_2'(t)=0. \label{thirdrel}
\end{align}

Note that the map $\nu_{x_1}$ can be extended to an algebra automorphism if and only if the definitions of $\nu_{x_1}(t)$, $\nu_{x_1}(x_1)$ and $\nu_{x_1}(x_2)$ respect relations {\rm (}\ref{relPBW3.1}{\rm )}, and {\rm (}\ref{relPBW3.3}{\rm )}, that is, 
\begin{align*}
   \nu_{x_1}(x_1)\nu_{x_1}(t)-\nu_{x_1}(a_1t+b_1)\nu_{x_1}(x_1) = &\ \nu_{x_1}(p_1(t)), \\
   \nu_{x_1}(x_2)\nu_{x_1}(t)-\nu_{x_1}(a_2t+b_2)\nu_{x_1}(x_2) = &\ \nu_{x_1}(p_2(t)),\quad {\rm and} \\
   \nu_{x_1}(x_2)\nu_{x_1}(x_1)-c_{1,2}\nu_{x_1}(x_1)\nu_{x_1}(x_2) = &\ q_{1,2}^{(0)}+q_{1,2}^{(1)}\nu_{x_1}(x_1)+q_{1,2}^{(2)}\nu_{x_1}(x_2).
\end{align*}

Therefore, 
\begin{align}
   a_{1}^{-1}p_1(t) = &\ p_1(a_{1}^{-1}(t-b_1)), \label{Firstnux1} \\
   c_{1,2}(a_1^{-1}(a_2b_1+b_2-b_1)-b_2)x_2\ + &\ a_1^{-1}(c_{1,2}p_2(t)-(1+a_2)q_{1,2}^{(1)}t) \notag \\
   +\ q_{1,2}^{(1)}(a_1^{-1}b_1(a_2-1)-b_2) = &\ p_2(a_1^{-1}(t-b_1)), \quad {\rm and} \label{Secondnux1}\\
   (c_{1,2}-1)q_{1,2}^{(0)}-q_{1,2}^{(1)}q_{1,2}^{(2)} = &\ 0. \label{Thirnux1}
\end{align}

The map $\nu_{x_2}$ can be extended to an algebra automorphism if and only if the definitions of $\nu_{x_2}(t)$, $\nu_{x_2}(x_1)$ and $\nu_{x_2}(x_2)$ respect relations {\rm (}\ref{relPBW3.1}{\rm )}, and {\rm (}\ref{relPBW3.3}{\rm )}, i.e.
\begin{align*}
   \nu_{x_2}(x_1)\nu_{x_2}(t)-\nu_{x_2}(a_1t+b_1)\nu_{x_2}(x_1) = &\ \nu_{x_2}(p_1(t)), \\
   \nu_{x_2}(x_2)\nu_{x_2}(t)-\nu_{x_2}(a_2t+b_2)\nu_{x_2}(x_2) = &\ \nu_{x_2}(p_2(t)),\quad {\rm and} \\
   \nu_{x_2}(x_2)\nu_{x_2}(x_1)-c_{1,2}\nu_{x_2}(x_1)\nu_{x_2}(x_2) = &\ q_{1,2}^{(0)}+q_{1,2}^{(1)}\nu_{x_2}(x_1)+q_{1,2}^{(2)}\nu_{x_2}(x_2).
\end{align*}

In other words,
\begin{align}
   c_{1,2}^{-1}(a_2^{-1}(b_1+a_1b_2-b_2)-b_1)x_1\ + &\ c_{1,2}^{-1}a_2^{-1}(p_1(t)-(1+a_1)q_{1,2}^{(2)}t) \notag\\
   +\ q_{1,2}^{(2)}c_{1,2}^{-1}(a_2^{-1}b_2(1+a_1)+b_1) = &\ p_1(a_2^{-1}(t-b_2)), \label{Firstnux2}\\
   a_{2}^{-1}p_2(t) = &\ p_2(a_{2}^{-1}(t-b_2)), \quad{\rm and} \label{Secondnux2} \\
(c_{1,2}^{-1}-1)q_{1,2}^{(0)}+c_{1,2}^{-1}q_{1,2}^{(1)}q_{1,2}^{(2)} = &\ 0. \label{Thirnux2}
\end{align}

Notice that expressions (\ref{firstsecondrel}) are the same as in \cite[Lemma 3.1]{Brzezinski2015}, and that these equations are independent of each other, so we have nine possible combinations for the values of $a_1, b_1, a_2$ and $b_2$. For each of these combinations, equations (\ref{Thirnux1}) and (\ref{Thirnux2}) will be used to determine the possible values for $c_{1,2}$, $p_1(t)$, $p_2(t)$ $q_{1,2}^{(i)}$, $i=0, 1, 2$. Let us see.

Consider $p_1(t) := \sum\limits_{j=0}^n m_jt^j$ and $p_2(t) := \sum\limits_{j=0}^n k_jt^j$.

\begin{enumerate}
    \item [\rm (a)] Equation (\ref{thirdrel}) leads to the equalities 
    \begin{align*}
        &\ \sum_{j=1}^n [jk_jt^{j-1}x_1-c_{1,2}x_1jk_jt^{j-1}+x_2jm_jt^{j-1}-c_{1,2}jm_jt^{j-1}x_2 \\
        &\ \ \ +\sum_{s=1}^{n}(1-c_{1,2})jk_jsm_st^{s+j-2} -q_{1,2}^{(1)}jm_jt^{j-1}-q_{1,2}^{(2)}jk_jt^{j-1} ] = 0, 
    \end{align*}
    \begin{align*}
        &\ \sum_{j=1}^n [jk_jt^{j-1}x_1-c_{1,2}jk_j\left(t^{j-1}x_1+(j-1)p_1(t)t^{n-1}\right)  \\
        &\ \ \ + jm_j\left(t^{j-1}x_2+(j-1)p_2(t)t^{n-1}\right)-c_{1,2}jm_jt^{j-1}x_2 \\
        &\ \ \ +\sum_{s=1}^{n}(1-c_{1,2})jk_jsm_st^{s+j-2} -q_{1,2}^{(1)}jm_jt^{j-1}-q_{1,2}^{(2)}jk_jt^{j-1} ] = 0, 
    \end{align*}

    and 
     \begin{align*}
        &\ \sum_{j=1}^nt^{j-1} [jk_j(1-c_{1,2})x_1+jm_j(1-c_{1,2})x_2 \\
        &\ \ + j(j-1)\sum_{i=0}^n\left(m_jk_i-c_{1,2}k_jm_i\right)t^{i+n-j} \\
        &\ \ + \sum_{s=1}^{n}(1-c_{1,2})jk_jsm_st^{s-1}-j(q_{1,2}^{(1)}m_j+q_{1,2}^{(2)}k_j) ]=0.
    \end{align*}
    
    If we focus on the coefficients of $x_1$ and $x_2$, these must be zero, that is, $k_j(1-c_{1,2}) = 0$ and $m_j(1-c_{1,2}) = 0$. This implies that $m_j=k_j=0$, for $1 \leq j \leq n$ and so the polynomials $p_1(t)$ and $p_2(t)$ are constants or $c_{1,2}=1$. From relations {\rm (}\ref{Secondnux1}{\rm )}, {\rm (}\ref{Thirnux1}{\rm )}, {\rm (}\ref{Firstnux2}{\rm )} and  {\rm (}\ref{Thirnux2}{\rm )}, we get that
    \begin{align*}
        (c_{1,2}-1)p_2(t)-2q_{1,2}^{(1)}t = &\ 0,\\(c_{1,2}^{-1}-1)p_1(t)-2c_{1,2}^{-1}q_{1,2}^{(2)}t = &\ 0, \\
        (c_{1,2}-1)q_{1,2}^{(0)} = &\ q_{1,2}^{(1)}q_{1,2}^{(2)}, \quad{\rm and} \\ (c_{1,2}^{-1}-1)q_{1,2}^{(0)} = &\ -c_{1,2}^{-1}q_{1,2}^{(1)}q_{1,2}^{(2)}.
    \end{align*}

    If $p_1(t) = p_1, p_2(t)=p_2 \in \Bbbk$, then $q_{1,2}^{(1)}=q_{1,2}^{(2)}=0$  and we obtain the following options:
    \begin{itemize}
        \item $c_{1,2}=1$, $q_{1,2}^{(0)}$ has no restrictions.
        
        \item $p_2=0$, $q_{1,2}^{(0)}=0$ and $p_1=0$, with no restriction over $c_{1,2}$.
    \end{itemize}
    
    Finally, if $c_{1,2}=1$ it is necessary that $q_{1,2}^{(1)}m_j+q_{1,2}^{(2)}k_j=0$ for all $1 \leq i \leq n$. One possibility is precisely when $m_j=k_j=0$, which means that $p_1(t)$ and $p_2(t)$ are constants (as in the previous case). The other option is that $q_{1,2}^{(1)}=q_{1,2}^{(2)}=0$, with no restrictions on the polynomials $p_1(t)$ and $p_2(t)$. We have considered all possible options.
    
    \item [\rm (b)] Equation (\ref{thirdrel}) leads to the following way of relating the coefficients
    {\small{
    \begin{align*}
        &\ \sum_{j=1}^n\left[jm_j\left(\left(t+b_2\right)^{j-1}x_2-c_{1,2}t^{j-1}x_2 +\sum_{l=0}^{j-2}p_2(t+b_2)^{l}t^{j-2-l}\right)-q_{1,2}^{(1)}jm_jt^{j-1}\right] = 0.
    \end{align*}
    }}
    
The coefficient of $x_2$ must be zero, that is, $jm_j((t+b_2)^{j-1}-c_{1,2}t^{j-1}) = 0$. This implies that $m_j=0$ for $1 \leq i \leq n$ whence the polynomial $p_1(t)$ is constant. From relations {\rm (}\ref{Secondnux1}{\rm )}, {\rm (}\ref{Thirnux1}{\rm )}, {\rm (}\ref{Firstnux2}{\rm )} and {\rm (}\ref{Thirnux2}{\rm )}, it follows that
\begin{align*}
     (c_{1,2}-1)p_2-b_2q_{1,2}^{(1)}-2q_{1,2}^{(1)}t = &\ 0, \\
(c_{1,2}^{-1}-1)p_1+2b_2c_{1,2}^{-1}q_{1,2}^{(2)}-2c_{1,2}^{-1}q_{1,2}^{(2)}t = &\ 0, \\
(c_{1,2}-1)q_{1,2}^{(0)} = &\ q_{1,2}^{(1)}q_{1,2}^{(2)}, \quad {\rm and} \\
(c_{1,2}^{-1}-1)q_{1,2}^{(0)} = &\ -c_{1,2}^{-1}q_{1,2}^{(1)}q_{1,2}^{(2)},
\end{align*}

and thus $q_{1,2}^{(1)}=q_{1,2}^{(2)}=0$. If $c_{1,2}=1$, then there are no restrictions over $q_{1,2}^{(0)}$. If $c_{1,2}\not=1$, then $p_1=p_2=q_{1,2}^{(0)}=0$. Again, all possible options are covered.

    \item [\rm (c)] Note that in this case the conditions are the same as in (b) by considering $x_2$ instead of $x_1$.
    
    \item [\rm (d)] It is clear that (\ref{thirdrel}) holds. By using the relations {\rm (}\ref{Secondnux1}{\rm )}, {\rm (}\ref{Thirnux1}{\rm )}, {\rm (}\ref{Firstnux2}{\rm )} and  {\rm (}\ref{Thirnux2}{\rm )} we obtain that
    \begin{align*}
        (c_{1,2}-1)p_2-b_2q_{1,2}^{(1)}-2q_{1,2}^{(1)}t = &\ (c_{1,2}^{-1}-1)p_1+(2b_2+b_1)c_{1,2}^{-1}q_{1,2}^{(2)}-2c_{1,2}^{-1}q_{1,2}^{(2)}t = 0, \\
        (c_{1,2}-1)q_{1,2}^{(0)} = &\ q_{1,2}^{(1)}q_{1,2}^{(2)}, \quad{\rm and} \\
        (c_{1,2}^{-1}-1)q_{1,2}^{(0)} = &\ - c_{1,2}^{-1}q_{1,2}^{(1)}q_{1,2}^{(2)}.
    \end{align*}
    
These equalities are satisfied when $q_{1,2}^{(1)}=q_{1,2}^{(2)}=0$. If $c_{1,2}=1$ then there are no restrictions on $q_{1,2}^{(0)}$; in other case, then $p_1=p_2=q_{1,2}^{(0)}=0$.
    \item [\rm (e)] From expression (\ref{thirdrel}) we have that 
    \begin{align*}
        &\ ((a_2-1)q_{1,2}^{(1)}+(1-c_{1,2})p_2)x_1 + (a_2-1)q_{1,2}^{(0)}-q_{1,2}^{(2)}p_2 \\
        &\ \ \ \ +\sum_{j=1}^n [ a_2jm_j\left(\left(a_2t+b_2\right)^{j-1}x_2-c_{1,2}t^{j-1}x_2 +p_2(t)\sum_{l=0}^{j-2}(a_2t+b_2)^{l}t^{j-2-l}\right) \\
        &\ \ \ \ - q_{1,2}^{(1)}jm_jt^{j-1}] = 0.
    \end{align*}

Again, necessarily the coefficient of $x_2$ is zero, that is, $jm_j((a_2t+b_2)^{j-1}-c_{1,2}t^{j-1}) = 0$, and hence necessarily $m_j=0$, for $1 \leq i \leq n$, which shows that the polynomial $p_1(t)$ is constant.

With respect to the coefficient of $x_1$ and the constant term, both must be zero, and so
    \begin{equation*}
       (a_2-1)q_{1,2}^{(1)}+(1-c_{1,2})p_2 = 0 \quad {\rm and} \quad (a_2-1)q_{1,2}^{(0)}-p_2q_{1,2}^{(2)}=0,
    \end{equation*}

or equivalently,
    \begin{equation*}
        q_{1,2}^{(0)}=\frac{q_{1,2}^{(2)}}{a_2-1}p_2 \quad {\rm and} \quad q_{1,2}^{(1)}=\frac{c_{1,2}-1}{a_2-1}p_2. 
    \end{equation*}
    
Expressions {\rm (}\ref{Secondnux1}{\rm )}, {\rm (}\ref{Thirnux1}{\rm )}, {\rm (}\ref{Firstnux2}{\rm )} and {\rm (}\ref{Thirnux2}{\rm )} imply that 
    \begin{align*}
     ((c_{1,2}-1)p_2-(a_2+1)q_{1,2}^{(1)})t+(c_{1,2}-1)\frac{b_2}{a_2-1}-b_2q_{1,2}^{(1)} = &\ 0, \\
 -2c_{1,2}^{-1}a_2^{-1}q_{1,2}^{(2)}t+(c_{1,2}^{-1}a_2^{-1}-1)p_1+2c_{1,2}^{-1}a_2^{-1}b_2q_{1,2}^{(2)} = &\ 0, \\
(c_{1,2}-1)q_{1,2}^{(0)} = &\ q_{1,2}^{(1)}q_{1,2}^{(2)}, \quad {\rm and} \\
(c_{1,2}^{-1}-1)q_{1,2}^{(0)} = &\ -c_{1,2}^{-1}q_{1,2}^{(1)}q_{1,2}^{(2)}.
    \end{align*}
    
In this way,
    \begin{align*}
         q_{1,2}^{(1)}&\ =\frac{c_{1,2}-1}{a_2-1}p_2, \\
        (c_{1,2}-1)p_2&\ =0, \\
         (c_{1,2}^{-1}a_2^{-1}-1)p_1&\ =0,  \quad {\rm and} \\
         q_{1,2}^{(2)}&\ =0,
    \end{align*}
    
so we get the restrictions $q_{1,2}^{(0)}=q_{1,2}^{(1)}=q_{1,2}^{(2)}=0$. Note that if $c_{1,2}=1$, then $p_2\in\Bbbk$ and $p_1=0$; or $c_{1,2}=a_2^{-1}$ with $p_1\in \Bbbk$ and $p_2=0$; or in other value of $c_{1,2}$, $p_1=p_2=0$.
    \item [\rm (f)] Equation (\ref{thirdrel}) becomes
    \begin{align*}
        (a_2-1)q_{1,2}^{(0)}+(a_2-1)q_{1,2}^{(1)}x_1+p_2x_1-c_{1,2}p_2x_1-q_{1,2}^{(2)}p_2 = &\ 0, \quad {\rm and} \\
     ((a_2-1)q_{1,2}^{(1)}+p_2-c_{1,2}p_2)x_1+(a_2-1)q_{1,2}^{(0)}-q_{1,2}^{(2)}p_2 = &\ 0,
    \end{align*}
    
whence 
     \begin{align*}
        q_{1,2}^{(1)}=\frac{c_{1,2}-1}{a_2-1}p_2 \quad {\rm and}\quad  q_{1,2}^{(0)}=\frac{q_{1,2}^{(2)}}{a_2-1}p_2.
    \end{align*}

From expression {\rm (}\ref{Secondnux1}{\rm )} we have that  $c_{1,2}b_1(a_2-1)x_2=0$, where the only options are $c_{1,2}=0$, $b_1=0$ or $a_2=1$. However, as it is clear none of these are possible.

    \item [\rm (g)] The conditions corresponding to this case are the same as \rm(e) since the hypotheses are completely analogous but replacing the indeterminate $x_1$ with $x_2$.
    
    \item [\rm (h)] This case is the same as \rm(f) by replacing the indeterminate $x_1$ with $x_2$.
    
    \item [\rm (i)] Equation (\ref{thirdrel}) leads to the following way of relating the coefficients:
    \begin{align*}
     &\ (a_1a_2-1)q_{1,2}^{(0)} + (a_2-1)a_1q_{1,2}^{(1)}x_1 + a_1(p_2x_1-c_{1,2}x_1p_2) + (a_1-1)a_2q_{1,2}^{(2)}x_2 \\
     &\ \ \ + a_2(x_2p_1-c_{1,2}p_1x_2) + (1-c_{1,2})p_1p_2-q_{1,2}^{(1)}p_1-q_{1,2}^{(2)}p_2=0.
    \end{align*}
    
After some computations, we get that $c_{1,2}=1$, $q_{1,2}^{(0)}=q_{1,2}^{(1)}=q_{1,2}^{(2)}=0$. Thus, expression {\rm (}\ref{Secondnux1}{\rm )} becomes
    \begin{align*}
     &\ (a_1^{-1}(a_2-1)b_1+(a_1^{-1}-1)b_2)x_1 +a_1^{-1}((1-a_1^{-1})p_2-(1+a_2)q_{1,2}^{(1)})t\\ 
    &\ \ \ + q_{1,2}^{(1)}(a_1^{-1}b_1a_2-a_1^{-1}b_1-b_2)+a_1^{-1}b_1p_2-\frac{b_2}{a_2-1}p_2 = 0.
    \end{align*}
    
By replacing the values found previously, we obtain that $b_1 = b_2 = 0$. Finally, note that relations {\rm (}\ref{Thirnux1}{\rm )}, {\rm (}\ref{Firstnux2}{\rm )} and  {\rm (}\ref{Thirnux2}{\rm )} are trivially satisfied.
\end{enumerate}

For the second assertion, it is enough to prove it for the generators $t$, $x_1$ and $x_2$:
\begin{align}
    (\nu_t \circ \nu_{x_1})(t) = &\ \nu_t(\sigma_1^{-1}(t))=a_1^{-1}(t-b_1), \label{comp11.} \\
    (\nu_{x_1} \circ \nu_t)(t) = &\ \nu_{x_1}(t) = a_1^{-1}(t-b_1), \label{comp11} \\
    (\nu_t \circ \nu_{x_1})(x_1) = &\ \nu_t(x_1)=a_1x_1+p_1'(t), \label{comp12.}\\ 
    (\nu_{x_1} \circ \nu_t)(x_1) = &\  a_1x_1+p_1'(a_1^{-1}(t-b_1)),\label{comp12} \\
    (\nu_t \circ \nu_{x_1})(x_2) = &\ c_{1,2}a_2x_2+c_{1,2}p_2'(t)+q_{1,2}^{(1)}, \quad {\rm and} \label{comp13.} \\  
 (\nu_{x_1} \circ \nu_t)(x_2) = &\ a_2c_{1,2}x_2+a_2q_{1,2}^{(1)} + p_2'(a_{1}^{-1}(t-b_1)). \label{comp13}
\end{align}

In any case, the two compositions shown in {\rm (}\ref{comp11.}{\rm )} and {\rm (}\ref{comp11}{\rm )} are the same. Relation {\rm (}\ref{comp12}{\rm )} was used to find the conditions of the polynomial $p_1(t)$ to be equal to the expression {\rm (}\ref{comp12.}{\rm )}. Thus, all of them are satisfied in every possible case.  As it is clear, relation {\rm (}\ref{comp13}{\rm )} holds in all cases to be equal to {\rm (}\ref{comp13.}{\rm )}. So, $\nu_{t}\circ\nu_{x_1}=\nu_{x_1}\circ\nu_{t}$.

Next, 
\begin{align}
    \nu_t \circ \nu_{x_2}(t) = &\ \nu_t(\sigma_2^{-1}(t))=a_2^{-1}(t-b_2), \notag \\
    \nu_{x_2} \circ \nu_t(t) = &\ \nu_{x_2}(t) = a_2^{-1}(t-b_2), \label{comp21} \\ 
     \nu_t \circ \nu_{x_2}(x_1) = &\ c_{1,2}^{-1}a_1x_1+c_{1,2}^{-1}p_1'(t)-c_{1,2}^{-1}q_{1,2}^{(2)}, \notag \\
     \nu_{x_2} \circ \nu_t(x_1) = &\ a_1c_{1,2}^{-1}x_1-a_1c_{1,2}^{-1}q_{1,2}^{(2)}+p_1'(a_{2}^{-1}(t-b_2)), \label{comp22} \\
     \nu_t \circ \nu_{x_2}(x_2) = &\ a_2x_2+p_2'(t), \quad {\rm and} \notag \\
     \nu_{x_2} \circ \nu_t(x_2) = &\ a_2x_2+p_2'(a_2^{-1}(t-b_2)). \label{comp23}
\end{align}

In any case, the two compositions shown in {\rm (}\ref{comp21}{\rm )} are the same. Relation {\rm (}\ref{comp23}{\rm )} was similarly used to find the conditions of the polynomial $p_1(t)$, whence they are satisfied in all cases. Note that relation {\rm (}\ref{comp22}{\rm )} works in all cases but case {\rm (}g{\rm )} only works when $c_{1,2}=1$. In this way, $\nu_{t}\circ\nu_{x_2}=\nu_{x_2}\circ\nu_{t}$.

Finally, note that 
\begin{align}
    \nu_{x_1} \circ \nu_{x_2}(t) = &\ a_2^{-1}(a_1^{-1}(t-b_1))-a_2^{-1}b_2, \label{comp31.}\\ 
    \nu_{x_2} \circ \nu_{x_1}(t) = &\ a_1^{-1}(a_2^{-1}(t-b_2))-a_1^{-1}b_1, \label{comp31} \\
    \nu_{x_1} \circ \nu_{x_2}(x_1) = &\ c_{1,2}^{-1}x_1-c_{1,2}^{-1}q_{1,2}^{(2)}, \label{comp32.}\\
    \nu_{x_2} \circ \nu_{x_1}(x_1) = &\ c_{1,2}^{-1}x_1-c_{1,2}^{-1}q_{1,2}^{(2)}, \label{comp32} \\
    \nu_{x_1} \circ \nu_{x_2}(x_2) = &\ c_{1,2}x_2+q_{1,2}^{(1)}, \label{comp33.} \quad {\rm and} \\
    \nu_{x_2} \circ \nu_{x_1}(x_2) = &\ c_{1,2}x_2 + q_{1,2}^{(1)}.\label{comp33}
\end{align}

In any case, relations {\rm (}\ref{comp32.}{\rm )}, {\rm (}\ref{comp32}{\rm )}, {\rm (}\ref{comp33.}{\rm )} and {\rm (}\ref{comp33}{\rm )} hold. Expressions {\rm (}\ref{comp31.}{\rm )} and  {\rm (}\ref{comp31}{\rm )} coincide when $b_2 = \frac{a_2-1}{a_1-1}b_1$. 
\end{proof}

Next, we formulate the first important result of the paper.

\begin{theorem}\label{smoothPBW2}
If a SPBW extension $\sigma(\Bbbk[t])\langle x_1,x_2\rangle$ satisfies one of the conditions {\rm (a)}-{\rm (i)}, except {\rm (f)} and {\rm (h)}, in Proposition  \ref{autoSkewPBW}, then it is differentially smooth.
\end{theorem}
\begin{proof}
We know that SPBW extensions of the form $\sigma(\Bbbk[t])\langle x_1,x_2\rangle$ have Gelfand-Kirillov dimension three \cite[Theorems 14 and 18]{Reyes2013}, so we are able to formulate a three-dimensional integrable calculus. With this aim, consider $\Omega^{1}(\sigma(\Bbbk[t])\langle x_1,x_2\rangle)$ a free right $\sigma(\Bbbk[t])\langle x_1,x_2\rangle$-module of rank three with generators $dt$, $dx_1$ and $dx_2$. Define a left $\sigma(\Bbbk[t])\langle x_1,x_2\rangle$-module structure by
\begin{equation}\label{relrightmod}
    f dt = dt \nu_t(f), \quad f dx_1 = dx_1\nu_{x_1}(f) \quad {\rm and} \quad  fdx_2 = dx_2\nu_{x_2}(f), 
\end{equation}

for all $f\in \sigma(\Bbbk[t])\langle x_1,x_2\rangle$, where $\nu_t$, $\nu_{x_1}$ and $\nu_{x_2}$ are the algebra automorphisms established in Proposition \ref{autoSkewPBW}. Notice that the relations in $\Omega^{1}(\sigma(\Bbbk[t])\langle x_1,x_2\rangle)$ are given by
{\footnotesize{
    \begin{align}
        tdt  = &\ dt t, & tdx_1 = &\ a_1^{-1}dx_1t - a_1^{-1}b_1dx_1, & tdx_2 = &\ a_2^{-1}dx_2t - a_2^{-1}b_2dx_2, \label{rel1} \\
        x_1dt = &\ a_1dtx_1+dtp_1'(t), & x_1dx_1 = &\ dx_1x_1, & x_1dx_2 = &\ dx_2c_{1,2}^{-1}x_1-dx_2c_{1,2}^{-1}q_{1,2}^{(2)}, \label{rel2} \\
        x_2dt = &\ a_2dtx_2+dtp_2'(t), & x_2dx_1 = &\ dx_1c_{1,2}x_2+dx_1q_{1,2}^{(1)}, &  x_2dx_2 = &\ dx_2x_2. \label{rel3} 
    \end{align}
}}

We want to extend $t\mapsto dt$, $x_1\mapsto dx_1$ and $x_2\mapsto dx_2$ to a map $d: \sigma(\Bbbk[t])\langle x_1,x_2\rangle \to \Omega^{1}(\sigma(\Bbbk[t])\langle x_1,x_2\rangle)$ satisfying Leibniz's rule. As expected, this is possible if Leibniz's rule is compatible with the non-trivial relations {\rm (}\ref{relPBW3.1}{\rm )} and {\rm (}\ref{relPBW3.3}{\rm )}, i.e. if the equalities
    \begin{align*}
        dx_1t+x_1dt &\ = a_1dtx_1+a_1tdx_1+b_1dx_1+dp_1(t), \\
        dx_2t+x_2dt &\ = a_2dtx_2+a_2tdx_2+b_2dx_2+dp_2(t), \quad {\rm and} \\
        dx_2x_1+x_2dx_1 &\ = c_{1,2}dx_1x_2+c_{1,2}x_1dx_2 + q_{1,2}^{(1)}dx_1+q_{1,2}^{(2)}dx_2, 
    \end{align*}

hold. In view of $tdt=dtt$ which defines the usual commutative calculus on the polynomial ring $\Bbbk[t]$, it follows that $dp_1(t) = dtp_1'(t)$ and $dp_2(t) = dtp_2'(t)$.

Now, we define $\Bbbk$-linear maps 
$$
\partial_t, \partial_{x_1}, \partial_{x_2}: \sigma(\Bbbk[t])\langle x_1,x_2\rangle \rightarrow \sigma(\Bbbk[t])\langle x_1,x_2\rangle
$$ 

such that
\begin{align*}
d(f) = dt\partial_t(f)+dx_1\partial_{x_1}(f) + dx_2\partial_{x_2}(f), \quad {\rm for\ all}\ f \in \sigma(\Bbbk[t])\langle x_1,x_2\rangle.
\end{align*}

Since $dt$, $dx_1$ and $dx_2$ are free generators of the right $\sigma(\Bbbk[t])\langle x_1,x_2\rangle$-module $\Omega^1(\sigma(\Bbbk[t])\langle x_1,x_2\rangle)$, these maps are well-defined. Then $d(a) = 0$ if and only if $\partial_t(a) = \partial_{x_1}(a) = \partial_{x_2}(a)=0$. Using relations {\rm (}\ref{relrightmod}{\rm )} and the definitions of the maps $\nu_t$, $\nu_{x_1}$ and $\nu_{x_2}$, we get that 
\begin{align*}
    \partial_t(t^kx_1^lx_2^s) = &\ kt^{k-1}x_1^lx_2^s, \\
    \partial_{x_1}(t^kx_1^lx_2^s) = &\ la_1^{-k}(t-b_1)^kx_1^{l-1}x_2^s, \quad {\rm and} \\
    \partial_{x_2}(t^kx_1^lx_2^s) = &\ a_2^{-k}c_{1,2}^{-l}s(t-b_2)^k(x_1-q_{1,2}^{(2)})^lx_2^{s-1}.
\end{align*}

Thus $d(f)=0$ if and only if $f$ is a scalar multiple of the identity. This shows that $(\Omega(\sigma(\Bbbk[t])\langle x_1,x_2\rangle), d)$ is connected with $\Omega(\sigma(\Bbbk[t])\langle x_1,x_2\rangle) = \bigoplus\limits_{i=0}^{3}\Omega^i (\sigma(\Bbbk[t])\langle x_1,x_2\rangle)$.

The universal extension of $d$ to higher forms compatible with {\rm (}\ref{rel1}{\rm )}, {\rm (}\ref{rel2}{\rm )} and {\rm (}\ref{rel3}{\rm )} gives the following rules for $\Omega^2 (\sigma(\Bbbk[t])\langle x_1,x_2\rangle)$:
\begin{align}
    dx_1\wedge dt = &\ -a_1dt\wedge dx_1, \notag \\
    dx_2\wedge dt = &\ -a_2dt\wedge dx_2, \quad {\rm and} \notag \\ 
    dx_2\wedge dx_1 = &\ -c_{1,2}dx_1\wedge dx_2. \label{relsecond}
\end{align}

Since the automorphisms $\nu_{t}$, $\nu_{x_1}$ and $\nu_{x_2}$ commute with each other, there are no additional relationships to the previous ones, so we can write
\begin{align*}
    \Omega^2 (\sigma(\Bbbk[t])\langle x_1,x_2\rangle) = &\ dt\wedge dx_1\sigma(\Bbbk[t])\langle x_1,x_2\rangle\\
    &\ \oplus dt\wedge dx_2\sigma(\Bbbk[t])\langle x_1,x_2\rangle\oplus dx_1\wedge dx_2\sigma(\Bbbk[t])\langle x_1,x_2\rangle.
\end{align*}

Note that 
\[
\Omega^3(\sigma(\Bbbk[t])\langle x_1,x_2\rangle) = \omega\sigma(\Bbbk[t])\langle x_1,x_2\rangle\cong \sigma(\Bbbk[t])\langle x_1,x_2\rangle
\]

as a right and left $\sigma(\Bbbk[t])\langle x_1,x_2\rangle$-module, with $\omega = dt\wedge dx_1 \wedge dx_2$, where $\nu_{\omega}=\nu_t\circ\nu_{x_1}\circ\nu_{x_2}$. This means that $\omega$ is a volume form of $\sigma(\Bbbk[t])\langle x_1,x_2\rangle$. From Proposition \ref{BrzezinskiSitarz2017Lemmas2.6and2.7} (2), $\omega$ is an integral form by setting
\begin{align*}
    \omega_1^1  = &\ dt,    & \omega_2^1 = &\ dx_1, & \omega_3^1  = &\ dx_2, \\
    \omega_1^2 = &\ dx_1\wedge dx_2, & \omega_2^2 = &\ dt\wedge dx_2, & \omega_3^2 = &\ dt\wedge dx_1, \\
    \bar{\omega}_1^1  = &\ dt,   &  \bar{\omega}_2^1= &\ -a_1^{-1}dx_1, &  \bar{\omega}_3^1= &\ a_2^{-1}c_{1,2}^{-1}dx_2, \\
    \bar{\omega}_1^2 = &\ a_1^{-1}a_2^{-1}dx_1\wedge dx_2 , & \bar{\omega}_2^2 = &\ -c_{1,2}^{-1}dt\wedge dx_2, & \bar{\omega}_3^2 = &\ dt\wedge dx_1.
\end{align*}

Indeed, let $\omega' = dta+dx_1b+dx_2c$ with $a,b,c \in \Bbbk$. Then 
\begin{align*}
\sum_{i=1}^{3}\omega_{i}^{1}\pi_{\omega}(\bar{\omega}_i^{2}\wedge \omega') = &\ dt\pi_{\omega}(a_1^{-1}a_2^{-1}adx_1\wedge dx_2\wedge dt)\\
&\ + dx_1\pi_{\omega}(-c_{1,2}^{-1}bdt\wedge dx_2\wedge dx_1)  + dx_2\pi_{\omega}(cdt\wedge dx_1\wedge dx_2) \\
= &\ dt a+dx_1 b + dx_2 c=\omega',
\end{align*}

and let $\omega'' = dt\wedge dx_1a+dt\wedge dx_2 b+dx_1\wedge dx_2 c$, with $a, b, c \in \Bbbk$. We obtain that
\begin{align*}
\sum_{i=1}^{3}\omega_{i}^{2}\pi_{\omega}(\bar{\omega}_i^{1}\wedge \omega'') = &\ dx_1\wedge dx_2\pi_{\omega}(cdt\wedge dx_1 \wedge dx_2) \\
    &\ +dt\wedge dx_2\pi_{\omega}(-a_{1}^{-1}bdx_1\wedge dt\wedge dx_2) \\
    &\ + dt\wedge dx_1\pi_{\omega}(a_{2}^{-1}c_{1,2}^{-1}adx_2 \wedge dt \wedge dx_1) \\
    = &\ dt\wedge dx_1a+dt\wedge dx_2b+dx_1\wedge dx_2= \omega''.
\end{align*}

Therefore, we have proved that $\sigma(\Bbbk[t])\langle x_1,x_2\rangle$ is differentially smooth.
\end{proof}

\subsection{SPBW extensions in three indeterminates}\label{SPBWTMThreeI}

In this section we develop a similar treatment to the presented in Section \ref{SPBWTMTwoI} but now we consider a SPBW extension of the form $\sigma(\Bbbk[t]) \langle x_1, x_2, x_3\rangle$ satisfying the defining relations
\begin{align*}
    x_1r(t) = &\ \sigma_1(r(t)) x_1 + \delta_1(r(t)), \\
    x_2 r(t) = &\ \sigma_2(r(t)) x_2 + \delta_2(r(t)), \\
    x_3 r(t) = &\ \sigma_3(r(t)) x_2 + \delta_3(r(t)), \\
    x_2 x_1 = &\ c_{1, 2}(t)x_1 x_2 + q_{1,2}^{(0)}(t) + q_{1,2}^{(1)}(t)x_1 + q_{1,2}^{(2)}(t)x_2+ q_{1,2}^{(3)}(t)x_3, \\
    x_3 x_1 = &\ c_{1, 3}(t)x_1 x_3 + q_{1,3}^{(0)}(t) + q_{1,3}^{(1)}(t)x_1 + q_{1,3}^{(2)}(t)x_2+q_{1,3}^{(3)}(t)x_3, \quad {\rm and} \\
    x_3 x_2 = &\ c_{2, 3}(t)x_2 x_3 + q_{2,3}^{(0)}(t) + q_{2,3}^{(1)}(t)x_1 + q_{2,3}^{(2)}(t)x_2+q_{2,3}^{(3)}(t)x_3, 
\end{align*}

where the elements $r(t)$, $c(t)$'s and $q(t)$'s belong to $\Bbbk[t]$ with $c_{1,2}(t), c_{1, 3}(t)$ and $c_{2, 3}(t)$ non-zero. Consider the automorphisms of $\Bbbk[t]$ given by $\sigma_i(t) = a_it+b_i$, for $a_i, b_i, \in \Bbbk$, with $a_i\not=0$, $i=1, 2, 3$, with the corresponding $\sigma_i$-derivations expressed as in (\ref{deltap}), that is, 
\begin{equation}
\delta_i(f)=\frac{f\left(\sigma_{i}(t)\right)-f(t)}{\sigma_{i}(t)-t}p_i(t), \text{ for } i=1,2,3, 
\end{equation}

where $p_i(t)$ is a fixed element of $\Bbbk[t]$ for each $i$. The relations between $t, x_1, x_2, x_3$ can be expressed as
\begin{align}
    x_it = &\ a_itx_i+b_ix_i+p_i(t), \quad {\rm for\ every}\ i, \quad{\rm and} \label{twithxi3} \\
    x_j x_i = &\ c_{i, j}(t) x_i x_j + q_{i,j}^{(0)} + q_{i,j}^{(1)}x_1 + q_{i,j}^{(2)}x_2+q_{i,j}^{(3)}x_3, \quad {\rm for}\ i < j. \label{xiandxj3} 
\end{align}

The following result is the natural extension of Proposition \ref{autoSkewPBW}.

\begin{proposition}\label{autoSkewPBWdim3} 
Let
\begin{align}
   \nu_t(t) = &\ t, & \nu_t(x_i) = &\ a_ix_i + p_i'(t),\quad i = 1, 2, 3, \label{Autoskew31} \\ 
   \nu_{x_i}(t) = &\ \sigma_{i}^{-1}(t), & \nu_{x_i}(x_i) = &\ x_i, \quad i = 1, 2, 3, \label{Autoskew32} \\ 
   \nu_{x_i}(x_j) = &\ c_{i,j}x_j+q_{i,j}^{(i)}, & \nu_{x_j}(x_i) = &\ c_{j,i}^{-1}x_i-c_{j,i}^{-1}q_{j,i}^{(j)}, \quad i < j, \label{Autoskew33}
\end{align}

where $p_i'(t)$ are the $t$-derivatives of $p_i(t)$ for $i=1, 2, 3$, and $c_{i,j}, q_{i,j}^{(k)} \in \Bbbk$, $c_{i,j}\not =0$, for all $1\leq i,j \leq 3$ and $0\leq k \leq 3$.
\begin{enumerate}
\item [\rm (1)] Leibniz's rule holds in the cases listed in Table \ref{Leibnizsruletwoindeterminates(2)}. The maps defined by {\rm (}\ref{Autoskew31}{\rm )}, {\rm (}\ref{Autoskew32}{\rm )} and {\rm (}\ref{Autoskew33}{\rm )} simultaneously extend to algebra automorphisms $\nu_t, \nu_{x_i}$, $i=1, 2, 3$, of $\sigma(\Bbbk[t])\langle x_1, x_2, x_3\rangle$ in cases {\rm (a) - (d)}.

\begin{table}[h]
\caption{Leibniz's rule}
\label{Leibnizsruletwoindeterminates(2)}
\begin{center}
\resizebox{12.5cm}{!}{
\setlength\extrarowheight{6pt}
\begin{tabular}{ |c|c|c|c| } 
\hline
{\rm Case} & {\rm Possibilities for} $a_i$, $b_i$, $i=1, 2, 3$ & {\em Polynomials} $p_i(t)$, $i=1, 2, 3$ & {\em Restrictions} \\
\hline
\multirow{2}{*}{{\rm (a)}} & \multirow{2}{*}{$a_i=1$, $b_i=0$, \text{ for all } $i=1, 2, 3$} & $p_i(t)=0$ for all $i=1, 2, 3$ & $q_{i,j}^{(k)}=0$,  $c_{i,j}\in\Bbbk^{\ast}$ for all $i, j=1, 2, 3$, $k \geq 0$ \\ \cline{3-4}
&  & $p_i(t)\in\Bbbk[t], \text{ for all } i=1, 2, 3$ & $q_{i,j}^{(k)}=0$,  $c_{i,j}=1$, $q_{i,j}^{(0)} \in\Bbbk$, \text{ for all } $i, j=1, 2, 3$, $k>0$ \\
\hline
\multirow{2}{*}{{\rm (b)}} & \multirow{2}{*}{$a_i=1$, for all $i=1, 2, 3$,  $b_l \not=0$, for some $l=1, 2, 3$} & $p_i(t)=p_i$, $p_i\in\Bbbk$, for all $i=1, 2, 3$ & $c_{i,j}=1$, $q_{i,j}^{(k)}=0$, $q_{i,j}^{(0)}\in\Bbbk$, \text{ for all } $i, j=1, 2, 3$, $k>0$ \\ \cline{3-4}
&  & $p_i(t)=0$, for all $i=1, 2, 3$ & $q_{i,j}^{(k)}=0$,  $c_{i,j}\in\Bbbk^{\ast}$, \text{ for all } $i, j=1, 2, 3$, $k\geq 0$ \\
\hline
\rm (c) & $a_r\not=1$, $a_s=1$, $b_s=0$, for $r\in S\subsetneq \{1, 2, 3\}$ and $s\in S^c$& $p_s(t)=0$, $p_r(t)=p_r\left(t+\frac{b_r}{a_r-1}\right)$, for $r\in S\subsetneq \{1, 2, 3\}$ and $s\in S^c$, $p_r\in \Bbbk$ & $q_{i,j}^{(k)}=0$,  $c_{i,j}=1$, \text{ for all } $i, j=1, 2, 3$, $k \geq 0$ \\ 
\hline
\rm (d) & $a_i\not =1$, $b_i=0$, \text{ for all } $i=1, 2, 3$  & $p_i(t)=p_it$, $p_i\in\Bbbk$, \text{ for all } $i=1, 2, 3$ & $q_{i,j}^{(k)}=0$,  $c_{i,j}=1$, \text{ for all } $i, j=1, 2, 3$, $k \geq 0$ \\   
\hline
\end{tabular}
}
\end{center}
\end{table}
    
\item [\rm (2)] Precisely, in cases {\rm (a) - (d)}, we have that 
    \begin{equation}\label{Eq1skew3}
   \nu_t \circ \nu_{x_i} = \nu_{x_i} \circ \nu_t \quad {\rm and} \quad \nu_{x_i} \circ \nu_{x_j} = \nu_{x_j} \circ \nu_{x_i}, \quad {\rm for}\ i = 1, 2, 3.
    \end{equation}
\end{enumerate}
\end{proposition}
\begin{proof}
For the first assertion, the map $\nu_t$ can be extended to an
algebra homomorphism if and only if the definitions of $\nu_t(t)$ and $\nu_t(x_i)$, $i=1, 2, 3$ respect relations {\rm (}\ref{twithxi3}{\rm )}, and {\rm (}\ref{xiandxj3}{\rm )}, i.e.
\begin{align*}
   \nu_t(x_i)\nu_t(t)-\nu_t(a_it+b_i)\nu_t(x_i) = &\ \nu_t(p_i(t)), \\
   \nu_t(x_j)\nu_t(x_i)-c_{i,j}\nu_t(x_i)\nu_t(x_j) = &\ q_{i,j}^{(0)}+q_{i,j}^{(1)}\nu_t(x_1)+q_{i,j}^{(2)}\nu_t(x_2)+q_{i,j}^{(3)}\nu_t(x_3), 
\end{align*}

for $i<j$. This yields the equalities
\begin{equation} \label{firstsecondrel3}
    ((a_i-1)t+b_i)p_i'(t)=(a_i-1)p_i(t), \quad i = 1, 2, 3,
\end{equation}

and
{\small{
\begin{align}
  &\ (a_ia_j-1)q_{i,j}^{(0)} + a_i(p_j'(t)x_i-c_{i,j}x_ip_j'(t)) + a_j(x_jp_i'(t)-c_{i,j}p_i'(t)x_j) + (a_ja_i-1)q_{i,j}^{(0)} \notag \\ 
   &\ +\sum_{r=1}^{3}q_{i,j}^{(r)}(a_ia_j-a_r)x_r +(1-c_{i,j})p_i'(t)p_j'(t) - q_{i,j}^{(1)}p_1'(t)-q_{i,j}^{(2)}p_2'(t)-q_{i,j}^{(3)}p_3'(t)=0. \label{thirdrel3}
\end{align}
}}

The map $\nu_{x_i}$ can be extended to an algebra homomorphism if and only if the definitions of $\nu_{x_i}(t)$ and $\nu_{x_i}(x_j)$, for each $i$, respect relations {\rm (}\ref{twithxi3}{\rm )}, and {\rm (}\ref{xiandxj3}{\rm )}, i.e. 
{\small{
\begin{align*}
   \nu_{x_i}(x_i)\nu_{x_i}(t)-\nu_{x_i}(a_it+b_i)\nu_{x_i}(x_i) = &\ \nu_{x_i}(p_i(t)), \\
   \nu_{x_i}(x_j)\nu_{x_i}(t)-\nu_{x_i}(a_jt+b_j)\nu_{x_i}(x_j) = &\ \nu_{x_i}(p_j(t)), \quad {\rm and} \\
   \nu_{x_i}(x_j)\nu_{x_i}(x_i)-c_{i,j}\nu_{x_i}(x_i)\nu_{x_i}(x_j) = &\ q_{i,j}^{(0)}+q_{i,j}^{(1)}\nu_{x_i}(x_1)+q_{i,j}^{(2)}\nu_{x_i}(x_2)+q_{i,j}^{(3)}\nu_{x_i}(x_3),
\end{align*}
}}

for $i < j$. In this way, 
\begin{align}
   &\ a_{i}^{-1}p_i(t)=p_i(a_{i}^{-1}(t-b_i)), \label{Firstnux13} \\
   &\ c_{i,j}(a_i^{-1}(a_jb_i+b_j-b_i)-b_j)x_j+a_i^{-1}(c_{i,j}p_j(t)-(1+a_j)q_{i,j}^{(i)}t), \notag \\
   &\ \ +q_{i,j}^{(i)}(a_i^{-1}b_i(a_j-1)-b_j)=p_j(a_i^{-1}(t-b_i)), \quad {\rm and} \label{Secondnux13}\\
   &\ (c_{1,2}-1)q_{1,2}^{(0)}+q_{1,2}^{(3)}(c_{1,2}-c_{1,3})x_3-q_{1,2}^{(2)}q_{1,2}^{(1)}-q_{1,2}^{(3)}q_{1,3}^{(1)} = 0, \notag \\
   &\ (c_{1,3}-1)q_{1,3}^{(0)}+q_{1,3}^{(2)}(c_{1,3}-c_{1,2})x_2-q_{1,3}^{(3)}q_{1,3}^{(1)}-q_{1,3}^{(2)}q_{1,2}^{(1)} = 0, \notag \\
   &\ (c_{2,3}-1)q_{2,3}^{(0)}+q_{2,3}^{(1)}(c_{2,3}-c_{1,2}^{-1})x_1-q_{2,3}^{(3)}q_{2,3}^{(2)}+c_{1,2}^{-1}q_{1,2}^{(2)}q_{2,3}^{(1)} = 0 .\label{Thirnux13}
\end{align}

Then, the map $\nu_{x_j}$ can be extended to an algebra homomorphism if and only if the definitions of $\nu_{x_j}(t)$ and $\nu_{x_j}(x_i)$, $i=1, 2, 3$ respect relations {\rm (}\ref{twithxi3}{\rm )}, and {\rm (}\ref{xiandxj3}{\rm )}, and so
{\small{
\begin{align*}
   \nu_{x_j}(x_i)\nu_{x_j}(t)-\nu_{x_j}(a_it+b_i)\nu_{x_j}(x_i) = &\ \nu_{x_j}(p_i(t)), \\
   \nu_{x_j}(x_j)\nu_{x_j}(t)-\nu_{x_j}(a_jt+b_j)\nu_{x_j}(x_j) = &\ \nu_{x_j}(p_j(t)), \quad {\rm and} \\
   \nu_{x_j}(x_j)\nu_{x_j}(x_i)-c_{i,j}\nu_{x_j}(x_i)\nu_{x_j}(x_j) = &\ q_{i,j}^{(0)}+q_{i,j}^{(1)}\nu_{x_j}(x_1)+q_{i,j}^{(2)}\nu_{x_j}(x_2)+q_{i,j}^{(3)}\nu_{x_j}(x_3),
\end{align*}
}}

for $i < j$. We obtain the expressions given by
\begin{align}
   &\  c_{i,j}^{-1}(a_j^{-1}(b_i+a_ib_j-b_j)-b_i)x_i+c_{i,j}^{-1}a_j^{-1}(p_i(t)-(1+a_i)q_{i,j}^{(j)}t) \notag\\
   &\ \ +q_{i,j}^{(j)}c_{i,j}^{-1}(a_j^{-1}b_j(1+a_i)+b_i)=p_i(a_j^{-1}(t-b_j)) \label{Firstnux23}\\
   &\  a_{j}^{-1}p_j(t)=p_j(a_{j}^{-1}(t-b_j)) \label{Secondnux23} \\
   &\ (c_{1,2}^{-1}-1)q_{1,2}^{(0)}+q_{1,2}^{(3)}(c_{1,2}^{-1}-c_{2,3}^{-1})x_3+c_{1,2}^{-1}q_{1,2}^{(2)}q_{1,2}^{(1)}-q_{1,2}^{(3)}q_{2,3}^{(2)} = 0, \notag \\
   &\ (c_{1,3}^{-1}-1)q_{1,3}^{(0)}+q_{1,3}^{(2)}(c_{1,3}^{-1}-c_{2,3}^{-1})x_2+c_{1,3}^{-1}q_{1,3}^{(3)}q_{1,3}^{(1)}+c_{2,3}^{-1}q_{1,3}^{(2)}q_{2,3}^{(3)} = 0, \notag \\
   &\ (c_{2,3}^{-1}-1)q_{2,3}^{(0)}+q_{2,3}^{(1)}(c_{2,3}^{-1}-c_{1,3}^{-1})x_1+c_{2,3}^{-1}q_{2,3}^{(3)}q_{2,3}^{(2)}+c_{1,3}^{-1}q_{1,3}^{(3)}q_{2,3}^{(1)} = 0 .\label{Thirnux23}
\end{align}

Expressions (\ref{firstsecondrel3}) are the same as in \cite[Lemma 3.1]{Brzezinski2015}. It should be noted that these equations are independent of each other, which means that there are different combinations considering the values of $a_i$, $b_i$ for $i=1,2,3$. Let us see.

We consider the expression $p_i(t) = \sum\limits_{j=0}^n m_{i,j}t^j$, for every $i = 1, 2,3 $.

\begin{enumerate}
    \item [\rm (a)] Equation (\ref{thirdrel3}) leads to the coefficients that accompany $x_i$, so these must be zero, $ m_{i,j}(1-c_{i,k})=0$. This implies that $m_{i,j}=0$, for $1 \leq j \leq n$, and so the polynomials $p_i(t)$ are constants or $c_{i,k}=1$, for $i = 1, 2, 3$ and $i < k$.
    From relations {\rm (}\ref{Secondnux13}{\rm )}, {\rm (}\ref{Thirnux13}{\rm )}, {\rm (}\ref{Firstnux23}{\rm )} and  {\rm (}\ref{Thirnux23}{\rm )}, we get that if $p_i(t)=p_i\in \Bbbk$ then $q_{i,j}^{(k)}=0$ with $k>0$, whence $c_{i,j}=1$ and $q_{i,j}^{(0)}$ has no restrictions. Also, it is necessary that $q_{i,k}^{(i)}m_{i,j}+q_{i,k}^{(k)}m_{k,j}=0$ for all $1 \leq i \leq 3$, which shows that $q_{i,j}^{(k)}=0$ with $k>0$ and there is not restrictions over polynomials $p_i(t)$. 

    In this way, we have considered all possibilities.
    \item [\rm (b)] Equation (\ref{thirdrel3}) leads that the coefficient that accompany $x_l$ must be zero, that is
    \begin{equation*}
       jm_{i,j}((t+b_l)^{j-1}-c_{i,l}t^{j-1})=0.
    \end{equation*}
    This implies that all the coefficients $m_{i,j}$ are zero, whence the polynomial $p_i(t)$ is constant.
    
    From relations {\rm (}\ref{Secondnux13}{\rm )}, {\rm (}\ref{Thirnux13}{\rm )}, {\rm (}\ref{Firstnux23}{\rm )} and  {\rm (}\ref{Thirnux23}{\rm )}, we obtain that $q_{i,j}^{(k)}=0$ with $k>0$, $c_{i,j} = 1$ and there are no restrictions on $q_{i,j}^{(0)}$.

    Again, all options are covered.
    
    \item [\rm (c)] Equation (\ref{thirdrel}) implies that the coefficient of $x_r$ is zero,
    \begin{equation*}
       jm_{i,j}((a_rt+b_r)^{j-1}-c_{i,r}t^{j-1})=0. 
    \end{equation*}
    
    Thus, $m_{i,j}=0$ for $1 \leq i \leq 3$, whence the polynomial $p_i(t)$ is constant. Also, if we focus on the coefficient that accompany $x_s$ and the constant element, both must be zero, 
    \begin{equation*}
        q_{i,r}^{(0)}=\frac{q_{i,r}^{(r)}}{a_r-1}p_r \quad {\rm and} \quad q_{i,r}^{(s)}=\frac{c_{i,r}-1}{a_r-1}p_r.
    \end{equation*}
    
By using expressions {\rm (}\ref{Secondnux13}{\rm )}, {\rm (}\ref{Thirnux13}{\rm )}, {\rm (}\ref{Firstnux23}{\rm )} and  {\rm (}\ref{Thirnux23}{\rm )} we obtain the restrictions $q_{i,j}^{(k)}=0$ for $1\leq i,j \leq 3$ and $k\geq 0$. Also, note that $c_{s,j}=1$ and $p_s=0$, for $s \in S$.
    
\item [\rm (d)] Equation (\ref{thirdrel3}) shows that $c_{i,j}=1$, $q_{i,j}^{(k)}=0$ for $1\leq i,j \leq 3$ and $k\geq 0$. If we consider the expression {\rm (}\ref{Secondnux13}{\rm )} then we get the condition $b_i = 0$ for each $i$. It is clear that relations {\rm (}\ref{Thirnux13}{\rm )}, {\rm (}\ref{Firstnux23}{\rm )} and {\rm (}\ref{Thirnux23}{\rm )} hold.
\end{enumerate}

For the second assertion, it is enough to prove it for the generators $t$, $x_1$ and $x_2$. Note that 
\begin{align}
    \nu_t \circ \nu_{x_i}(t) = &\ \nu_t(\sigma_i^{-1}(t))=a_i^{-1}(t-b_i), \label{comp113.}\\
    \nu_{x_i} \circ \nu_t(t) = &\ \nu_{x_i}(t) = a_i^{-1}(t-b_i), \label{comp113} \\
    \nu_t \circ \nu_{x_i}(x_i) = &\ \nu_t(x_i)=a_ix_i+p_i'(t), \label{comp123.} \\ \nu_{x_i} \circ \nu_t(x_i) = &\ a_ix_i+p_i'(a_i^{-1}(t-b_i)), \label{comp123} \\
    \nu_t \circ \nu_{x_i}(x_j) = &\ c_{i,j}a_jx_j+c_{i,j}p_j'(t) + q_{i,j}^{(i)}, \label{comp133.} \\ 
    \nu_{x_i} \circ \nu_t(x_j) = &\ a_jc_{i,j}x_j+a_jq_{i,j}^{(i)}+p_j'(a_{i}^{-1}(t-b_i)), \quad i < j, \label{comp133} \\
    \nu_t \circ \nu_{x_i}(x_k) = &\ c_{k,i}^{-1}a_kx_k+c_{k,i}^{-1}p_k'(t)-c_{k,i}^{-1}q_{k,i}^{(i)}, \quad {\rm and} \label{comp143.} \\
    \nu_{x_i} \circ \nu_t(x_k) = &\ a_kc_{k,i}^{-1}x_k-a_kc_{k,i}^{-1}q_{k,i}^{(i)}+p_k'(a_{i}^{-1}(t-b_i)), \quad i > k.\label{comp143}
\end{align}

In any case, the two compositions shown in {\rm (}\ref{comp113}{\rm )} are the same. Relation {\rm (}\ref{comp123}{\rm )} was similarly used to find the conditions of the polynomial $p_i(t)$. Thus, they hold in all cases, and relations {\rm (}\ref{comp133}{\rm )} and {\rm (}\ref{comp143}{\rm )} are correct. Then, $\nu_{t}\circ\nu_{x_i}=\nu_{x_i}\circ\nu_{t}$.

Finally, we have that 
\begin{align}
    \nu_{x_i} \circ \nu_{x_j}(t) = &\ a_j^{-1}(a_i^{-1}(t-b_i))-a_j^{-1}b_j, \label{comp313.} \\ 
    \nu_{x_j} \circ \nu_{x_i}(t) = &\ a_i^{-1}(a_j^{-1}(t-b_j))-a_i^{-1}b_i, \label{comp313} \\
    \nu_{x_i} \circ \nu_{x_j}(x_i) = &\ c_{i,j}^{-1}x_i-c_{i,j}^{-1}q_{i,j}^{(j)}, \label{comp323.} \\
    \nu_{x_j} \circ \nu_{x_i}(x_i) = &\ c_{i,j}^{-1}x_i-c_{i,j}^{-1}q_{i,j}^{(j)}, \label{comp323} \\
    \nu_{x_i} \circ \nu_{x_j}(x_j) = &\ c_{i,j}x_j+q_{i,j}^{(i)}, \quad {\rm and} \label{comp333.} \\
    \nu_{x_j} \circ \nu_{x_i}(x_j) = &\ c_{i,j}x_j+q_{i,j}^{(i)}. \label{comp333}
\end{align}

In any case, relations {\rm (}\ref{comp323}{\rm )} and {\rm (}\ref{comp333}{\rm )} hold. Relation {\rm (}\ref{comp313}{\rm )} works in all cases. Then, $\nu_{x_i}\circ\nu_{x_j}=\nu_{x_j}\circ\nu_{x_i}$.
\end{proof}

\begin{theorem}\label{smoothPBW3}
If a SPBW extension $\sigma(\Bbbk[t])\langle x_1, x_2, x_3\rangle$ satisfies one of the conditions {\rm (a)-(d)} in Proposition \ref{autoSkewPBWdim3}, then it is differentially smooth.
\end{theorem}
\begin{proof}
Since the SPBW extension $\sigma(\Bbbk[t])\langle x_1, x_2, x_3\rangle$ has Gelfand-Kirillov dimension $4$, a $4$-dimensional integrable calculus can be constructed. We know that we have to consider $\Omega^{1}(\sigma(\Bbbk[t])\langle x_1, x_2, x_3\rangle)$, a free right $\sigma(\Bbbk[t])\langle x_1, x_2, x_3\rangle$-module of rank $4$ with generators $dt$, $dx_1, dx_2, dx_3$. Define a left $\sigma(\Bbbk[t])\langle x_1, x_2, x_3\rangle$-module structure by
    \begin{equation}\label{relrightmod3}
        adt = dt \nu_t(a), \quad adx_i = dx_i\nu_{x_i}(a),  \quad {\rm for\ all}\ 1\leq i \leq 3, a\in \sigma(\Bbbk[t])\langle x_1, x_2, x_3\rangle,
    \end{equation}
    
where $\nu_t$, $\nu_{x_i}$, $1 \leq i \leq 3$ are the algebra automorphisms established in Proposition \ref{autoSkewPBWdim3}. Notice that the relations in $\Omega^{1}(\sigma(\Bbbk[t])\langle x_1, x_2, x_3\rangle)$ are given by
    \begin{align}
        tdt = &\ dt t & tdx_i = &\ dx_ia_i^{-1}(t-b_i), \quad \text{ for all } 1\leq i \leq 3, \label{rel13}  \\
        x_idx_i = &\ dx_ix_i,  &  x_idt = &\ dt(a_ix_i+p_i'(t)), \quad \text{ for all } 1\leq i \leq 3, \label{rel23} 
         \end{align}

and 
        \begin{align}
        x_idx_j = &\ dx_j(c_{i,j}^{-1}x_i-c_{i,j}^{-1}q_{i,j}^{(j)}), \text{ for } i<j, \\   x_idx_j = &\ dx_j(c_{j,i}x_i+q_{j,i}^{(j)}), {\rm for}\ i > j. \label{rel33} 
    \end{align}

We want to extend $t\mapsto dt$, $x_i\mapsto dx_i$, $1\leq i \leq 3$ to a map $d: \sigma(\Bbbk[t])\langle x_1, x_2, x_3\rangle \to \Omega^{1}(\sigma(\Bbbk[t])\langle x_1, x_2, x_3\rangle)$ satisfying the Leibniz's rule. This is possible if the Leibniz's rule is compatible with the non-trivial relations {\rm (}\ref{twithxi3}{\rm )} and {\rm (}\ref{xiandxj3}{\rm )}, i.e.
    \begin{align*}
        dx_it+x_idt &\ = a_idtx_i+a_itdx_i+b_idx_i+dp_i(t), \quad \text{for}\ 1\leq i \leq 3 \\
        dx_jx_i+x_jdx_i &\ = c_{i,j}dx_ix_j+c_{i,j}x_idx_j +\sum_{k=1}^{3}q_{i,j}^{(k)}dx_k, \text{ for } i<j.
    \end{align*}
    
Due to that $t dt=dt t$, which defines the usual commutative calculus on the polynomial ring $\Bbbk[t]$, $dp_i(t) = dtp_i'(t)$, $1 \leq i \leq 3$.

Define $\Bbbk$-linear maps $$
\partial_t, \partial_{x_i}: \sigma(\Bbbk[t])\langle x_1, x_2, x_3\rangle \rightarrow \sigma(\Bbbk[t])\langle x_1, x_2, x_3\rangle
$$ 

such that
\begin{align*}
d(a)=dt\partial_t(a)+\sum_{i=1}^{3}dx_i\partial_{x_i}(a), \text{ for all } a \in \sigma(\Bbbk[t])\langle x_1, x_2, x_3\rangle.
\end{align*}

These maps are well-defined since $dt$, $dx_i$, $1\leq i \leq 3$ are free generators of the right $\sigma(\Bbbk[t])\langle x_1, x_2, x_3\rangle$-module $\Omega^1(\sigma(\Bbbk[t])\langle x_1, x_2, x_3\rangle)$. With that, $d(a)=0$ if and only if $\partial_t(a)=\partial_{x_i}(a)=0$, $1 \leq i \leq 3$. Using relations {\rm (}\ref{relrightmod3}{\rm )} and definitions of the maps $\nu_t$, $\nu_{x_i}$, $1 \leq i \leq 3$, we obtain that
\begin{align}
\partial_t(t^kx_1^{l_1}x_2^{l_2}x_3^{l_3}) = &\ kt^{k-1}x_1^{l_1}x_2^{l_2}x_3^{l_3}, \\
\partial_{x_1}(t^kx_1^{l_1}x_2^{l_2}x_3^{l_3}) = &\ l_{1}a_1^{-k}(t-b_1)^kx_1^{l_1-1}x_{2}^{l_{2}}x_3^{l_3}, \notag \\
\partial_{x_2}(t^kx_1^{l_1}x_2^{l_2}x_3^{l_3}) = &\ l_{2}a_2^{-k}(t-b_2)^kc_{1,2}^{-l_1}(x_1-q_{1,2}^{(2)})^{l_1}x_2^{l_2-1}x_{3}^{l_{3}}, \quad {\rm and} \notag \\
\partial_{x_3}(t^kx_1^{l_1}x_2^{l_2}x_3^{l_3}) = &\ l_{3}a_3^{-k}(t-b_3)^kc_{1,3}^{-l_1}(x_1-q_{1,3}^{(3)})^{l_1}c_{2,3}^{-l_2}(x_2-q_{2,3}^{(3)})^{l_2} x_3^{l_3-1}. \notag
\end{align}

Then, $d(a)=0$ if and only if $a$ is a scalar multiple of the identity. This shows that $\Omega(\sigma(\Bbbk[t])\langle x_1, x_2, x_3\rangle,d)$ is connected, where 
\[
\Omega(\sigma(\Bbbk[t])\langle x_1, x_2, x_3\rangle) = \bigoplus_{i=0}^{4}\Omega^i(\sigma(\Bbbk[t])\langle x_1, x_2, x_3\rangle).
\]

The universal extension of $d$ to higher forms compatible with {\rm (}\ref{rel13}{\rm )}, {\rm (}\ref{rel23}{\rm )} and {\rm (}\ref{rel33}{\rm )} gives the following rules for $\Omega^l(\sigma(\Bbbk[t])\langle x_1, x_2, x_3\rangle)$ $(l = 2, 3)$:
\begin{align}\label{rel3wedge}
 dx_i\wedge dt = &\ -a_idt\wedge dx_i, \quad \text{ for } 1\leq i \leq 3, \notag \\
 dx_j\wedge dx_i = &\ -c_{i,j}dx_i\wedge dx_j, \quad \text{ for } 1\leq i < j \leq 3, \notag \\
dx_2\wedge dx_1\wedge dt = &\ -c_{1,2}a_1a_2dt\wedge dx_1\wedge dx_2,  \\
dx_3\wedge dx_2\wedge dt = &\ -c_{2,3}a_2a_3dt\wedge dx_2\wedge dx_3,  \notag \\
dx_3\wedge dx_2\wedge dx_1 = &\ -c_{1,2}c_{1,3}c_{2,3}dx_1\wedge dx_2\wedge dx_3, \quad {\rm and} \notag \\
dx_3\wedge dx_1\wedge dt = &\ -a_{1}a_{3}c_{1,3}dt\wedge dx_1\wedge dx_3. \notag
\end{align}

Since the automorphisms $\nu_{t}$, $\nu_{x_i}$, $1\leq i \leq 3$ commute with each other, there are no additional relationships to the previous ones, so
\begin{align*}
\Omega^{3}(\sigma(\Bbbk[t])\langle x_1, x_2, x_3\rangle) = &\ [dt\wedge dx_1\wedge dx_2 \oplus dt\wedge dx_2\wedge dx_3 \oplus dt\wedge dx_1\wedge dx_3 \\
&\ \oplus dx_1\wedge dx_2\wedge dx_3] \sigma(\Bbbk[t])\langle x_1, x_2, x_3\rangle.
\end{align*}

Now, 
$$
\Omega^4(\sigma(\Bbbk[t])\langle x_1, x_2, x_3\rangle) = \omega\sigma(\Bbbk[t])\langle x_1, x_2, x_3\rangle\cong \sigma(\Bbbk[t])\langle x_1, x_2, x_3\rangle
$$ 

as a right and left $\sigma(\Bbbk[t])\langle x_1, x_2, x_3\rangle$-module, with $\omega=dt\wedge dx_1 \wedge dx_2 \wedge dx_3$, where $\nu_{\omega}=\nu_t\circ\nu_{x_1}\circ\nu_{x_2}\circ\nu_{x_3}$, this means that $\omega$ is a volume form of $\sigma(\Bbbk[t])\langle x_1, x_2, x_3\rangle$. From Proposition \ref{BrzezinskiSitarz2017Lemmas2.6and2.7} (2), it follows that $\omega$ is an integral form by setting
\begin{align*}
    &\ \omega_1^1= dt, \quad \omega_j^1= dx_{j-1}, \text{ for } 2\leq j \leq 4, \\
    &\ \omega_1^2=dt\wedge dx_1, \quad \omega_2^2=dt\wedge dx_2, \quad \omega_3^2= dt\wedge dx_3, \quad \omega_4^2= dx_1\wedge dx_2, \\
    &\ \omega_5^2= dx_1\wedge dx_3, \quad  \omega_6^2= dx_2\wedge dx_3,   \\
    &\ \omega_1^3=dt\wedge dx_1\wedge dx_2, \quad \omega_2^3= dt\wedge dx_2 \wedge dx_3, \quad \omega_3^3= dx_1\wedge dx_2 \wedge dx_3, \\
    &\ \omega_4^3= dt\wedge dx_1 \wedge dx_3, \\
     &\ \bar{\omega}_1^1=-a_3^{-1}c_{1,3}^{-1}c_{2,3}^{-1}dx_3, \quad \bar{\omega}_2^1=-a_1^{-1}dx_1,\quad \bar{\omega}_3^1=dt, \quad \bar{\omega}_4^1=c_{1,2}^{-1}a_2^{-1}dx_2 \\
    &\ \bar{\omega}_1^2=a_2^{-1}a_{3}^{-1}c_{1,3}^{-1}c_{1,2}^{-1}dx_2\wedge dx_3, \\
    &\ \bar{\omega}_2^2=-a_1^{-1}a_3^{-1}c_{2,3}^{-1}dx_1\wedge dx_3, \quad \bar{\omega}_3^2 = a_1^{-1}a_2^{-1}dx_1\wedge dx_2, \quad \bar{\omega}_4^2 = c_{1,3}^{-1}c_{2,3}^{-1}dt\wedge dx_3, \\
    &\ \bar{\omega}_5^2 = -c_{1,2}^{-1}dt\wedge dx_2, \quad \bar{\omega}_6^2 = dt\wedge dx_1,\\
    &\ \bar{\omega}_1^3=-a_1^{-1}a_2^{-1}a_3^{-1}dx_1\wedge dx_2\wedge dx_3, \quad \bar{\omega}_2^3=c_{1,3}^{-1}c_{1,2}^{-1}dt\wedge dx_2 \wedge dx_3, \\
    &\ \bar{\omega}_3^3=-c_{2,3}^{-1}dt\wedge dx_1\wedge dx_3, \quad \bar{\omega}_4^3=dt \wedge dx_1 \wedge dx_2. 
\end{align*}

Let $\omega' = dta+dx_1b+dx_2c+dx_3d$, $a,b,c,d \in \Bbbk$. Then
\begin{align*}
\sum_{i=1}^{4}\omega_{i}^{1}\pi_{\omega}(\bar{\omega}_i^{3}\wedge \omega') &\ =dt\pi_{\omega}(-aa_1^{-1}a_2^{-1}a_3^{-1}dx_1\wedge dx_2\wedge dx_3 \wedge dt)\\
    &\ \quad + dx_1\pi_{\omega}(bc_{1,3}^{-1}c_{1,2}^{-1}dt\wedge dx_2\wedge dx_3\wedge dx_1 ) \\
    &\ \quad + dx_2\pi_{\omega}(-cc_{2,3}^{-1}dt\wedge dx_1\wedge dx_3 \wedge dx_2)\\
    &\ \quad + dx_3\pi_{\omega}(ddt\wedge dx_1\wedge dx_2 \wedge dx_3)\\
    &\ = dta + dx_1b + dx_2c + dx_3d \\
    &\ = \omega'.
\end{align*}

On the other hand, if 
\[
\omega'' = dt\wedge dx_1a+dt\wedge dx_2 b+dt\wedge dx_3 c+dx_1\wedge dx_2d+dx_1\wedge dx_3e+dx_2\wedge dx_3f 
\]

with $a,b,c,d,e,f \in \Bbbk$, it yields that 
\begin{align*}
\sum_{i=1}^{6}\omega_{i}^{2}\pi_{\omega}(\bar{\omega}_i^{2}\wedge \omega'') = &\ dt\wedge dx_1\pi_{\omega}(aa_2^{-1}a_{3}^{-1}c_{1,3}^{-1}c_{1,2}^{-1}dx_2\wedge dx_3\wedge dt\wedge dx_1) \\
    &\ + dt\wedge dx_2\pi_{\omega}(-ba_1^{-1}a_3^{-1}c_{2,3}^{-1}dx_1\wedge dx_3\wedge dt\wedge dx_2) \\
    &\ +dt\wedge dx_3\pi_{\omega}(ca_1^{-1}a_2^{-1}dx_1\wedge dx_2\wedge dt\wedge dx_3) \\
    &\ +dx_1\wedge dx_2\pi_{\omega}(dc_{1,3}^{-1}c_{2,3}^{-1}dt\wedge dx_3\wedge dx_1\wedge dx_2) \\
    &\ +dx_1\wedge dx_3\pi_{\omega}(-ec_{1,2}^{-1}dt\wedge dx_2\wedge dx_1\wedge dx_3) \\
    &\ +dx_2\wedge dx_3\pi_{\omega}(fdt\wedge dx_1\wedge dx_2\wedge dx_3) \\
    = &\ dt\wedge dx_1a+dt\wedge dx_2 b+dt\wedge dx_3 c \\
    &\ + dx_1\wedge dx_2d+dx_1\wedge dx_3e+dx_2\wedge dx_3f \\
    = &\ \omega''.
\end{align*}

Finally, let 
\[
\omega''' = dt\wedge dx_1\wedge dx_2 a+dt\wedge dx_2 \wedge dx_3 b+dx_1\wedge dx_2 \wedge dx_3 c+dt\wedge dx_1 \wedge dx_3 d,
\]

with $a, b, c, d \in \Bbbk$. Since that 
\begin{align*}
\sum_{i=1}^{3}\omega_{i}^{3}\pi_{\omega}(\bar{\omega}_i^{1}\wedge \omega''') = &\ dt\wedge dx_1\wedge dx_2 \pi_{\omega}(-aa_{3}^{-1}c_{1,3}^{-1}c_{2,3}^{-1}dx_3\wedge dt\wedge dx_1\wedge dx_2) \\
    &\ + dt\wedge dx_2 \wedge dx_3\pi_{\omega}(-ba_1^{-1}dx_1\wedge dt\wedge dx_2 \wedge dx_3 )\\
    &\ +dx_1\wedge dx_2 \wedge dx_3\pi_{\omega}(cdt\wedge dx_1\wedge dx_2 \wedge dx_3) \\
    &\ +dt\wedge dx_1 \wedge dx_3\pi_{\omega}(dc_{1,2}^{-1}a_2^{-1}dx_2\wedge dt\wedge dx_1 \wedge dx_2) \\
    = &\ dt\wedge dx_1\wedge dx_2 a+dt\wedge dx_2 \wedge dx_3 b+dx_1\wedge dx_2 \wedge dx_3 c \\
    &\ +dt\wedge dx_1 \wedge dx_3 d = \omega''', 
\end{align*}

we conclude that $\sigma(\Bbbk[t])\langle x_1, x_2, x_3\rangle$ is differentially smooth.
\end{proof}

\subsection{SPBW extensions in \texorpdfstring{$n$}{Lg} indeterminates}\label{SPBWTMGeneralCase}

The noncommutative differential geometry of SPBW extensions of the form $\sigma(\Bbbk[t]) \langle x_1, \ldots, x_n\rangle$ satisfying the defining relations
\begin{align*}
    x_ir(t) = &\ \sigma_i(r(t)) x_i + \delta_i(r(t)), \quad {\rm and} \\
    x_j x_i = &\ c_{i, j}(t)x_i x_j + q_{i,j}^{(0)}(t) + \sum_{k=1}^{n}q_{i,j}^{(k)}(t)x_k, 
\end{align*}

In this case, we consider $\sigma_i(t)=a_it+b_i$, for $a_i, b_i, \in \Bbbk$,  and $a_i\not=0$, $1\leq i \leq n$,  and the derivations $\delta_i$ are motivated by (\ref{deltap}), that is,
\begin{equation}
    \delta_i(f)=\frac{f\left(\sigma_{i}(t)\right)-f(t)}{\sigma_{i}(t)-t}p_i(t), \text{ for } 1 \leq i \leq n, 
\end{equation}

where $p_i(t) \in\Bbbk[t]$, $1 \leq i \leq n$. The relations between $t, x_1, \ldots, x_n$ can be expressed as
\begin{align}
    x_it = &\ a_itx_i+b_ix_i+p_i(t), \quad {\rm and} \label{twithxi} \\
    x_j x_i = &\ c_{i, j}x_i x_j + q_{i,j}^{(0)} + \sum_{k=1}^{n}q_{i,j}^{(k)}x_k, \quad\ {\rm for}\ 1 \leq i,j \leq n. \label{xiandxj} 
\end{align}

\begin{lemma}\label{autoSkewPBWdimn} 
Let
\begin{align}
   \nu_t(t) = t,\quad &\ \nu_t(x_i) = a_ix_i + p_i'(t), \quad {\rm for}\ 1 \leq i \leq n \label{Autoskewn1} \\ \nu_{x_i}(t) = \sigma_{i}^{-1}(t),\quad &\ \nu_{x_i}(x_i) = x_i, \quad {\rm for}\ 1 \leq i \leq n \label{Autoskewn2} \\ 
   \nu_{x_i}(x_j) = c_{i,j}x_j+q_{i,j}^{(i)}, \quad {\rm for}\ i < j \quad {\rm and} \quad &\ \nu_{x_i}(x_j) = c_{j,i}^{-1}x_j-c_{j,i}^{-1}q_{j,i}^{(i)},  \quad {\rm for}\ i > j, \label{Autoskewn3}
\end{align}

where $p_i'(t)$ are the $t$-derivatives of $p_i(t)$ for $1 \leq i \leq n$, and $c_{i,j}, q_{i,j}^{(k)} \in \Bbbk$, $c_{i,j}\not =0$, for all $1\leq i,j,k \leq n$.
\begin{enumerate}
\item [\rm (1)] Leibniz's rule holds in the cases listed in Table \ref{Leibnizsruletwoindeterminates(3)}.
\begin{table}[h]
\caption{Leibniz's rule}
\label{Leibnizsruletwoindeterminates(3)}
\begin{center}
\resizebox{12.7cm}{!}{
\setlength\extrarowheight{6pt}
\begin{tabular}{ |c|c|c|c| } 
\hline
{\rm Case} & {\rm Possibilities for} $a_i$, $b_i$, $i=1, 2, 3$ & {\em Polynomials} $p_i(t)$, $i=1, 2, 3$ & {\em Restrictions} \\
\hline
\multirow{2}{*}{{\rm (a)}} & \multirow{2}{*}{$a_i=1$, $b_i=0$, \text{ for all } $1\leq i \leq n$} & $p_i(t)=0$ for all $1 \leq i \leq n$ & $q_{i,j}^{(k)}=0$,  $c_{i,j}\in\Bbbk^{\ast}$ for all $1\leq i,j \leq n$, $k \geq 0$ \\ \cline{3-4}
&  & $p_i(t)\in\Bbbk[t], \text{ for all } 1\leq i \leq n$ & $q_{i,j}^{(k)}=0$,  $c_{i,j}=1$, $q_{i,j}^{(0)} \in\Bbbk$, \text{ for all } $1\leq i,j \leq n$, $k>0$ \\
\hline
\multirow{2}{*}{{\rm (b)}} & \multirow{2}{*}{$a_i=1$, for all $1\leq i \leq n$,  $b_l \not=0$, for some $1\leq l \leq n$} & $p_i(t)=p_i$, $p_i\in\Bbbk$, for all $1\leq i \leq n$ & $c_{i,j}=1$, $q_{i,j}^{(k)}=0$, $q_{i,j}^{(0)}\in\Bbbk$, \text{ for all } $1\leq i,j \leq n$, $k>0$ \\ \cline{3-4}
&  & $p_i(t)=0$, for all $1\leq i \leq n$ & $q_{i,j}^{(k)}=0$,  $c_{i,j}\in\Bbbk^{\ast}$, \text{ for all } $1\leq i,j \leq n$, $k\geq 0$ \\
\hline
\rm (c) & $a_r\not=1$, $a_s=1$, $b_s=0$, for $r\in S\subsetneq \{1,\ldots, n\}$ and $s\in S^c$& $p_s(t)=0$, $p_r(t)=p_r\left(t+\frac{b_r}{a_r-1}\right)$, for $r\in S\subsetneq \{1, \ldots, n\}$ and $s\in S^c$, $p_r\in \Bbbk$ & $q_{i,j}^{(k)}=0$,  $c_{i,j}=1$, \text{ for all } $1\leq i,j\leq n$, $k \geq 0$ \\ 
\hline
\rm (d) & $a_i\not =1$, $b_i=0$, \text{ for all } $1\leq i \leq n$  & $p_i(t)=p_it$, $p_i\in\Bbbk$, \text{ for all } $1\leq i \leq n$ & $q_{i,j}^{(k)}=0$,  $c_{i,j}=1$, \text{ for all } $1\leq i,j \leq n$, $k \geq 0$ \\   
\hline
\end{tabular}
}
\end{center}
\end{table}
    
The symbols defined by {\rm (}\ref{Autoskewn1}{\rm )}, {\rm (}\ref{Autoskewn2}{\rm )} and {\rm (}\ref{Autoskewn3}{\rm )} simultaneously extend to algebra automorphisms $\nu_t, \nu_{x_i}$, $1 \leq i \leq n$ of $\sigma(\Bbbk[t])\langle x_1,\ldots, x_n\rangle$ in cases {\rm (a)} - {\rm (d)}.
    
\item [\rm (2)] In cases {\rm (a)} - {\rm (d)}, we have that
\begin{equation}\label{Eq1skewn}
   \nu_t \circ \nu_{x_i} = \nu_{x_i} \circ \nu_t \quad {\rm and} \quad \nu_{x_i} \circ \nu_{x_j} = \nu_{x_j} \circ \nu_{x_i}, \quad \text{for}\ 1\leq i,j \leq n.
\end{equation}
\end{enumerate}
\end{lemma}

\begin{theorem}\label{smoothPBWn}
If a SPBW extension $\sigma(\Bbbk[t])\langle x_1,\ldots, x_n\rangle$ satisfies one of the conditions {\rm (a)}-{\rm (d)} in Lemma \ref{autoSkewPBWdimn}, then it is differentially smooth.
\end{theorem}
\begin{proof}
Since $\sigma(\Bbbk[t])\langle x_1,\ldots, x_n\rangle$ has Gelfand-Kirillov dimension $n+1$, we can construct an $n+1$-dimensional integrable. Consider $\Omega^{1}(\sigma(\Bbbk[t])\langle x_1,\ldots, x_n\rangle)$ a free right $\sigma(\Bbbk[t])\langle x_1,\ldots, x_n\rangle$-module of rank $n+1$ with generators $dt$, $dx_1, \ldots, dx_n$. Define a left $\sigma(\Bbbk[t])\langle x_1,\ldots, x_n\rangle$-module structure by
\begin{equation}\label{relrightmodn}
        adt = dt \nu_t(a) \quad {\rm and} \quad adx_i = dx_i\nu_{x_i}(a),  \ {\rm for\ all}\ 1\leq i \leq n, \ a\in \sigma(\Bbbk[t])\langle x_1,\ldots, x_n\rangle,
\end{equation}
    
where $\nu_t$ and $\nu_{x_i}$ with $i = 1, \dotsc, n$ are the algebra automorphisms established in Lemma \ref{autoSkewPBWdimn}. The relations in $\Omega^{1}(\sigma(\Bbbk[t])\langle x_1,\ldots, x_n\rangle)$ are given by
\begin{align}
    tdt = &\ dt t, & tdx_i = &\ dx_ia_i^{-1}(t-b_i), \quad \text{for all}\ 1\leq i \leq n, \label{rel1n}  \\
    x_idx_i = &\ dx_ix_i, & x_idt = &\ dt(a_ix_i+p_i'(t)), \quad \text{for all}\ 1\leq i \leq n,  \label{rel2n}  
\end{align}

and 
\begin{align}
x_idx_j = &\ dx_j(c_{i,j}^{-1}x_i-c_{i,j}^{-1}q_{i,j}^{(j)}), \quad \text{for}\ i < j, \quad {\rm and} \\
x_idx_j = &\ dx_j(c_{j,i}x_i + q_{j,i}^{(j)}), \quad \text{for}\ i > j. \label{rel3n} 
\end{align}

We want to extend the assignments $t\mapsto dt$, $x_i\mapsto dx_i$, $1\leq i \leq n$ to a map 
\[
d: \sigma(\Bbbk[t])\langle x_1,\ldots, x_n\rangle \to \Omega^{1}(\sigma(\Bbbk[t])\langle x_1,\ldots, x_n\rangle)
\]

satisfying the Leibniz's rule, so we need to impose the compatibility between this rule and the non-trivial relations {\rm (}\ref{twithxi}{\rm )} and {\rm (}\ref{xiandxj}{\rm )}. In this way, we have that
\begin{align*}
dx_it+x_idt &\ = a_idtx_i+a_itdx_i+b_idx_i+dp_i(t), \quad \text{for}\ 1\leq i \leq n, \quad {\rm and}\\
dx_jx_i+x_jdx_i &\ = c_{i,j}dx_ix_j+c_{i,j}x_idx_j +\sum_{k=1}^{n}q_{i,j}^{(k)}dx_k, \quad \text{for}\ i<j.
\end{align*}
    
Note that in view of the equality $t dt = dt t$, which defines the usual commutative calculus on the polynomial ring $\Bbbk[t]$, we get that $dp_i(t) = dtp_i'(t)$ for $1 \leq i \leq n$.

Define $\Bbbk$-linear maps 
$$
\partial_t, \partial_{x_i}: \sigma(\Bbbk[t])\langle x_1,\ldots, x_n\rangle \rightarrow \sigma(\Bbbk[t])\langle x_1,\ldots, x_n\rangle, \quad i = 1, \dotsc, n,
$$ 

such that
\begin{align*}
    d(f)=dt\partial_t(f) + \sum_{i=1}^{n}dx_i\partial_{x_i}(f), \quad \text{for all}\ f \in \sigma(\Bbbk[t])\langle x_1,\ldots, x_n\rangle.
\end{align*}

These maps are well-defined since $dt$ and $dx_i$ $(1\leq i \leq n)$ are free generators of the right $\sigma(\Bbbk[t])\langle x_1,\ldots, x_n\rangle$-module $\Omega^1(\sigma(\Bbbk[t])\langle x_1,\ldots, x_n\rangle)$. Thus, $d(a)=0$ if and only if $\partial_t(a)=\partial_{x_i}(a)=0$ for $1 \leq i \leq n$. Using relations appearing in {\rm (}\ref{relrightmodn}{\rm )} and the definitions of the maps $\nu_t$ and $\nu_{x_i}$ $(1 \leq i \leq n)$, we obtain that
\begin{align}
    \partial_t(t^kx_1^{l_1} \dotsb x_n^{l_n}) = &\ kt^{k-1}x_1^{l_1}\cdots x_n^{l_n}, \quad {\rm and} \\
    \partial_{x_i}(t^kx_1^{l_1}\cdots x_n^{l_n}) = &\ l_{i}a_i^{-k}(t-b_i)^k\prod_{s=1}^{i-1}c_{s,i}^{-l_s}(x_s-q_{s,i}^{(i)})^{l_s} x_i^{l_i-1}x_{i+1}^{l_{i+1}}\cdots x_n^{l_n}. \notag
\end{align}

Hence, $d(a)=0$ if and only if $a$ is a scalar multiple of the identity. This fact shows that $(\Omega(\sigma(\Bbbk[t])\langle x_1,\ldots, x_n\rangle,d))$ is connected, where 
$$
\Omega (\sigma(\Bbbk[t])\langle x_1,\ldots, x_n\rangle) = \bigoplus_{i=0}^{n+1}\Omega^i (\sigma(\Bbbk[t])\langle x_1,\ldots, x_n\rangle).
$$

The universal extension of $d$ to higher forms compatible with {\rm (}\ref{rel1n}{\rm )}, {\rm (}\ref{rel2n}{\rm )} and {\rm (}\ref{rel3n}{\rm )} gives the following rules for $\Omega^l (\sigma(\Bbbk[t])\langle x_1,\ldots, x_n\rangle)$ $(2\leq l \leq n)$: 
\begin{align}
    dx_{q(1)}\wedge \dotsb \wedge dx_{q(s)}\wedge dt \wedge dx_{q(s+1)}\wedge \dotsb \wedge dx_{q(l)} = &\ (-1)^s\prod_{r=1}^{s} a_{q(r)}^{-1} dt \wedge \bigwedge_{\substack{k=1,\\k \ne s_1}}^{l}dx_{q(k)}, \label{relnos1}  \\
    \bigwedge_{k=1}^{l}dx_{q(k)} = &\ (-1)^{\sharp}\prod_{r,s\in P}c_{r,s}^{-1}\bigwedge_{k=1}^ldx_{p(k)}, 
\end{align}

where $s_1 \in \{1,\ldots,l\}$ do not appear in Relation {\rm (}\ref{relnos1}{\rm )}, $q:\{1,\ldots,l\}\rightarrow \{1,\ldots,n\}$ is an injective map, $p:\{1,\ldots,l\}\rightarrow \text{Im}(q)$ is an increasing injective map and $\sharp$ is the number of $2$-permutations needed to transform $q$ into $p$, and $P := \{(s, t) \in \{1, \ldots, l\} \times \{1, \ldots, l\} \mid q(s) >  q(t)\}$.

Since the automorphisms $\nu_{t}$, $\nu_{x_i}$, $1\leq i \leq n$ commute with each other, there are no additional relations to the previous ones, so we get that
\begin{align*}
\Omega^{n} (\sigma(\Bbbk[t])\langle x_1,\ldots, x_n\rangle) = &\ \left[\bigoplus_{r=2}^{n-1}dt\wedge dx_1\wedge \cdots dx_{r-1}\wedge dx_{r+1} \wedge \cdots \wedge dx_{n} \right. \\
&\ \oplus dt\wedge dx_2\wedge \cdots \wedge dx_{n} \oplus dt\wedge dx_1\wedge \cdots \wedge dx_{n-1}\\
&\ \left. \oplus dx_1\wedge \cdots \wedge dx_{n}\right]\sigma(\Bbbk[t])\langle x_1,\ldots, x_n\rangle.
\end{align*}

Now, since
\[
\Omega^{n+1} (\sigma(\Bbbk[t])\langle x_1,\ldots, x_n\rangle) = \omega\sigma(\Bbbk[t])\langle x_1,\ldots, x_n\rangle\cong \sigma(\Bbbk[t])\langle x_1,\ldots, x_n\rangle
\]

as a right and left $\sigma(\Bbbk[t])\langle x_1,\ldots, x_n\rangle$-module, with 
\[
\omega=dt\wedge dx_1 \wedge \cdots \wedge dx_n \quad {\rm and} \quad \nu_{\omega}=\nu_t\circ\nu_{x_1}\circ\cdots\circ\nu_{x_n},
\]

it follows that $\omega$ is a volume form of $\sigma(\Bbbk[t])\langle x_1,\ldots, x_n\rangle$. In order to make the calculations easier, we consider the following notation $t=x_0$, $c_{0,i}=a_i$ for $1\leq i \leq n$.

From Proposition \ref{BrzezinskiSitarz2017Lemmas2.6and2.7} (2) we get that $\omega$ is an integral form by setting
\begin{align*}
    \omega_i^j = &\ \bigwedge_{k=0}^{j-1}dx_{p_{i,j}(k)}, \text{ for } 1\leq i \leq \binom{n+1}{j}, \\
    \bar{\omega}_i^{n+1-j} = &\ (-1)^{\sharp_{i,j}}\prod_{r,s\in P_{i,j}}c_{r,s}^{-1}\bigwedge_{k=j}^{n}dx_{\bar{p}_{i,j}(k)}, \text{ for } 1\leq i \leq \binom{n+1}{j},
\end{align*}
for $1\leq j \leq n+1$ and where 
\begin{align*}
    p_{i,j}:\{0,\ldots,j-1\}\rightarrow &\ \{0,\ldots,n\}, \quad {\rm and} \\
\bar{p}_{i,j}:\{j,\ldots,n\}\rightarrow &\ (\text{Im}(p_{i,j}))^c
\end{align*}

(the symbol $\square^c$ denotes the complement of the set $\square$), are increasing injective maps, and $\sharp_{i,j}$ is the number of $2$-permutation needed to transform 
\[
\left\{\bar{p}_{i,j}(j),\ldots, \bar{p}_{i,j}(n), p_{i,j}(0), \ldots, p_{i,j}(j-1)\right\} \quad {\rm into\ the\ set} \quad \{0, \ldots, n\},
\]

and 
\[
P_{i,j} :=\{(s, t) \in \{0, \ldots, j-1\}\times\{j, \ldots, n\} \mid p_{i,j}(s)< \bar{p}_{i,j}(t)\}.
\]

Consider $\omega' \in \Omega^j(\sigma(\Bbbk[t])\langle x_1,\ldots, x_n\rangle)$, that is,  
\begin{align*}
\omega' =\sum_{i=1}^{\binom{n+1}{j}}\bigwedge_{k=0}^{j-1}dx_{p_{i,j}(k)}b_i, \quad {\rm with} \  b_i \in \Bbbk.
\end{align*}

Then
\begin{align*}
 \sum_{i=1}^{\binom{n+1}{j}}\omega_{i}^{j}\pi_{\omega}(\bar{\omega}_i^{n+1-j}\wedge \omega') = &\ \sum_{i=1}^{\binom{n+1}{j}}\left[\bigwedge_{k=0}^{j-1}dx_{p_i(k)}\right] \cdot  \pi_{\omega} \left[(-1)^{\sharp_{i,j}} \square^{*} \wedge \omega'\right] \\
 = &\ \displaystyle \sum_{i=1}^{\binom{n+1}{j}}\bigwedge_{k=0}^{j-1}dx_{p_{i,j}(k)}b_i =  \omega',
\end{align*}

where 
\begin{align*}
    \square^{*} := &\ \prod_{r,s\in P_{i,j}}c_{r,s}^{-1} \bigwedge_{k=j}^{n}dx_{\bar{p}_{i,j}(k)}.
\end{align*}

By Proposition \ref{BrzezinskiSitarz2017Lemmas2.6and2.7} (2), it follows that $\sigma(\Bbbk[t])\langle x_1,\dotsc, x_n\rangle$ is differentially smooth.
\end{proof}

\begin{remark} 
Note that there is no unique way to define $\omega_i^j$ and $\bar{\omega}_i^{n-j}$. Our way of defining them is because it is the simplest.
\end{remark}

\section{Differential smoothness of SPBW extensions over \texorpdfstring{$\Bbbk[t_1, t_2]$}{Lg}}\label{DICSPBWKt1t2}

Finally, we investigate the differential smoothness of bijective SPBW extensions over the commutative polynomial ring $\Bbbk[t_1, t_2]$.

${\rm Aut}(\Bbbk[t_1, t_2])$ are compositions of automorphisms of two types (see McKay and Wang \cite{McKayWang1988}, Shestakov and Umirbaev \cite{ShestakovUmirbaev2003} or Van den Essen \cite{VandenEssen2000} for more details):
\begin{itemize}
    \item First type: 
\begin{equation}
    t_1 \longmapsto a_{11}t_1+a_{12}t_2+a_{13} \quad {\rm and} \quad t_2 \longmapsto a_{21}t_1+a_{22}t_2+a_{23},
\end{equation}

where $a_{i,j}\in \Bbbk$ and $a_{11}a_{22}-a_{12}a_{21}\not = 0$.
    \item Second type:
    \begin{equation}
t_1 \longmapsto t_1, \quad t_2 \longmapsto t_2+h(t_1), 
    \end{equation}

where $h(t_1)\in\Bbbk[t_1]$.
\end{itemize}

With these facts in our hands, we proceed to study the differential smoothness of SPBW extension on two generators.

\subsection{SPBW extensions in two indeterminates}\label{SPBWTMTwoI1}

Let $\sigma(\Bbbk[t_1, t_2]) \langle x_1, x_2\rangle$. From Definition \ref{defpbwextension} we know that
\begin{align*}
    x_1r(t_1) = &\ \sigma_1(r(t_1)) x_1 + \delta_1(r(t_1)),\quad x_2 r(t_1) =  \sigma_2(r(t_1)) x_2 + \delta_2(r(t_1)), \\
    x_1r(t_2) = &\ \sigma_1(r(t_2)) x_1 + \delta_1(r(t_2)),\quad x_2 r(t_2) =  \sigma_2(r(t_2)) x_2 + \delta_2(r(t_2)), \quad {\rm and} \\
    x_2 x_1 = &\ c_{1, 2}(t_1,t_2)x_1 x_2 + q_{1,2}^{(0)}(t_1,t_2) + q_{1,2}^{(1)}(t_1,t_2)x_1 + q_{1,2}^{(2)}(t_1,t_2)x_2, 
\end{align*}

where the polynomials $r(t_1,t_2), c_{1, 2}(t_1,t_2), q_{1, 2}^{(0)}(t_1,t_2), q_{1, 2}^{(1)}(t_1,t_2), q_{1, 2}^{(2)}(t_1,t_2)$ belong to $\Bbbk[t_1, t_2]$, and $c_{1,2}(t_1, t_2)$ is a non-zero element.

Considering the notation above, we write $\sigma_1, \sigma_2 \in {\rm Aut}(\Bbbk[t_1, t_2])$ as follows:
\begin{align*}
    \sigma_1(t_1) = &\ a_{111}t_1+a_{112}t_2+b_{11}, \\
    \sigma_1(t_2) = &\ a_{121}t_1+a_{122}t_2+b_{12}, \\
    \sigma_2(t_1) = &\ a_{211}t_1+a_{212}t_2+b_{21}, \quad {\rm and} \\
    \sigma_2(t_2) = &\ a_{221}t_1+a_{222}t_2+b_{22}.
\end{align*}

As in Section \ref{SPBWTMTwoI}, the polynomials $p_1(t_1, t_2), p_2(t_1, t_2) \in \Bbbk[t_1, t_2]$ are considered in such a way that the following identities 
\begin{align}
    x_1t_1 = &\  a_{111}t_1x_1+a_{112}t_2x_1+b_{11}x_1+ p_1(t_1, t_2),  \notag \\
    x_2t_1  = &\ a_{211}t_1x_2+a_{212}t_2x_2+b_{21}x_2+ p_2(t_1,t_2),  \notag \\
    x_1t_2 = &\ a_{121}t_1x_1+a_{122}t_2x_1+b_{12}x_1+ p_1(t_1,t_2),  \label{DSdosdosSPBW} \\
    x_2t_2  = &\ a_{221}t_1x_2+a_{222}t_2x_2+b_{22}x_2+ p_2(t_1, t_2), \quad {\rm and}  \notag \\
    x_2 x_1 = &\ c_{1, 2}(t_1, t_2)x_1 x_2 + q_{1,2}^{(0)}(t_1, t_2) + q_{1,2}^{(1)}(t_1, t_2)x_1 + q_{1,2}^{(2)}(t_1, t_2)x_2, \notag
\end{align}

hold. Since the map $d:\sigma(\Bbbk[t_1, t_2]) \langle x_1, x_2\rangle\rightarrow \Omega^1(\sigma(\Bbbk[t_1, t_2]))\langle x_1, x_2\rangle$ must satisfy Leibniz's rule, it is straightforward to see that we need to guarantee the conditions
\begin{itemize}
    \item $c_{1,2}, q_{1,2}^{(0)}, q_{1,2}^{(1)}, q_{1,2}^{(2)} \in \Bbbk$, with $c_{1,2}$ non-zero.
    \item $p_1(t_1,t_2)=p_1$ and $p_2(t_1, t_2)=p_2$, where $p_1, p_2 \in \Bbbk$.
    \item $a_{221}=a_{212}=a_{112}=a_{121}=0$.
\end{itemize}
Indeed, the first four relations in {\rm (}\ref{DSdosdosSPBW}{\rm )} can be written as
\begin{equation}\label{relij}
    x_it_j=a_{ij1}t_1x_i+a_{ij2}t_2x_i+b_{ij}x_i+p_i(t_1, t_2), \quad \text{for}\ i,j \in \{1, 2\}.
\end{equation}

By applying $d$ to (\ref{relij}) we get that
\[
0 = -d(x_it_j)+d(a_{ij1}t_1x_i+a_{ij2}t_2x_i+b_{ij}x_i+p_i(t_1, t_2)). 
\]

Since $d$ is $\Bbbk$-linear, the Leibniz's rule implies that 
\begin{align*}
0 = &\ -dx_it_j-x_idt_j+a_{ij1}dt_1x_i+a_{ij1}t_1dx_i+a_{ij2}dt_2x_i+a_{ij2}t_2dx_i \\
    & \ +b_{ij}dx_i+d(p_i(t_1, t_2)). 
\end{align*}

By (\ref{BrzezinskiSitarz2017(2.2)}), the action of the module is written using the automorphisms $\nu_{t_1}$, $\nu_{t_2}$, $\nu_{x_1}$ and $\nu_{x_2}$, that is, 
\begin{align*}
0 = &\ -dx_it_j-dt_j\nu_{t_j}(x_i)+a_{ij1}dt_1x_i+a_{ij1}dx_i\nu_{x_i}(t_1)+a_{ij2}dt_2x_i+a_{ij2}dx_i\nu_{x_i}(t_2) \\
    & \ +b_{ij}dx_i+dt_1\frac{\partial p_i}{\partial t_1}+dt_2\frac{\partial p_i}{\partial t_2}.
\end{align*}

If we put together the terms that multiply the different differentials, then 
\begin{align*}
0 = &\ dx_i(-t_j+a_{ij1}\nu_{x_i}(t_1)+a_{ij2}\nu_{x_i}(t_2)+b_{ij})+dt_1\left(a_{ij1}x_i+\frac{\partial p_i}{\partial t_1}\right) \\
&\ dt_2\left(a_{ij2}x_i+\frac{\partial p_i}{\partial t_2}\right) -dt_j\nu_{t_j}(x_i).
\end{align*}

For $j=1$, we obtain the term
\begin{align*}
    a_{i12}x_i+\frac{\partial p_i}{\partial t_2} = 0,
\end{align*} 

whence $a_{i12} = 0$ and $\frac{\partial p_i}{\partial t_2} = 0$ for $i \in \{1, 2\}$.

Next, when $j=2$,
\begin{align*}
    a_{i21}x_i+\frac{\partial p_i}{\partial t_1} = 0. 
\end{align*}

Once more again, it follows that $a_{i21}=0$ and $\frac{\partial p_i}{\partial t_1} = 0$ for $i \in \{1, 2\}$.

Since the partial derivatives of $p_i$ are zero, we conclude that $p_i$ is a constant element for $i \in \{1, 2\}$.

Finally, by applying $d$ to the last equation in {\rm (}\ref{DSdosdosSPBW}{\rm )} we get that
\begin{align*}
    d(x_2 x_1) = &\ d(c_{1, 2}(t_1, t_2)x_1 x_2 + q_{1,2}^{(0)}(t_1, t_2) + q_{1,2}^{(1)}(t_1, t_2)x_1 + q_{1,2}^{(2)}(t_1, t_2)x_2) \\
    = &\ d(c_{1,2}(t_1, t_2))x_1x_2+c_{1, 2}(t_1, t_2)d(x_1x_2)+d(q_{1,2}^{(0)}(t_1, t_2)) \\
     &\ + d(q_{1,2}^{(1)}(t_1, t_2))x_1 + q_{1,2}^{(1)}(t_1, t_2)dx_1+d(q_{1,2}^{(2)}(t_1, t_2))x_2+q_{1,2}^{(2)}(t_1, t_2)dx_2 \\
     = &\ dt_1\frac{\partial c_{1,2}}{\partial t_1}x_1x_2+dt_2\frac{\partial c_{1,2}}{\partial t_2}x_1x_2+c_{1,2}(t_1, t_2)dx_1x_2+c_{1,2}(t_1, t_2)x_1dx_2 \\
     &\ + dt_1\frac{\partial q_{1,2}^{(0)}}{\partial t_1}+dt_2\frac{\partial q_{1,2}^{(0)}}{\partial t_2}+dt_1\frac{\partial q_{1,2}^{(1)}}{\partial t_1}x_1+dt_2\frac{\partial q_{1,2}^{(1)}}{\partial t_2}x_1+q_{1,2}^{(1)}(t_1, t_2)dx_1 \\
     &\ +dt_1\frac{\partial q_{1,2}^{(2)}}{\partial t_1}x_2+dt_2\frac{\partial q_{1,2}^{(2)}}{\partial t_2}x_2+q_{1,2}^{(2)}(t_1, t_2)dx_2.
\end{align*}

The expression (\ref{BrzezinskiSitarz2017(2.2)}) implies that the action of the module is written using the automorphisms $\nu_{t_1}$, $\nu_{t_2}$, $\nu_{x_1}$ and $\nu_{x_2}$ as follows:
\begin{align*}
    0 & = -dx_2x_1-dx_1\nu_{x_1}(x_2)+dt_1\frac{\partial c_{1,2}}{\partial t_1}x_1x_2+dt_2\frac{\partial c_{1,2}}{\partial t_2}x_1x_2 
     +dx_1\nu_{x_1}(c_{1,2}(t_1, t_2))x_2 \\
     &\ \ \ + dx_2\nu_{x_2}(c_{1,2}(t_1, t_2))\nu_{x_2}(x_1) + dt_1\frac{\partial q_{1,2}^{(0)}}{\partial t_1}+dt_2\frac{\partial q_{1,2}^{(0)}}{\partial t_2}+dt_1\frac{\partial q_{1,2}^{(1)}}{\partial t_1}x_1+dt_2\frac{\partial q_{1,2}^{(1)}}{\partial t_2}x_1 \\
     &\ \ \ + dx_1\nu_{x_1}(q_{1,2}^{(1)}(t_1, t_2))+dt_1\frac{\partial q_{1,2}^{(2)}}{\partial t_1}x_2+dt_2\frac{\partial q_{1,2}^{(2)}}{\partial t_2}x_2+dx_2\nu_{x_2}(q_{1,2}^{(2)}(t_1, t_2)).
\end{align*}

In this way, 
\begin{align*}
    0 = &\ dt_1\left(\frac{\partial c_{1,2}}{\partial t_1}x_1x_2+\frac{\partial q_{1,2}^{(0)}}{\partial t_1}+\frac{\partial q_{1,2}^{(1)}}{\partial t_1}x_1+\frac{\partial q_{1,2}^{(2)}}{\partial t_1}x_2 \right) \\
    &\ + dt_2\left(\frac{\partial c_{1,2}}{\partial t_2}x_1x_2+\frac{\partial q_{1,2}^{(0)}}{\partial t_2}+\frac{\partial q_{1,2}^{(1)}}{\partial t_2}x_1+\frac{\partial q_{1,2}^{(2)}}{\partial t_2}x_2 \right).
\end{align*}

From the reasoning above, it can be seen that all partial derivatives must be equal to zero. Equivalently, $c_{1,2}, q_{1,2}^{(0)}, q_{1,2}^{(1)}, q_{1,2}^{(2)} \in \Bbbk$, with $c_{1,2}$ a non-zero element of the field $\Bbbk$.

The five relations in (\ref{DSdosdosSPBW}) are reduced to
\begin{align}
    x_1t_1 = &\  a_{111}t_1x_1+b_{11}x_1+ p_1, \label{relPBW4.1} \\
    x_2t_1  = &\ a_{211}t_1x_2+b_{21}x_2+ p_2, \label{relPBW4.2} \\
    x_1t_2 = &\ a_{122}t_2x_1+b_{12}x_1+ p_1, \label{relPBW4.3} \\
    x_2t_2  = &\ a_{222}t_2x_2+b_{22}x_2+ p_2, \quad {\rm and} \label{relPBW4.4} \\
    x_2 x_1 = &\ c_{1, 2}x_1 x_2 + q_{1,2}^{(0)} + q_{1,2}^{(1)}x_1 + q_{1,2}^{(2)}x_2. \label{relPBW4.5}
\end{align}

All these facts allow us to formulate the following proposition.

\begin{proposition}\label{autoSkewPBW22} 
Let
\begin{align}
   \nu_{t_1}(t_1) = &\ t_1, & \nu_{t_1}(t_2) = &\ t_2, & \nu_{t_1}(x_1) = &\ a_{111}x_1,  \label{Autoskew4.6} \\ 
   \nu_{t_1}(x_2) = &\ a_{211}x_2, & \nu_{t_2}(t_1) = &\ t_1, & \nu_{t_2}(t_2) = &\ t_2,  \label{Autoskew4.7} \\
   \nu_{t_2}(x_1) = &\ a_{122}x_1, & \nu_{t_2}(x_2) = &\ a_{222}x_2, & \nu_{x_1}(t_1) = &\ a_{111}^{-1}(t_1-b_{11}),  \label{Autoskew4.8} \\
    \nu_{x_1}(t_2) = &\ a_{122}^{-1}(t_2-b_{12}), & \nu_{x_1}(x_1) = &\ x_1, &  \nu_{x_1}(x_2) = &\ c_{1,2}x_2+q_{1,2}^{(1)}, \label{Autoskew4.9} \\
    \nu_{x_2}(t_1) = &\ a_{211}^{-1}(t_1-b_{21}), & \nu_{x_2}(t_2) = &\ a_{222}^{-1}(t_2-b_{22}), & \nu_{x_2}(x_1) = &\ c_{1,2}^{-1}x_1-c_{1,2}^{-1}q_{1,2}^{(2)}, \label{Autoskew4.10} \\
    \nu_{x_2}(x_2) = &\ x_2. \label{Autoskew4.11}
\end{align}

Then:
\begin{enumerate}
\item [\rm (1)] Leibniz's rule holds in the cases listed in Table \ref{Leibnizsruletwoindeterminates(1)2}. The maps defined by {\rm (}\ref{Autoskew4.6}{\rm )}, {\rm (}\ref{Autoskew4.7}{\rm )}, {\rm (}\ref{Autoskew4.8}{\rm )}, {\rm (}\ref{Autoskew4.9}{\rm )}, {\rm (}\ref{Autoskew4.10}{\rm )} and {\rm (}\ref{Autoskew4.11}{\rm )} simultaneously extend to algebra automorphisms $\nu_{t_1}, \nu_{t_2}, \nu _{x_1}, \nu_{x_2}$ of $\sigma(\Bbbk[t_1, t_2])\langle x_1, x_2\rangle$ in cases {\rm (a) - (p)}.
\end{enumerate}

\begin{table}[h]
\caption{Leibniz's rule}
\label{Leibnizsruletwoindeterminates(1)2}
\begin{center}
\resizebox{12.5cm}{!}{
\setlength\extrarowheight{6pt}
\begin{tabular}{ |c|c|c|c| } 
\hline
{\rm Case} & {\rm Possibilities for} $a_{111}$, $a_{122}$, $a_{211}$, $a_{222}$ & {\em Polynomials} $p_1$ {\em and} $p_2$ & {\em Restrictions} \\
\hline
\multirow{7}{*}{{\rm (a)}} & \multirow{7}{*}{$a_{111}=1$, $a_{122}=1$, $a_{211}=1$, $a_{222}=1$} & \multirow{3}{*}{$p_1, p_2 \in \Bbbk$} & $c_{1,2}=1$, $q_{1,2}^{(1)}=q_{1,2}^{(2)}=0$,  $b_{11}, b_{12}, b_{21}, b_{22}, q_{1,2}^{(0)} \in \Bbbk$ \\ \cline{4-4}
&  &  & $c_{1,2}=1$, $q_{1,2}^{(1)}=0, q_{1,2}^{(2)}\not =0$,  $b_{11}=b_{12}=0$, $b_{21}, b_{22}, q_{1,2}^{(0)} \in \Bbbk$ \\ \cline{4-4}
&  &  & $c_{1,2}=1$, $q_{1,2}^{(1)}\not =0, q_{1,2}^{(2)} =0$,  $b_{21}=b_{22}=0$, $b_{11}, b_{12}, q_{1,2}^{(0)} \in \Bbbk$ \\ \cline{3-4}
& & $p_1=p_2=0$ & $c_{1,2}\not = 1$, $q_{1,2}^{(0)}=q_{1,2}^{(1)}=q_{1,2}^{(2)}=0$, $b_{11}, b_{12}, b_{21}, b_{22} \in \Bbbk$ \\ \cline{3-4}
& & $p_1=\frac{b_{11}q_{1,2}^{(2)}}{c_{1,2}-1}, p_2=0$ & $c_{1,2}\not = 1$, $q_{1,2}^{(0)}=q_{1,2}^{(1)}=0$ $q_{1,2}^{(2)}\not =0$, $b_{11}=b_{12}$, $ b_{11}, b_{21}, b_{22} \in \Bbbk$ \\ \cline{3-4}
& & $p_1=0, p_2=\frac{b_{22}q_{1,2}^{(1)}}{c_{1,2}-1}$ & $c_{1,2}\not = 1$, $q_{1,2}^{(0)}=q_{1,2}^{(2)}=0$ $q_{1,2}^{(1)}\not =0$, $b_{22}=b_{21}$, $ b_{11}, b_{12}, b_{22} \in \Bbbk$ \\ \cline{3-4}
& & $p_1=\frac{b_{11}q_{1,2}^{(2)}}{c_{1,2}-1}, p_2=\frac{b_{22}q_{1,2}^{(1)}}{c_{1,2}-1}$ & $c_{1,2}\not = 1$, $q_{1,2}^{(0)}=\frac{q_{1,2}^{(1)}q_{1,2}^{(2)}}{c_{1,2}-1}$ $q_{1,2}^{(1)}\not=0$ $q_{1,2}^{(2)}\not =0$, $b_{11}=b_{12}$, $b_{22}=b_{21}$, $ b_{11}, b_{22} \in \Bbbk$ \\
\hline
\multirow{2}{*}{{\rm (b)}} & \multirow{2}{*}{$a_{111}\not= 1$, $a_{122}=1$, $a_{211}=1$, $a_{222}=1$}  & $p_1=p_2=0$ & $b_{21}=0$, $q_{1,2}^{(0)}=q_{1,2}^{(2)}=0$, $c_{1,2}, b_{11}, b_{12}\in\Bbbk$ {\rm with} $b_{22}=0$ and $q_{1,2}^{(1)} \in \Bbbk$ or $q_{1,2}^{(1)}=0$ and $b_{22}\in\Bbbk$ \\ \cline{3-4}
&  & $p_1=0$, $p_2=\frac{b_{22}q_{1,2}^{(1)}}{c_{1,2}-1}$ & $b_{21}=0$, $q_{1,2}^{(0)}=q_{1,2}^{(2)}=0$, $c_{1,2}=a_{111}$, $b_{11}$, $b_{12}$, $b_{22}, q_{1,2}^{(1)}\in\Bbbk$ \\   
\hline
\multirow{2}{*}{{\rm (c)}} & \multirow{2}{*}{$a_{111}= 1$, $a_{122}\not=1$, $a_{211}=1$, $a_{222}=1$}  & $p_1=p_2=0$ & $b_{22}=0$, $q_{1,2}^{(0)}=q_{1,2}^{(2)}=0$, $c_{1,2}, b_{11}, b_{12}\in\Bbbk$ {\rm with} $b_{21}=0$ and $q_{1,2}^{(1)} \in \Bbbk$ or $q_{1,2}^{(1)}=0$ and $b_{21}\in\Bbbk$ \\ \cline{3-4}
&  & $p_1=0$, $p_2=\frac{b_{21}q_{1,2}^{(1)}}{c_{1,2}-1}$ & $b_{22}=0$, $q_{1,2}^{(0)}=q_{1,2}^{(2)}=0$, $c_{1,2}=a_{122}$, $b_{11}$, $b_{12}$, $b_{21}, q_{1,2}^{(1)}\in\Bbbk$ \\   
\hline
\multirow{2}{*}{{\rm (d)}} & \multirow{2}{*}{$a_{111}= 1$, $a_{122}=1$, $a_{211}\not=1$, $a_{222}=1$}  & $p_1=p_2=0$ & $b_{11}=0$, $q_{1,2}^{(0)}=q_{1,2}^{(1)}=0$, $c_{1,2}, b_{22}, b_{21}\in\Bbbk$ {\rm with} $b_{12}=0$ and $q_{1,2}^{(2)} \in \Bbbk$ or $q_{1,2}^{(2)}=0$ and $b_{12}\in\Bbbk$ \\ \cline{3-4}
&  & $p_1=\frac{b_{12}q_{1,2}^{(2)}}{c_{1,2}-1}$, $p_2=0$ & $b_{11}=0$, $q_{1,2}^{(0)}=q_{1,2}^{(1)}=0$, $c_{1,2}=a_{211}$, $b_{12}, b_{21}, b_{22}, q_{1,2}^{(2)}\in\Bbbk$ \\   
\hline
\multirow{2}{*}{{\rm (e)}} & \multirow{2}{*}{$a_{111}= 1$, $a_{122}=1$, $a_{211}=1$, $a_{222}\not=1$}  & $p_1=p_2=0$ & $b_{12}=0$, $q_{1,2}^{(0)}=q_{1,2}^{(1)}=0$, $c_{1,2}, b_{22}, b_{21}\in\Bbbk$ {\rm with} $b_{11}=0$ and $q_{1,2}^{(2)} \in \Bbbk$ or $q_{1,2}^{(2)}=0$ and $b_{11}\in\Bbbk$ \\ \cline{3-4}
&  & $p_1=\frac{b_{11}q_{1,2}^{(2)}}{c_{1,2}-1}$, $p_2=0$ & $b_{12}=0$, $q_{1,2}^{(0)}=q_{1,2}^{(1)}=0$, $c_{1,2}=a_{211}$, $b_{11}, b_{21}, b_{22}, q_{1,2}^{(2)}\in\Bbbk$ \\   
\hline
\multirow{2}{*}{{\rm (f)}} & \multirow{2}{*}{$a_{111}\not =1$, $a_{122}\not=1$, $a_{211}=1$, $a_{222}=1$} & $p_1=p_2=0$ & $b_{21}=b_{22}=0$, $q_{1,2}^{(0)}=q_{1,2}^{(2)}=0$, $c_{1,2}, q_{1,2}^{(1)}, b_{11}, b_{12}\in\Bbbk$ \\ \cline{3-4}
&  & $p_1=0$, $p_2\in \Bbbk$ & $b_{21}=b_{22}=0$, $q_{1,2}^{(0)}=q_{1,2}^{(2)}=0$, $a_{111}=a_{122}$ $c_{1,2}=a_{111}$, $q_{1,2}^{(1)}, b_{11}, b_{12}\in\Bbbk$ \\   
\hline
\multirow{2}{*}{{\rm (g)}} & \multirow{2}{*}{$a_{111}\not =1$, $a_{122}=1$, $a_{211}\not=1$, $a_{222}=1$} & \multirow{2}{*}{$p_1=p_2=0$} & $q_{1,2}^{(0)}=q_{1,2}^{(1)}=q_{1,2}^{(2)}=0$, $b_{21}=\frac{b_{11}(a_{211}-1)}{a_{111}-1}$ $c_{1,2}, b_{11}, b_{12}, b_{22}\in\Bbbk$ \\ \cline{4-4}
&  &  & $a_{111}=a_{211}^{-1}$ $q_{1,2}^{(1)}=q_{1,2}^{(2)}=0$, $b_{21}=-a_{211}^{-1}b_{11}$, $b_{11}, b_{12}, b_{22}\in\Bbbk$ {\rm with} $q_{1,2}^{(0)}=0$ and $c_{1,2}\in\Bbbk$ or $q_{1,2}^{(0)}\in\Bbbk$ and $c_{1,2}=0$ \\   
\hline
{\rm (h)} & $a_{111}\not =1$, $a_{122}=1$, $a_{211}=1$, $a_{222}\not=1$ & $p_1=p_2=0$ & $q_{1,2}^{(0)}=q_{1,2}^{(1)}=q_{1,2}^{(2)}=0$, $b_{21}=b_{12}=0$, $c_{1,2}, b_{11}, b_{22}\in\Bbbk$ \\
\hline
{\rm (i)} & $a_{111} =1$, $a_{122}\not=1$, $a_{211}\not=1$, $a_{222}=1$ & $p_1=p_2=0$ & $q_{1,2}^{(0)}=q_{1,2}^{(1)}=q_{1,2}^{(2)}=0$, $b_{11}=b_{22}=0$, $c_{1,2}, b_{12}, b_{21}\in\Bbbk$ \\
\hline
\multirow{2}{*}{{\rm (j)}} & \multirow{2}{*}{$a_{111} =1$, $a_{122}\not=1$, $a_{211}=1$, $a_{222}\not=1$} & \multirow{2}{*}{$p_1=p_2=0$} & $q_{1,2}^{(0)}=q_{1,2}^{(1)}=q_{1,2}^{(2)}=0$, $b_{12}=\frac{b_{22}(a_{122}-1)}{a_{222}-1}$ $c_{1,2}, b_{11}, b_{21}, b_{22}\in\Bbbk$ \\ \cline{4-4}
&  &  & $a_{222}=a_{122}^{-1}$ $q_{1,2}^{(1)}=q_{1,2}^{(2)}=0$, $b_{12}=-a_{122}^{-1}b_{22}$, $b_{11}, b_{21}, b_{22}\in\Bbbk$ with $q_{1,2}^{(0)}=0$ and $c_{1,2}\in\Bbbk$ or $q_{1,2}^{(0)}\in\Bbbk$ and $c_{1,2}=0$ \\   
\hline
\multirow{2}{*}{{\rm (k)}} & \multirow{2}{*}{$a_{111} =1$, $a_{122}=1$, $a_{211}\not=1$, $a_{222}\not=1$} & $p_1=p_2=0$ & $b_{11}=b_{12}=0$, $q_{1,2}^{(0)}=q_{1,2}^{(1)}=0$, $c_{1,2}, q_{1,2}^{(2)}, b_{21}, b_{22}\in\Bbbk$ \\ \cline{3-4}
&  & $p_1\in \Bbbk$, $p_2=0$ & $b_{11}=b_{12}=0$, $q_{1,2}^{(0)}=q_{1,2}^{(1)}=0$, $a_{211}=a_{222}$ $c_{1,2}=a_{222}$, $q_{1,2}^{(2)}, b_{21}, b_{22}\in\Bbbk$ \\   
\hline
{\rm (l)} & $a_{111} =1$, $a_{122}\not=1$, $a_{211}\not=1$, $a_{222}\not=1$ & \multirow{4}{*}{$p_1=p_2=0$} &  $q_{1,2}^{(0)}=q_{1,2}^{(1)}=q_{1,2}^{(2)}=0$, $b_{11}=0$, $b_{12}=\frac{b_{22}(a_{122}-1)}{a_{222}-1}$, $c_{1,2},  b_{21}, b_{22}\in\Bbbk$ \\ \cline{4-4} \cline{1-2}
{\rm (m)} & $a_{111} \not=1$, $a_{122}=1$, $a_{211}\not=1$, $a_{222}\not=1$ &  & $q_{1,2}^{(0)}=q_{1,2}^{(1)}=q_{1,2}^{(2)}=0$, $b_{12}=0$, $b_{21}=\frac{b_{11}(a_{211}-1)}{a_{111}-1}$, $c_{1,2},  b_{11}, b_{22}\in\Bbbk$ \\ \cline{1-2}  \cline{4-4}
{\rm (n)} & $a_{111} \not=1$, $a_{122}\not=1$, $a_{211}=1$, $a_{222}\not=1$ &  & $q_{1,2}^{(0)}=q_{1,2}^{(1)}=q_{1,2}^{(2)}=0$, $b_{21}=0$, $b_{12}=\frac{b_{22}(a_{122}-1)}{a_{222}-1}$, $c_{1,2},  b_{11}, b_{22}\in\Bbbk$ \\   \cline{4-4} \cline{1-2}
{\rm (o)} & $a_{111} \not=1$, $a_{122}\not=1$, $a_{211}\not=1$, $a_{222}=1$ &  & $q_{1,2}^{(0)}=q_{1,2}^{(1)}=q_{1,2}^{(2)}=0$, $b_{22}=0$, $b_{21}=\frac{b_{11}(a_{211}-1)}{a_{111}-1}$, $c_{1,2},  b_{12}, b_{11}\in\Bbbk$ \\ 
\hline
\multirow{3}{*}{{\rm (p)}} & \multirow{3}{*}{$a_{111} \not=1$, $a_{122}\not=1$, $a_{211}\not=1$, $a_{222}\not=1$} & \multirow{3}{*}{$p_1=p_2=0$} & $q_{1,2}^{(0)}=q_{1,2}^{(1)}=q_{1,2}^{(2)}=0$, $b_{12}=\frac{b_{22}(a_{122}-1)}{a_{222}-1}$, $b_{21}=\frac{b_{11}(a_{211}-1)}{a_{111}-1}$, $c_{1,2},  b_{11}, b_{22}\in\Bbbk$ \\ \cline{4-4}
 &  &  & $q_{1,2}^{(0)}=q_{1,2}^{(1)}=q_{1,2}^{(2)}=0$, $a_{211}=a_{111}^{-1}$, $b_{12}=\frac{b_{22}(a_{122}-1)}{a_{222}-1}$, $b_{21}=-a_{111}^{-1}b_{11}$, $c_{1,2},  b_{11}, b_{22}\in\Bbbk$ \\ \cline{4-4}
 &  &  & $q_{1,2}^{(1)}=q_{1,2}^{(2)}=0$, $a_{211}=a_{111}^{-1}$, $a_{122}=a_{222}^{-1}$, $b_{21}=-a_{111}^{-1}b_{11}$, $b_{12}=-a_{222}^{-1}b_{22}$, $b_{11}, b_{22}\in\Bbbk$ {\rm with} $c_{1,2}=1$ and $q_{1,2}^{(0)}\in\Bbbk$ or $q_{1,2}^{(0)}=0$ and $c_{1,2}\in\Bbbk$\\ 
\hline
\end{tabular}
}
\end{center}
\end{table}

\begin{enumerate}
\item [\rm (2)] In cases {\rm (a) - (p)}, we have that 
    \begin{equation}\label{Eq1skew2}
   \nu_{t_i} \circ \nu_{t_j} = \nu_{t_j} \circ \nu_{t_i}, \quad \nu_{t_i} \circ \nu_{x_j} = \nu_{x_j} \circ \nu_{t_i}, \quad \nu_{x_i} \circ \nu_{x_j} = \nu_{x_j} \circ \nu_{x_i}, \quad {\rm for}\ i,j =1, 2.
    \end{equation}
\end{enumerate}
\end{proposition}
\begin{proof}
For the first assertion, the map $\nu_{t_1}$ can be extended to an algebra homomorphism if and only if the definitions of $\nu_{t_1}(t_1)$, $\nu_{t_1}(t_2)$,  $\nu_{t_1}(x_1)$ and $\nu_{t_1}(x_2)$ respect relations {\rm (}\ref{relPBW4.1}{\rm )}, {\rm (}\ref{relPBW4.2}{\rm )}, {\rm (}\ref{relPBW4.3}{\rm )}, {\rm (}\ref{relPBW4.4}{\rm )} and {\rm (}\ref{relPBW4.5}{\rm )}, i.e.
\begin{align*}
   \nu_{t_1}(x_1)\nu_{t_1}(t_1)-\nu_{t_1}(a_{111}t_1+b_{11})\nu_{t_1}(x_1) = &\ \nu_{t_1}(p_1), \\
   \nu_{t_1}(x_2)\nu_{t_1}(t_1)-\nu_{t_1}(a_{211}t_1+b_{21})\nu_{t_1}(x_2) = &\ \nu_{t_1}(p_2), \\
   \nu_{t_1}(x_1)\nu_{t_1}(t_2)-\nu_{t_1}(a_{122}t_2+b_{12})\nu_{t_1}(x_1) = &\ \nu_{t_1}(p_1), \\
    \nu_{t_1}(x_2)\nu_{t_1}(t_2)-\nu_{t_1}(a_{222}t_2+b_{22})\nu_{t_1}(x_2) = &\ \nu_{t_1}(p_2), \quad {\rm and} \\
    \nu_{t_1}(x_2)\nu_{t_1}(x_1)-c_{1,2}\nu_{t_1}(x_1)\nu_{t_1}(x_2) = &\ q_{1,2}^{(0)}+q_{1,2}^{(1)}\nu_{t_1}(x_1)+q_{1,2}^{(2)}\nu_{t_1}(x_2).
\end{align*}

We obtain the equations given by
\begin{align}
p_1(a_{111}-1) = &\ 0, \notag \\
p_2(a_{211}-1) = &\ 0, \notag \\
q_{1,2}^{(0)}(a_{211}a_{111}-1) = &\ 0, \label{firstsecondrel2} \\
q_{1,2}^{(1)}(a_{211}-1) = &\ 0, \quad {\rm and} \\
q_{1,2}^{(2)}(a_{111}-1) = &\ 0. \notag 
\end{align}

Again, the map $\nu_{t_2}$ can be extended to an algebra homomorphism if and only if the definitions of $\nu_{t_2}(t_1)$, $\nu_{t_2}(t_2)$,  $\nu_{t_2}(x_1)$ and $\nu_{t_2}(x_2)$ respect relations {\rm (}\ref{relPBW4.1}{\rm )}, {\rm (}\ref{relPBW4.2}{\rm )}, {\rm (}\ref{relPBW4.3}{\rm )}, {\rm (}\ref{relPBW4.4}{\rm )} and {\rm (}\ref{relPBW4.5}{\rm )}, that is, 
\begin{align*}
   \nu_{t_2}(x_1)\nu_{t_2}(t_1)-\nu_{t_2}(a_{111}t_1+b_{11})\nu_{t_2}(x_1) = &\ \nu_{t_2}(p_1), \\
   \nu_{t_2}(x_2)\nu_{t_2}(t_1)-\nu_{t_2}(a_{211}t_1+b_{21})\nu_{t_2}(x_2) = &\ \nu_{t_2}(p_2), \\
   \nu_{t_2}(x_1)\nu_{t_2}(t_2)-\nu_{t_2}(a_{122}t_2+b_{12})\nu_{t_2}(x_1) = &\ \nu_{t_2}(p_1), \\
    \nu_{t_2}(x_2)\nu_{t_2}(t_2)-\nu_{t_2}(a_{222}t_2+b_{22})\nu_{t_2}(x_2) = &\ \nu_{t_2}(p_2), \quad {\rm and} \\
    \nu_{t_2}(x_2)\nu_{t_2}(x_1)-c_{1,2}\nu_{t_2}(x_1)\nu_{t_2}(x_2) = &\ q_{1,2}^{(0)}+q_{1,2}^{(1)}\nu_{t_2}(x_1)+q_{1,2}^{(2)}\nu_{t_2}(x_2).
\end{align*}

Then
\begin{align}
p_1(a_{122}-1) = &\ 0, \notag \\
p_2(a_{222}-1) = &\ 0, \notag \\
q_{1,2}^{(0)}(a_{222}a_{122}-1) = &\ 0, \label{secondsecondrel2} \\
q_{1,2}^{(1)}(a_{222}-1) = &\ 0, \quad {\rm and} \notag \\
q_{1,2}^{(2)}(a_{122}-1) = &\ 0. \notag 
\end{align}

The map $\nu_{x_1}$ can be extended to an algebra homomorphism if and only if the definitions of $\nu_{x_1}(t_1)$, $\nu_{x_1}(t_2)$,  $\nu_{x_1}(x_1)$ and $\nu_{x_1}(x_2)$ respect relations {\rm (}\ref{relPBW4.1}{\rm )}, {\rm (}\ref{relPBW4.2}{\rm )}, {\rm (}\ref{relPBW4.3}{\rm )}, {\rm (}\ref{relPBW4.4}{\rm )} and {\rm (}\ref{relPBW4.5}{\rm )}. This yields that 
\begin{align*}
   \nu_{x_1}(x_1)\nu_{x_1}(t_1)-\nu_{x_1}(a_{111}t_1+b_{11})\nu_{x_1}(x_1) = &\ \nu_{x_1}(p_1), \\
   \nu_{x_1}(x_2)\nu_{x_1}(t_1)-\nu_{x_1}(a_{211}t_1+b_{21})\nu_{x_1}(x_2) = &\ \nu_{x_1}(p_2), \\
   \nu_{x_1}(x_1)\nu_{x_1}(t_2)-\nu_{x_1}(a_{122}t_2+b_{12})\nu_{x_1}(x_1) = &\ \nu_{x_1}(p_1), \\
    \nu_{x_1}(x_2)\nu_{x_1}(t_2)-\nu_{x_1}(a_{222}t_2+b_{22})\nu_{x_1}(x_2) = &\ \nu_{x_1}(p_2), \quad {\rm and} \\
    \nu_{x_1}(x_2)\nu_{x_1}(x_1)-c_{1,2}\nu_{x_1}(x_1)\nu_{x_1}(x_2) = &\ q_{1,2}^{(0)}+q_{1,2}^{(1)}\nu_{x_1}(x_1)+q_{1,2}^{(2)}\nu_{x_1}(x_2).
\end{align*}

Thus, we get that 
\begin{align}
p_1(a_{111}^{-1}-1) = &\ 0, \notag \\
p_1(a_{122}^{-1}-1) = &\ 0, \notag \\
a_{111}^{-1}(b_{21}-b_{11}+a_{211}b_
{11})-b_{21} = &\ 0, \label{thirdsecondrel2} \\
p_2(a_{111}^{-1}c_{1,2}-1) = &\ a_{111}^{-1}b_{21}q_{1,2}^{(1)}, \notag \\
a_{122}^{-1}(b_{22}-b_{12}+a_{222}b_
{12})-b_{22} = &\ 0, \notag \\
p_2(a_{122}^{-1}c_{1,2}-1) = &\ a_{122}^{-1}b_{22}q_{1,2}^{(1)}, \quad {\rm and} \notag \\
q_{1,2}^{(0)}(c_{1,2}-1) = &\ q_{1,2}^{(1)}q_{1,2}^{(2)}. \notag 
\end{align}

As above, the map $\nu_{x_2}$ can be extended to an algebra homomorphism if and only if the definitions of $\nu_{x_2}(t_1)$, $\nu_{x_2}(t_2)$,  $\nu_{x_2}(x_1)$ and $\nu_{x_2}(x_2)$ respect relations {\rm (}\ref{relPBW4.1}{\rm )}, {\rm (}\ref{relPBW4.2}{\rm )}, {\rm (}\ref{relPBW4.3}{\rm )}, {\rm (}\ref{relPBW4.4}{\rm )} and {\rm (}\ref{relPBW4.5}{\rm )}. Then:
\begin{align*}
   \nu_{x_2}(x_1)\nu_{x_2}(t_1)-\nu_{x_2}(a_{111}t_1+b_{11})\nu_{x_2}(x_1) = &\ \nu_{x_2}(p_1), \\
   \nu_{x_2}(x_2)\nu_{x_2}(t_1)-\nu_{x_2}(a_{211}t_1+b_{21})\nu_{x_2}(x_2) = &\ \nu_{x_2}(p_2), \\
   \nu_{x_2}(x_1)\nu_{x_2}(t_2)-\nu_{x_2}(a_{122}t_2+b_{12})\nu_{x_2}(x_1) = &\ \nu_{x_2}(p_1), \\
    \nu_{x_2}(x_2)\nu_{x_2}(t_2)-\nu_{x_2}(a_{222}t_2+b_{22})\nu_{x_2}(x_2) = &\ \nu_{x_2}(p_2), \quad {\rm and} \\
    \nu_{x_2}(x_2)\nu_{x_2}(x_1)-c_{1,2}\nu_{x_2}(x_1)\nu_{x_2}(x_2) = &\ q_{1,2}^{(0)}+q_{1,2}^{(1)}\nu_{x_2}(x_1)+q_{1,2}^{(2)}\nu_{x_2}(x_2).
\end{align*}

Equivalently, 
\begin{align}
p_2(a_{211}^{-1}-1) = &\ 0, \notag \\
p_2(a_{222}^{-1}-1) = &\ 0, \notag \\
a_{211}^{-1}(b_{11}-b_{21}+a_{111}b_
{21})-b_{11} = &\ 0, \label{fourthsecondrel2} \\
p_1(a_{211}^{-1}c_{1,2}^{-1}-1) = &\ -c_{1,2}^{-1}a_{211}^{-1}b_{11}q_{1,2}^{(2)}, \notag \\
a_{222}^{-1}(b_{12}-b_{22}+a_{122}b_
{22})-b_{12} = &\ 0, \notag \\
p_1(a_{122}^{-1}c_{1,2}^{-1}-1) = &\ -c_{1,2}^{-1}a_{222}^{-1}b_{12}q_{1,2}^{(2)}, \quad {\rm and} \notag \\ 
q_{1,2}^{(0)}(c_{1,2}^{-1}-1) = &\ -c_{1,2}^{-1}q_{1,2}^{(1)}q_{1,2}^{(2)}. \notag 
\end{align}

These equations are satisfied by the conditions formulated in the Table \ref{Leibnizsruletwoindeterminates(1)2}.

For the second assertion, it is enough to prove it for the generators $t_1$, $t_2$, $x_1$ and $x_2$:
\begin{align}
    (\nu_{t_1} \circ \nu_{t_2})(t_1) = &\ \nu_{t_1}(t_1)=t_1, \notag \\
    (\nu_{t_2} \circ \nu_{t_1})(t_1) = &\ \nu_{t_2}(t_1) = t_1, \notag \\
    (\nu_{t_1} \circ \nu_{t_2})(t_2) = &\ \nu_{t_1}(t_2)=t_2, \notag \\
    (\nu_{t_2} \circ \nu_{t_1})(t_2) = &\ \nu_{t_2}(t_2) = t_2, \notag \\
    (\nu_{t_1} \circ \nu_{t_2})(x_1) = &\ \nu_{t_1}(a_{122}x_1)=a_{111}a_{122}x_1, \label{comp1221}\\ 
    (\nu_{t_2} \circ \nu_{t_1})(x_1) = &\  \nu_{t_2}(a_{111}x_1)=a_{111}a_{122}x_1, \notag \\
    (\nu_{t_1} \circ \nu_{t_2})(x_2) = &\ \nu_{t_1}(a_{222}x_2)=a_{222}a_{211}x_2, \quad {\rm and} \notag \\  
 (\nu_{t_2} \circ \nu_{t_1})(x_2) = &\ \nu_{t_2}(a_{211}x_2)=a_{222}a_{211}x_2. \notag
\end{align}

In each case, the conditions shown in {\rm (}\ref{comp1221}{\rm )} hold, and so $\nu_{t_1} \circ \nu_{t_2}=\nu_{t_2} \circ \nu_{t_1}$.

Next, 
\begin{align}
    (\nu_{t_1} \circ \nu_{x_1})(t_1) = &\ \nu_{t_1}(a_{111}^{-1}(t_1-b_{11}))=a_{111}^{-1}(t_1-b_{11}), \notag \\
    (\nu_{x_1} \circ \nu_{t_1})(t_1) = &\ \nu_{x_1}(t_1) = a_{111}^{-1}(t_1-b_{11}), \notag \\
    (\nu_{t_1} \circ \nu_{x_1})(t_2) = &\ \nu_{t_1}(a_{122}^{-1}(t_2-b_{12}))=a_{122}^{-1}(t_2-b_{12}), \notag \\
    (\nu_{x_1} \circ \nu_{t_1})(t_2) = &\ \nu_{x_1}(t_2) = a_{122}^{-1}(t_2-b_{12}), \notag \\
    (\nu_{t_1} \circ \nu_{x_1})(x_1) = &\ \nu_{t_1}(x_1)=a_{111}x_1, \label{comp1111}\\ 
    (\nu_{x_1} \circ \nu_{t_1})(x_1) = &\  \nu_{x_1}(a_{111}x_1)=a_{111}x_1, \notag \\
    (\nu_{t_1} \circ \nu_{x_1})(x_2) = &\ \nu_{t_1}(c_{1,2}x_2+q_{1,2}^{(1)})=a_{211}c_{1,2}x_2+q_{1,2}^{(1)}, \quad {\rm and} \notag \\  
 (\nu_{x_1} \circ \nu_{t_1})(x_2) = &\ \nu_{x_1}(a_{211}x_2)=a_{211}(c_{1,2}x_2+q_{1,2}^{(1)}). \notag
\end{align}

Note that the only conditions that appear to be different in {\rm (}\ref{comp1111}{\rm )} are $(\nu_{t_1} \circ \nu_{x_1})(x_2)$ and $(\nu_{x_1} \circ \nu_{t_1})(x_2)$; these are satisfied when $q_{1,2}^{(1)}=a_{211}q_{1,2}^{(1)}$. As we can see, every case in Table   \ref{Leibnizsruletwoindeterminates(1)2} satisfies these conditions. 

Now, 
\begin{align}
    (\nu_{t_1} \circ \nu_{x_2})(t_1) = &\ \nu_{t_1}(a_{211}^{-1}(t_1-b_{21}))=a_{211}^{-1}(t_1-b_{21}), \notag \\
    (\nu_{x_2} \circ \nu_{t_1})(t_1) = &\ \nu_{x_2}(t_1) = a_{211}^{-1}(t_1-b_{21}), \notag \\
    (\nu_{t_1} \circ \nu_{x_2})(t_2) = &\ \nu_{t_1}(a_{222}^{-1}(t_2-b_{22}))=a_{222}^{-1}(t_2-b_{22}), \notag \\
    (\nu_{x_2} \circ \nu_{t_1})(t_2) = &\ \nu_{x_2}(t_2) = a_{222}^{-1}(t_2-b_{22}), \notag \\
    (\nu_{t_1} \circ \nu_{x_2})(x_1) = &\ \nu_{t_1}(c_{1,2}^{-1}(x_1-q_{1,2}^{(2)}))=c_{1,2}^{-1}(a_{111}x_1-q_{1,2}^{(2)})), \label{comp1221.}\\ 
    (\nu_{x_2} \circ \nu_{t_1})(x_1) = &\  \nu_{x_2}(a_{111}x_1)=a_{111}c_{1,2}^{-1}(x_1-q_{1,2}^{(2)}), \notag \\
    (\nu_{t_1} \circ \nu_{x_2})(x_2) = &\ \nu_{t_1}(x_2)=a_{211}x_2, \quad {\rm and} \notag \\  
 (\nu_{x_2} \circ \nu_{t_1})(x_2) = &\ \nu_{x_2}(a_{211}x_2)=a_{211}x_2. \notag
\end{align}

Once more again, note that the only conditions that appear to be different are $(\nu_{t_1} \circ \nu_{x_2})(x_1)$ and $(\nu_{x_2} \circ \nu_{t_1})(x_1)$; these are satisfied if $q_{1,2}^{(2)}=a_{111}q_{1,2}^{(2)}$, and all cases in Table \ref{Leibnizsruletwoindeterminates(1)2} satisfy both conditions. 

Consider 
\begin{align}
    (\nu_{t_2} \circ \nu_{x_1})(t_1) = &\ \nu_{t_2}(a_{111}^{-1}(t_1-b_{11}))=a_{111}^{-1}(t_1-b_{11}), \notag \\
    (\nu_{x_1} \circ \nu_{t_2})(t_1) = &\ \nu_{x_1}(t_1) = a_{111}^{-1}(t_1-b_{11}), \notag \\
    (\nu_{t_2} \circ \nu_{x_1})(t_2) = &\ \nu_{t_2}(a_{122}^{-1}(t_2-b_{12}))=a_{122}^{-1}(t_2-b_{12}), \notag \\
    (\nu_{x_1} \circ \nu_{t_2})(t_2) = &\ \nu_{x_1}(t_2) = a_{122}^{-1}(t_2-b_{12}), \notag \\
    (\nu_{t_2} \circ \nu_{x_1})(x_1) = &\ \nu_{t_2}(x_1)=a_{122}x_1, \label{comp2112.}\\ 
    (\nu_{x_1} \circ \nu_{t_2})(x_1) = &\  \nu_{x_1}(a_{122}x_1)=a_{122}x_1, \notag \\
    (\nu_{t_2} \circ \nu_{x_1})(x_2) = &\ \nu_{t_2}(c_{1,2}x_2+q_{1,2}^{(1)})=c_{1,2}a_{222}x_2+q_{1,2}^{(1)}, \quad {\rm and} \notag \\  
 (\nu_{x_1} \circ \nu_{t_2})(x_2) = &\ \nu_{x_1}(a_{222}x_2)=a_{222}(c_{1,2}x_2+q_{1,2}^{(1)}). \notag
\end{align}

By using a similar reasoning, it can be seen that all cases in Table \ref{Leibnizsruletwoindeterminates(1)2} satisfy these conditions.

We continue with the following compositions:
\begin{align}
    (\nu_{t_2} \circ \nu_{x_2})(t_1) = &\ \nu_{t_2}(a_{211}^{-1}(t_1-b_{21}))=a_{211}^{-1}(t_1-b_{21}), \notag \\
    (\nu_{x_2} \circ \nu_{t_2})(t_1) = &\ \nu_{x_2}(t_1) = a_{211}^{-1}(t_1-b_{21}), \notag \\
    (\nu_{t_2} \circ \nu_{x_2})(t_2) = &\ \nu_{t_2}(a_{222}^{-1}(t_2-b_{22}))=a_{222}^{-1}(t_2-b_{22}), \notag \\
    (\nu_{x_2} \circ \nu_{t_2})(t_2) = &\ \nu_{x_2}(t_2) = a_{222}^{-1}(t_2-b_{22}), \notag \\
    (\nu_{t_2} \circ \nu_{x_2})(x_1) = &\ \nu_{t_2}(c_{1,2}^{-1}(x_1-q_{1,2}^{(2)})=c_{1,2}^{-1}(a_{122}x_1-q_{1,2}^{(2)}), \label{comp2222}\\ 
    (\nu_{x_2} \circ \nu_{t_2})(x_1) = &\  \nu_{x_2}(a_{122}x_1)=a_{122}c_{1,2}^{-1}(x_1-q_{1,2}^{(2)}), \notag \\
    (\nu_{t_2} \circ \nu_{x_2})(x_2) = &\ \nu_{t_2}(x_2)=a_{222}x_2, \quad {\rm and} \notag \\  
 (\nu_{x_2} \circ \nu_{t_2})(x_2) = &\ \nu_{x_2}(a_{222}x_2)=a_{222}x_2. \notag
\end{align}

All cases in Table \ref{Leibnizsruletwoindeterminates(1)2} satisfy conditions in (\ref{comp2222}).

Finally, 
\begin{align}
    (\nu_{x_1} \circ \nu_{x_2})(t_1) = &\ \nu_{x_1}(a_{211}^{-1}(t_1-b_{21}))=a_{211}^{-1}(a_{111}^{-1}(t_1-b_{11})-b_{21}), \notag \\
    (\nu_{x_2} \circ \nu_{x_1})(t_1) = &\ \nu_{x_2}(a_{111}^{-1}(t_1-b_{11})) = a_{111}^{-1}(a_{211}^{-1}(t_1-b_{21})-b_{11}), \notag \\
    (\nu_{x_1} \circ \nu_{x_2})(t_2) = &\ \nu_{x_1}(a_{222}^{-1}(t_2-b_{22}))=a_{222}^{-1}(a_{122}^{-1}(t_2-b_{12})-b_{22}), \notag \\
    (\nu_{x_2} \circ \nu_{x_1})(t_2) = &\ \nu_{x_2}(a_{122}^{-1}(t_2-b_{12})) = a_{122}^{-1}(a_{222}^{-1}(t_2-b_{22})-b_{12}), \notag \\
    (\nu_{x_1} \circ \nu_{x_2})(x_1) = &\ \nu_{x_1}(c_{1,2}^{-1}(x_1-q_{1,2}^{(2)})=c_{1,2}^{-1}(x_1-q_{1,2}^{(2)}), \label{comp1221..}\\ 
    (\nu_{x_2} \circ \nu_{x_1})(x_1) = &\  \nu_{x_2}(x_1)=c_{1,2}^{-1}(x_1-q_{1,2}^{(2)}), \notag \\
    (\nu_{x_1} \circ \nu_{x_2})(x_2) = &\ \nu_{x_1}(x_2)=c_{1,2}x_2+q_{1,2}^{(1)}, \quad {\rm and} \notag \\  
 (\nu_{x_2} \circ \nu_{x_1})(x_2) = &\ \nu_{x_2}(c_{1,2}x_2+q_{1,2}^{(1)})=c_{1,2}x_2+q_{1,2}^{(1)}. \notag
\end{align}

At first glance, it seems that compositions $(\nu_{x_1} \circ \nu_{x_2})(t_1)$ and $(\nu_{x_2} \circ \nu_{x_1})(t_1)$ are different. However, due to the expression {\rm (}\ref{fourthsecondrel2}{\rm )}, they coincide. Similarly, it occurs with the compositions $(\nu_{x_1} \circ \nu_{x_2})(t_2)$ and $(\nu_{x_2} \circ \nu_{x_1})(t_2)$. All are satisfied.
\end{proof}

We arrive to the important result of this section.

\begin{theorem}\label{smoothPBW22}
If a SPBW extension $\sigma(\Bbbk[t_1, t_2])\langle x_1, x_2\rangle$ satisfies one of the conditions {\rm (a)-(p)} in Proposition \ref{autoSkewPBW22}, then it is differentially smooth.
\end{theorem}
\begin{proof}
As ${\rm GKdim}(\sigma(\Bbbk[t_1, t_2])\langle x_1, x_2\rangle) = 4$, we proceed to construct a $4$-dimensional integrable calculus. Consider $\Omega^{1}(\sigma(\Bbbk[t_1, t_2])\langle x_1, x_2\rangle)$, a free right $\sigma(\Bbbk[t_1, t_2])\langle x_1, x_2\rangle$-module of rank $4$ with generators $dt_1$, $dt_2, dx_1, dx_2$. Define a left $\sigma(\Bbbk[t_1, t_2])\langle x_1, x_2\rangle$-module structure by
    \begin{equation}\label{relrightmod322}
        adt_i = dt_i \nu_{t_i}(a), \quad adx_i = dx_i\nu_{x_i}(a),  \quad {\rm for\ all}\ i\in\{1, 2\}, a\in \sigma(\Bbbk[t_1, t_2])\langle x_1, x_2\rangle,
    \end{equation}
    
where $\nu_{t_i}$, $\nu_{x_i}$, $i\in\{1, 2\}$ are the algebra automorphisms established in Proposition \ref{autoSkewPBW22}. Notice that the relations in $\Omega^{1}(\sigma(\Bbbk[t_1, t_2])\langle x_1, x_2\rangle)$ are given by
    \begin{align}
        t_idt_j = &\ dt_j t_i & t_idx_j = &\ dx_ja_{jii}^{-1}(t_i-b_{ij}), \quad \text{ for all } i,j \in \{1, 2\}, \label{reltx4}  \\
        x_idx_i = &\ dx_ix_i,  &  x_idt_j = &\ dt_ja_{ijj}x_i, \quad \text{ for all } i,j \in \{1, 2\}, \label{reltx4.} 
         \end{align}

and 
        \begin{align}
        x_1dx_2 = &\ dx_2(c_{1,2}^{-1}x_1-c_{1,2}^{-1}q_{1,2}^{(2)}),  \\   x_2dx_1 = &\ dx_1(c_{1,2}x_2+q_{1,2}^{(1)}). \label{relxx4} 
    \end{align}

We extend the maps $t_i\mapsto dt_i$, $x_i\mapsto dx_i$ for $i \in \{1, 2\}$ to a map 
$$
d: \sigma(\Bbbk[t_1, t_2])\langle x_1, x_2\rangle \to \Omega^{1}(\sigma(\Bbbk[t_1, t_2])\langle x_1, x_2\rangle)
$$ 

satisfying the Leibniz's rule. From relations {\rm (}\ref{relPBW4.1}{\rm )}, {\rm (}\ref{relPBW4.2}{\rm )}, {\rm (}\ref{relPBW4.3}{\rm )}, {\rm (}\ref{relPBW4.4}{\rm )} and {\rm (}\ref{relPBW4.5}{\rm )} we get that
    \begin{align*}
        dx_it_j+x_idt_j &\ = a_{ijj}dt_jx_i+a_{ijj}t_jdx_i+b_{ij}dx_i, \quad \text{for}\ i,j \in \{1, 2\}, \quad {\rm and} \\
        dx_2x_1+x_2dx_1 &\ = c_{1,2}dx_1x_2+c_{1,2}x_1dx_2 +q_{1,2}^{(1)}dx_1+q_{1,2}^{(2)}dx_2.
    \end{align*}
    
Define $\Bbbk$-linear maps $$
\partial_{t_i}, \partial_{x_i}: \sigma(\Bbbk[t_1, t_2])\langle x_1, x_2\rangle \rightarrow \sigma(\Bbbk[t_1, t_2])\langle x_1, x_2\rangle
$$ 

such that
\begin{align*}
d(a)=dt_1\partial_{t_1}(a)+dt_2\partial_{t_2}(a)+\sum_{i=1}^{2}dx_i\partial_{x_i}(a), \text{ for all } a \in \sigma(\Bbbk[t_1, t_2])\langle x_1, x_2\rangle.
\end{align*}

These maps are well-defined since $dt_i$ and $dx_i$ for $i \in \{1, 2\}$ are free generators of the right $\sigma(\Bbbk[t_1, t_2])\langle x_1, x_2\rangle$-module $\Omega^1(\sigma(\Bbbk[t_1, t_2])\langle x_1, x_2\rangle)$. Hence, $d(a)=0$ if and only if $\partial_{t_i}(a)=\partial_{x_i}(a)=0$ for $i \in \{1, 2\}$. Using relations {\rm (}\ref{relrightmod322}{\rm )} and the definitions of the maps $\nu_{t_i}$, $\nu_{x_i}$, $i \in \{1, 2\}$, we obtain that
\begin{align}
\partial_{t_1}(t_1^kt_2^{s}x_1^{l_1}x_2^{l_2}) = &\ kt_1^{k-1}t_2^{s}x_1^{l_1}x_2^{l_2}, \\
\partial_{t_2}(t_1^kt_2^{s}x_1^{l_1}x_2^{l_2}) = &\ st_1^kt_2^{s-1}x_1^{l_1}x_2^{l_2}, \notag \\
\partial_{x_1}(t_1^kt_2^{s}x_1^{l_1}x_2^{l_2}) = &\ l_1a_{111}^{-k}a_{122}^{-s}(t_1-b_{11})^k(t_2-b_{12})^sx_1^{l_1-1}x_2^{l_2}, \quad {\rm and} \notag \\
\partial_{x_2}(t_1^kt_2^{s}x_1^{l_1}x_2^{l_2}) = &\ l_2c_{1,2}^{-l_1}a_{211}^{-k}a_{222}^{-s}(t_1-b_{21})^k(t_2-b_{22})^s(x_1-q_{1,2}^{(2)})^{l_1}x_2^{l_2-1}. \notag
\end{align}

Then, $d(a)=0$ if and only if $a$ is a scalar multiple of the identity. This shows that $\Omega(\sigma(\Bbbk[t_1, t_2])\langle x_1, x_2\rangle,d)$ is connected, where 
\[
\Omega(\sigma(\Bbbk[t_1, t_2])\langle x_1, x_2\rangle) = \bigoplus_{i=0}^{4}\Omega^i(\sigma(\Bbbk[t_1, t_2])\langle x_1, x_2\rangle).
\]

The universal extension of $d$ to higher forms compatible with {\rm (}\ref{reltx4}{\rm )}, {\rm (}\ref{reltx4.}{\rm )} and {\rm (}\ref{relxx4}{\rm )} gives the following rules for $\Omega^l(\sigma(\Bbbk[t_1, t_2])\langle x_1, x_2\rangle)$ with $l = 2, 3$:
\begin{align}\label{rel3wedge22}
 dt_2 \wedge dt_1  = &\ - dt_1 \wedge dt_2, \\
 dx_i\wedge dt_j = &\ -a_{ijj}dt_j\wedge dx_i, \quad \text{ for } i, j \in \{1, 2\}, \notag \\
 dx_2\wedge dx_1 = &\ -c_{i,j}dx_1\wedge dx_2,  \notag \\
dx_1\wedge dt_2\wedge dt_1 = &\ -a_{111}a_{122}dt_1\wedge dt_2\wedge dx_1,  \\
dx_2\wedge dt_2\wedge dt_1 = &\ -a_{211}a_{222}dt_1\wedge dt_2\wedge dx_2, \notag  \\
dx_2\wedge dx_1\wedge dt_1 = &\ -c_{1,2}a_{111}a_{211}dt_1\wedge dx_1\wedge dx_2, \quad {\rm and} \notag \\
dx_2\wedge dx_1\wedge dt_2 = &\  -c_{1,2}a_{122}a_{222}dt_2\wedge dx_1\wedge dx_2. \notag
\end{align}

Since the automorphisms $\nu_{t_i}$ and $\nu_{x_i}$ for $i \in \{1, 2\}$ commute with each other, there are no additional relationships to the previous ones, so we write
\begin{align*}
\Omega^{3}(\sigma(\Bbbk[t_1, t_2])\langle x_1, x_2\rangle) = &\ [dt_1\wedge dx_1\wedge dx_2 \oplus dt_2\wedge dx_1\wedge dx_2 \oplus dt_1\wedge dt_2\wedge dx_2 \\
&\ \oplus dt_1\wedge dt_2\wedge dx_1] \sigma(\Bbbk[t_1, t_2])\langle x_1, x_2\rangle.
\end{align*}

Now, due to that 
$$
\Omega^4(\sigma(\Bbbk[t_1, t_2])\langle x_1, x_2\rangle) = \omega\sigma(\Bbbk[t_1, t_2])\langle x_1, x_2\rangle\cong \sigma(\Bbbk[t_1, t_2])\langle x_1, x_2\rangle,
$$ 

as a right and left $\sigma(\Bbbk[t_1, t_2])\langle x_1, x_2\rangle$-module, with $\omega=dt_1\wedge dt_2 \wedge dx_1 \wedge dx_2$, where $\nu_{\omega}=\nu_{t_1}\circ\nu_{t_2}\circ\nu_{x_1}\circ\nu_{x_2}$, then $\omega$ is a volume form of $\sigma(\Bbbk[t_1, t_2])\langle x_1, x_2\rangle$. From Proposition \ref{BrzezinskiSitarz2017Lemmas2.6and2.7} (2), it follows that $\omega$ is an integral form by setting
\begin{align*}
    &\ \omega_1^1= dt_1, \quad \omega_2^1= dt_2, \quad \omega_3^1= dx_1, \quad \omega_4^1= dx_2, \\
    &\ \omega_1^2=dt_1\wedge dt_2, \quad \omega_2^2=dt_1\wedge dx_1, \quad \omega_3^2= dt_1\wedge dx_2, \quad \omega_4^2= dt_2\wedge dx_1, \\
    &\ \omega_5^2= dt_2\wedge dx_2, \quad  \omega_6^2= dx_1\wedge dx_2,   \\
    &\ \omega_1^3=dt_1\wedge dt_2\wedge dx_1, \quad \omega_2^3= dt_1\wedge dt_2 \wedge dx_2, \quad \omega_3^3= dt_1\wedge dx_1 \wedge dx_2, \\
    &\ \omega_4^3= dt_2\wedge dx_1 \wedge dx_2, \\
     &\ \bar{\omega}_1^1=-a_{211}^{-1}a_{222}^{-1}c_{1,2}^{-1}dx_2, \quad \bar{\omega}_2^1=a_{122}^{-1}a_{111}^{-1}dx_1,\quad \bar{\omega}_3^1=-dt_2, \quad \bar{\omega}_4^1=dt_1, \\
    &\ \bar{\omega}_1^2=a_{211}^{-1}a_{111}^{-1}a_{122}^{-1}a_{222}^{-1}dx_1\wedge dx_2, \\
    &\ \bar{\omega}_2^2=-a_{211}^{-1}c_{1,2}^{-1}dt_2\wedge dx_2, \quad \bar{\omega}_3^2 = a_{111}^{-1}dt_2\wedge dx_1, \quad \bar{\omega}_4^2 = a_{222}^{-1}c_{1,2}^{-1}dt_1\wedge dx_2, \\
    &\ \bar{\omega}_5^2 = -a_{122}^{-1}dt_1\wedge dx_1, \quad \bar{\omega}_6^2 = dt_1\wedge dt_2,\\
    &\ \bar{\omega}_1^3=-a_{111}^{-1}a_{211}^{-1}dt_2\wedge dx_1\wedge dx_2, \quad \bar{\omega}_2^3=a_{122}^{-1}a_{222}^{-1}dt_1\wedge dx_1 \wedge dx_2, \\
    &\ \bar{\omega}_3^3=-c_{1,2}^{-1}dt_1\wedge dt_2\wedge dx_2, \quad \bar{\omega}_4^3=dt_1 \wedge dt_2 \wedge dx_1
\end{align*}

Let $\omega' = dt_1a+dt_2b+dx_1c+dx_2d$, $a,b,c,d \in \Bbbk$. Then
\begin{align*}
\sum_{i=1}^{4}\omega_{i}^{1}\pi_{\omega}(\bar{\omega}_i^{3}\wedge \omega') &\ =dt_1\pi_{\omega}(-aa_{111}^{-1}a_{211}^{-1}dt_2\wedge dx_1\wedge dx_2 \wedge dt_1)\\
    &\ \quad + dt_2\pi_{\omega}(ba_{122}^{-1}a_{222}^{-1}dt_1\wedge dx_1 \wedge dx_2\wedge dt_2 ) \\
    &\ \quad + dx_1\pi_{\omega}(-cc_{1,2}^{-1}dt_1\wedge dt_2\wedge dx_2 \wedge dx_1)\\
    &\ \quad + dx_2\pi_{\omega}(ddt_1 \wedge dt_2 \wedge dx_1 \wedge dx_2)\\
    &\ = dt_1a + dt_2b + dx_1c + dx_2d \\
    &\ = \omega'.
\end{align*}

On the other hand, if 
\[
\omega'' = dt_1\wedge dt_2a+dt_1\wedge dx_1 b+dt_1\wedge dx_2 c+dt_2\wedge dx_1d+dt_2\wedge dx_2e+dx_1\wedge dx_2f 
\]

with $a,b,c,d,e,f \in \Bbbk$, it yields that 
\begin{align*}
\sum_{i=1}^{6}\omega_{i}^{2}\pi_{\omega}(\bar{\omega}_i^{2}\wedge \omega'') = &\ dt_1\wedge dt_2\pi_{\omega}(aa_{211}^{-1}a_{111}^{-1}a_{122}^{-1}a_{222}^{-1}dx_1\wedge dx_2\wedge dt_1\wedge dt_2) \\
    &\ + dt_1\wedge dx_1\pi_{\omega}(-ba_{211}^{-1}c_{1,2}^{-1}dt_2\wedge dx_2\wedge dt_1\wedge dx_1) \\
    &\ +dt_1\wedge dx_2\pi_{\omega}(ca_{111}^{-1}dt_2\wedge dx_1\wedge dt_1\wedge dx_2) \\
    &\ +dt_2\wedge dx_1\pi_{\omega}(da_{222}^{-1}c_{1,2}^{-1}dt_1\wedge dx_2\wedge dt_2\wedge dx_1) \\
    &\ +dt_2\wedge dx_2\pi_{\omega}(-ea_{122}^{-1}dt_1\wedge dx_1\wedge dt_2\wedge dx_2) \\
    &\ +dx_1\wedge dx_2\pi_{\omega}(fdt_1\wedge dt_2\wedge dx_1\wedge dx_2) \\
    = &\ dt_1\wedge dt_2a+dt_1\wedge dx_1 b+dt_1\wedge dx_2 c \\
    &\ + dt_2\wedge dx_1d+dt_2\wedge dx_2e+dx_1\wedge dx_2f \\
    = &\ \omega''.
\end{align*}

Finally, let 
\begin{align*}
\omega''' = &\ dt_1\wedge dt_2 \wedge dx_1 a+ dt_1\wedge dt_2\wedge dx_2 b dt_1\wedge dx_1 \\
&\ \wedge dx_2 c+dt_2\wedge dx_1 \wedge dx_2 d, \quad {\rm with}\ a, b, c, d \in \Bbbk.
\end{align*}

Since 
\begin{align*}
\sum_{i=1}^{3}\omega_{i}^{3}\pi_{\omega}(\bar{\omega}_i^{1}\wedge \omega''') = &\ dt_1\wedge dt_2\wedge dx_1 \pi_{\omega}(-aa_{211}^{-1}a_{222}^{-1}c_{1,2}^{-1}dx_2\wedge dt_1\wedge dt_2\wedge dx_1) \\
    &\ + dt_1\wedge dt_2 \wedge dx_2\pi_{\omega}(ba_{122}^{-1}a_{111}^{-1}dx_1\wedge dt_1\wedge dt_2 \wedge dx_2 )\\
    &\ + dt_1\wedge dx_1 \wedge dx_2\pi_{\omega}(-cdt_2\wedge dt_1\wedge dx_1 \wedge dx_2 )\\
    &\ +dt_2\wedge dx_1 \wedge dx_2\pi_{\omega}(ddt_1\wedge dt_2\wedge dx_1 \wedge dx_2) \\
    = &\ dt_1\wedge dt_2\wedge dx_1 a+ dt_1\wedge dt_2\wedge dx_2 b dt_1\wedge dx_1 \wedge dx_2 c \\
    &\ +dt_2\wedge dx_1 \wedge dx_2 d =  \omega''', 
\end{align*}

we conclude that $\sigma(\Bbbk[t_1, t_2])\langle x_1, x_2\rangle$ is differentially smooth.
\end{proof}
\begin{remark}
Only automorphisms of the form $t_1 \longmapsto a_{11}t_1+a_{12}t_2+a_{13}$, $t_2 \longmapsto a_{21}t_1+a_{22}t_2+a_{23}$, where $a_{i,j}\in \Bbbk$ and $a_{11}a_{22}-a_{12}a_{21}\not = 0$ are taken. This holds for the following reason. Without loss of generality, suppose that we have a relation with an automorphism of the form $t_1 \longmapsto t_1$, $t_2 \longmapsto t_2+h(t_1)$, where $h(t_1)\in\Bbbk[t_1]$ as follows
\begin{align*}
    x_1t_1 &\ =t_1x_1+p(t_1,t_2), \\
    x_1t_2 &\ =t_2x_1+h(t_1)x_1+p_1(t_1,t_2).
\end{align*}
When we want to take the automorphisms that can work to build the differential calculus, we obtain the following:
\begin{align*}
    d(x_1t_2) &\ =d(t_2x_1)+d(h(t_1)x_1)+d(p_1(t_1,t_2)), \\
    dx_1t_2+x_1dt_2 &\ =dt_2x_1+t_2dx_1+d(h(t_1))x_1+h(t_1)dx_1+d(p_1(t_1,t_2)). 
\end{align*}

Consider $h(t_1)=\sum_{s}a_st_1^{s}$, $dt_1h'(t_1)$ is the usual derivative with respect to $t_1$ and $dt_1\frac{\partial p_1}{\partial t_1}$ is the usual partial derivative with respect to $t_1$. Then
\begin{align*}
    &\ dx_1t_2+dt_2\nu_{t_2}(x_1)  =dt_2x_1+dx_1\nu_{x_1}(t_2)+dt_1h'(t_1)x_1 \\
    &\ \quad \quad \quad \quad \quad \quad \quad \quad \quad + dx_1\sum_{s}a_s\left[\nu_{x_1}(t_1)\right]^{s}+dt_1\frac{\partial p_1}{\partial t_1},
\end{align*}

and 
\begin{align*}
dx_1\left(t_2-\nu_{x_1}(t_2)-\sum_{s}a_s[\nu_{x_1}(t_1)]^{s}\right) &\ + dt_2(\nu_{t_2}(x_1)-x_1) \\
&\ + dt_1\left(h'(t_1)x_1+\frac{\partial p_1}{\partial t_1}\right)=0.
\end{align*}

Since that $dt_1$, $dt_2$, $dx_1$ and $dx_2$ are generators for $\Omega^1(\Bbbk[t_1, t_2])$, the elements that multiply them must be zero. In particular, 
\begin{equation}\label{otherauto}
    h'(t_1)x_1+\frac{\partial p_1}{\partial t_1}=0.
\end{equation}

Now, 
\begin{align*}
    h'(t_1)x_1 &\ =\sum_{s}sa_st_1^{s-1}x_1 = \sum_{s}sa_s\left(x_1t_1^{s-1}-(s-1)t_1^{s-2}p_1(t_1,t_2)\right) \\
    &\ = x_1h'(t_1)-\sum_{s}sa_s(s-1)t_1^{s-2}p_1(t_1,t_2).
\end{align*}

Therefore, expression {\rm (}\ref{otherauto}{\rm )} can be written as
\begin{equation*}
    x_1h'(t_1)-\sum_{s}sa_s(s-1)t_1^{s-2}p_1(t_1,t_2)+\frac{\partial p_1}{\partial t_1}=0.
\end{equation*}

In this way, it is necessary that $h'(t_1)=0$, i.e. $h(t_1)=h\in\Bbbk$, whence the automorphism has the form $t_1 \longmapsto t_1$, $t_2 \longmapsto t_2+h$, where $h\in\Bbbk$. Note that this automorphism coincides with the first type above when $a_{11} = 1,\ a_{12}=0, \ a_{13}=0$, $a_{21}=0, \ a_{22} = 1$ and $a_{23} = h$.
\end{remark}

Sections \ref{SPBWTMTwoI3} and \ref{SPBWTMTwoIn} contain sufficient conditions to guarantee the differential smoothness of SPBW extensions over $\sigma(\Bbbk[t_1, t_2])$ on three and $n$ generators. However, as we saw in the previous sections, the number of cases in which the Leibniz's rule holds increases considerably as we consider more indeterminates on the SPBW extension. Due to this reason, in the next two sections we will not include any table but will simply formulate the adequate conditions to the extension of this rule.

\subsection{SPBW extensions in three indeterminates}\label{SPBWTMTwoI3}

Let $\sigma(\Bbbk[t_1, t_2]) \langle x_1, x_2, x_3\rangle$. Consider the family of automorphisms as in the previous section. Since the conditions in {\rm (}\ref{DSdosdosSPBW}{\rm )} also hold in this case, it yields that
\begin{align}
    x_it_j = &\  a_{ijj}t_jx_i+b_{ij}x_i+ p_i, \quad \text{ for } i=1, 2, 3 \text{ and } j=1, 2, \label{DSdostresSPBW11} \\
    x_j x_i = &\ c_{i, j}x_i x_j + q_{i,j}^{(0)} + \sum_{k=1}^{3}q_{i,j}^{(k)}x_k, \quad\ {\rm for}\ i, j, k \in \{1, 2, 3\}, \ {\rm and} \ i < j, \label{DSdostresSPBW1}
\end{align}

where $c_{i,j}\in\Bbbk^{\ast}$, $q_{i,j}^{(0)}, q_{i,j}^{(k)}\in \Bbbk$,  for $i, j, k \in \{1, 2, 3\}$ with $i < j$, and $ a_{ijj}\in\Bbbk^{\ast}$, $b_{ij}, p_i \in\Bbbk$ for $i\in \{1, 2, 3\}$ and $j\in \{1,2 \}$.

\begin{proposition}\label{autoSkewPBW23}
Let
\begin{align}
   \nu_{t_i}(t_j) = &\ t_j, & \nu_{t_i}(x_k) = &\ a_{kii}x_k, \label{Autoskew4.38} \\
   \nu_{x_k}(t_i) = &\ a_{kii}^{-1}(t_i-b_{ki}),   &  \nu_{x_k}(x_k) = &\ x_k, \label{Autoskew4.39} \\
   \nu_{x_i}(x_j) = &\ c_{i,j}x_j+q_{i,j}^{(i)}, \quad {\rm for}\ i < j, & \nu_{x_i}(x_j) = &\ c_{j,i}^{-1}x_j-c_{j,i}^{-1}q_{j,i}^{(i)},  \quad {\rm for}\ i > j,  \label{Autoskew4.40}   
\end{align}

Then:
\begin{enumerate}
\item [\rm (1)] Leibniz's rule holds if the following relationships hold:
\begin{align}
    b_{sl}(a_{ill}-1) &\ =b_{il}(a_{sll}-1), \label{cond4.563} \\
    q_{s,i}^{(s)}(a_{ill}-1) &\ =0,\label{cond4.573} \\
    p_i(c_{s,i}-a_{sll}) &\ =b_{il}q_{s,i}^{(s)}, \label{cond4.583} \\
    q_{i,j}^{(s)}(c_{s,j}c_{s,i}-1) &\ =0, \label{cond4.593} \\
    q_{s,j}^{(s)}(c_{i,j}-1) &\ = q_{i,j}^{(i)}(c_{s,j}-1), \label{cond4.603} \\
    q_{s,i}^{(s)}(c_{i,j}-1) &\ = q_{i,j}^{(j)}(c_{s,i}-1), \label{cond4.613} \\
    q_{i,j}^{(k)}(c_{s,j}c_{s,i}-c_{k,s}^{-1}) &\ =0, \quad 1\leq k \leq 3, \quad {\rm and}  \label{cond4.623} \\
     \sum_{k=1}^{s-1}c_{k,s}^{-1}q_{i,j}^{(k)}q_{k,s}^{(s)}-\sum_{k=s+1}^{n}q_{i,j}^{(k)}q_{k,s}^{(s)} &\ 
     +q_{s,i}^{(s)}q_{s,j}^{(s)}(1-c_{i,j})+(c_{s,j}c_{s,i}-1)q_{i,j}^{(0)}  = 0. \label{cond4.633}
\end{align}

for $1 \leq i,j,s \leq n$, $i<j$, $l \in \{1, 2\}$, where $c_{r,t}:=c_{t,r}^{-1}$, $q_{r,t}^{(p)}:=-c_{t,r}^{-1}q_{t,r}^{(p)}$, for all $1 \leq r,t,p \leq n$, $t<r$, and  $c_{r,r}=1$ and $q_{r,r}^{(p)}=0$ for all $1 \leq r,p \leq n$.

The maps defined by {\rm (}\ref{Autoskew4.38}{\rm )}, {\rm (}\ref{Autoskew4.39}{\rm )} and {\rm (}\ref{Autoskew4.40}{\rm )} simultaneously extend to algebra automorphisms $\nu_{t_1}, \nu_{t_2}, \nu _{x_1}, \nu_{x_2}, \nu_{x_3}$ of $\sigma(\Bbbk[t_1, t_2])\langle x_1, x_2, x_3\rangle$ when the previous relations are satisfied.
\end{enumerate}
\begin{enumerate}
\item [\rm (2)] If the mentioned conditions in (1) hold, then we have that 
\begin{equation}\label{Eq1skew23}
   \nu_{t_i} \circ \nu_{t_j} = \nu_{t_j} \circ \nu_{t_i}, \quad \nu_{t_j} \circ \nu_{x_k} = \nu_{x_k} \circ \nu_{t_j}, \quad \nu_{x_k} \circ \nu_{x_l} = \nu_{x_l} \circ \nu_{x_k}, 
\end{equation}

for $i, j = 1, 2$ and $1\le k, l \leq 3$.
\end{enumerate}
\end{proposition}
\begin{proof}
For the first assertion, the map $\nu_{t_l}$, $l=1, 2$ can be extended to an algebra homomorphism if and only if the definitions of $\nu_{t_l}(t_j)$,  and $\nu_{t_l}(x_i)$ respect relations  {\rm (}\ref{DSdostresSPBW1}{\rm )}: 
\begin{equation*}
\nu_{t_l}(x_i)\nu_{t_l}(t_j)-\nu_{t_l}(a_{ijj}t_j+b_{ij})\nu_{t_l}(x_i) = \nu_{t_l}(p_i), \quad \text{for}\ i\in\{1, 2, 3\},\ j, l \in \{1, 2\},
\end{equation*}

and
\begin{equation*}
\nu_{t_l}(x_j)\nu_{t_l}(x_i)-c_{i,j}\nu_{t_l}(x_i)\nu_{t_l}(x_j) = q_{i,j}^{(0)}+\sum_{k=1}^{3}q_{i,j}^{(k)}\nu_{t_l}(x_k), 
\end{equation*}

for $i, j \in \{1, 2, 3\}, \ l \in \{1, 2\}$ with $i < j$. Then
\begin{align}
p_i(a_{ill}-1) = &\ 0, \quad {\rm for}\ i \in\{1, 2, 3\}, \ l \in \{1, 2\}, \notag \\
q_{i,j}^{(0)}(a_{jll}a_{ill}-1) = &\ 0, \quad {\rm and} \label{firstsecondrel33} \\
q_{i,j}^{(l)}(a_{jll}a_{ill}-a_{kll}) = &\ 0, \quad {\rm for} \ i, j,k \in \{1, 2, 3\}, \ l \in \{1, 2\}. \notag 
\end{align}

It is necessary that $\nu_{x_s}$ for $s = 1, 2, 3$ can be extended to $\sigma(\Bbbk[t_1, t_2])\langle x_1, x_2, x_3\rangle$, that is, 
\begin{equation}\label{mapvxt3}
\nu_{x_s}(x_i)\nu_{x_s}(t_j)-\nu_{x_s}(a_{ijj}t_j+b_{ij})\nu_{x_s}(x_i) = \nu_{x_s}(p_i), \quad \text{for}\ i\in\{1, 2, 3\},\ j, l \in \{1, 2\},
\end{equation}

and
\begin{equation}\label{mapvxx33}
\nu_{x_s}(x_j)\nu_{x_s}(x_i)-c_{i,j}\nu_{x_s}(x_i)\nu_{x_s}(x_j) = q_{i,j}^{(0)}+\sum_{k=1}^{3}q_{i,j}^{(k)}\nu_{x_s}(x_k).
\end{equation}

From Equation {\rm (}\ref{mapvxt3}{\rm )} we get the following three cases to be considered:
\begin{itemize}
    \item $s < i$:
    \begin{align*}
     \nu_{x_s}(x_i)\nu_{x_s}(t_j) &\ - \nu_{x_s}(a_{ijj}t_j+b_{ij})\nu_{x_s}(x_i) = \nu_{x_s}(p_i)\\
     c_{s,i}(x_i+q_{s,i}^{(s)})a_{sjj}^{-1}(t_j-b_{sj}) &\ - a_{ijj}a_{sjj}^{-1}(t_j-b_{sj})c_{s,i}(x_i+q_{s,i}^{(s)}) \\
        &\ - b_{ij}c_{s,i}(x_i+q_{s,i}^{(s)}= p_i.
    \end{align*}
    
   In this way, 
    \begin{align*}
        b_{sj}(a_{ijj}-1) &\ =b_{ij}(a_{sjj}-1), \\
        q_{s,i}^{(s)}(a_{ijj}-1) &\ =0, \quad {\rm and} \\
        p_i(c_{s,i}-a_{sjj}) &\ =b_{ij}q_{s,i}^{(s)}.
    \end{align*}
    
    \item $s = i$: It is straightforward that no new conditions are obtained.
    
    \item $s > i$: 

    \begin{align*}
          \nu_{x_s}(x_i)\nu_{x_s}(t_j) &\ - \nu_{x_s}(a_{ijj}t_j+b_{ij})\nu_{x_s}(x_i) = \nu_{x_s}(p_i) \\
           c_{i,s}^{-1}(x_i-q_{i,s}^{(s)})a_{sjj}^{-1}(t_j-b_{sj}) &\ -a_{ijj}a_{sjj}^{-1}(t_j-b_{sj})c_{i,s}^{-1}(x_i-q_{i,s}^{(s)}) \\
           &\ - b_{ij}c_{i,s}^{-1}(x_i-q_{i,s}^{(s)}) = p_i.
    \end{align*}
  
Then, 
\[
q_{i,s}^{(s)}(a_{ijj}-1) = 0 \quad {\rm and} \quad p_i(c_{i,s}a_{sjj}-1) = a_{sjj}b_{ij}q_{i,s}^{(s)}. 
\]
\end{itemize}

Due to expression {\rm (}\ref{mapvxx33}{\rm )} we have to consider the following cases:
\begin{itemize}
    \item $s < i$:
    \begin{equation*}
        \nu_{x_s}(x_j)\nu_{x_s}(x_i)-c_{i,j}\nu_{x_s}(x_i)\nu_{x_s}(x_j) = q_{i,j}^{(0)}+\sum_{k=1}^{3}q_{i,j}^{(k)}\nu_{x_s}(x_k)
    \end{equation*}
    From this equation, we obtain the following new conditions
    \begin{align*}
        q_{i,j}^{(s)}(c_{s,j}c_{s,i}-1) &\ =0, \\
        q_{s,j}^{(s)}(c_{i,j}-1) &\ = q_{i,j}^{i}(c_{s,j}-1), \\
        q_{s,i}^{(s)}(c_{i,j}-1) &\ = q_{i,j}^{j}(c_{s,i}-1), \\
        q_{i,j}^{(k)}(c_{s,j}c_{s,i}-c_{k,s}^{-1}) &\ =0, \quad 1\leq k \leq s-1, \\
        q_{i,j}^{(k)}(c_{s,j}c_{s,i}-c_{s,k}) &\ =0, \quad s+1\leq k \leq 3, k\ne i, j, \quad {\rm and}  \\
        \sum_{k=1}^{s-1}c_{k,s}^{-1}q_{i,j}^{(k)}q_{k,s}^{(s)}-\sum_{k=s+1}^{3}q_{i,j}^{(k)}q_{k,s}^{(s)} &\ +q_{s,i}^{(s)}q_{s,j}^{(s)}(1-c_{i,j})+(c_{s,j}c_{s,i}-1)q_{i,j}^{(0)}  = 0.
    \end{align*}
    
     \item $i < s < j$: 
    \begin{equation*}
        \nu_{x_s}(x_j)\nu_{x_s}(x_i)-c_{i,j}\nu_{x_s}(x_i)\nu_{x_s}(x_j) = q_{i,j}^{(0)}+\sum_{k=1}^{3}q_{i,j}^{(k)}\nu_{x_s}(x_k).
    \end{equation*}
    
    Hence, 
    \begin{align*}
        q_{i,j}^{(s)}(c_{s,j}-c_{i,s}) &\ =0, \\
        q_{s,j}^{(s)}(c_{i,j}-1) &\ = q_{i,j}^{(i)}(c_{s,j}-1), \\
        q_{i,s}^{(s)}(c_{i,j}-1) &\ = q_{i,j}^{(j)}(c_{i,s}-1), \\
        q_{i,j}^{(k)}(c_{s,j}c_{i,s}^{-1}-c_{k,s}^{-1}) &\ =0, \quad 1\leq k \leq s-1, k \ne i,  \\
        q_{i,j}^{(k)}(c_{s,j}c_{i,s}^{-1}-c_{s,k}) &\ =0, \quad s+1\leq k \leq n, k\ne j, \quad {\rm and} \\
        \sum_{k=1}^{s-1}c_{k,s}^{-1}q_{i,j}^{(k)}q_{k,s}^{(s)} -  \sum_{k=s+1}^{3}q_{i,j}^{(k)}q_{s,k}^{(s)} &\ + c_{i,s}^{-1}q_{s,j}^{(s)}q_{i,s}^{(s)}(c_{i,j}-1)  +(c_{s,j}c_{i,s}^{-1}-1)q_{i,j}^{(0)}  = 0.
    \end{align*}
    
    \item $s\geq j$: 
    \begin{equation*}
        \nu_{x_s}(x_j)\nu_{x_s}(x_i)-c_{i,j}\nu_{x_s}(x_i)\nu_{x_s}(x_j) = q_{i,j}^{(0)}+\sum_{k=1}^{3}q_{i,j}^{(k)}\nu_{x_s}(x_k).
    \end{equation*}

Then, 
    \begin{align*}
        q_{i,j}^{(s)}(c_{j,s}c_{i,s}-1) &\ =0, \\
        q_{j,s}^{(s)}(c_{i,j}-1) &\ = q_{i,j}^{(i)}(c_{j,s}-1), \\
        q_{i,j}^{(k)}(c_{j,s}^{-1}c_{i,s}^{-1}-c_{k,s}^{-1}) &\ =0, \quad 1\leq k \leq s-1, k \ne i, j\\
        q_{i,j}^{(k)}(c_{j,s}^{-1}c_{i,s}^{-1}-c_{s,k}) &\ =0, \quad s+1\leq k \leq n,  \quad {\rm and} \\
        \sum_{k=1}^{s-1}c_{k,s}^{-1}q_{i,j}^{(k)}q_{k,s}^{(s)} - \sum_{k=s+1}^{3}q_{i,j}^{(k)}q_{s,k}^{(s)}  &\ + c_{j,s}^{-1}c_{i,s}^{-1}q_{j,s}^{(s)}q_{i,s}^{(s)}(1-c_{i,j})  +(c_{j,s}^{-1}c_{i,s}^{-1}-1)q_{i,j}^{(0)}  = 0.
    \end{align*}
\end{itemize}

If we put together all the conditions for $s \in \{1, 2, 3\}$, then we obtain the restrictions for the extension of the automorphisms.

For the second assertion, it is enough to prove it for the generators $t_j$, $x_i$, $1\leq i \leq 3$ and $j \in \{1, 2\}$. Since
\begin{align}
    \nu_{t_k} \circ \nu_{t_m}(t_j) = &\ \nu_{t_k}(t_j)=t_j, \label{comptttj3}\\
    \nu_{t_m} \circ \nu_{t_k}(t_j) = &\ \nu_{t_m}(t_j)=t_j, \label{comptttj3.} \\
    \nu_{t_k} \circ \nu_{t_m}(x_i) = &\ \nu_{t_k}(a_{imm}x_i)=a_{imm}a_{ikk}x_i, \quad {\rm and}  \label{compttxi3}\\
    \nu_{t_m} \circ \nu_{t_k}(x_i) = &\ \nu_{t_m}(a_{ikk}x_i)=a_{imm}a_{ikk}x_i. \label{compttxi3.}
\end{align}

for all $1\leq i \leq 3$ and $k,m,j \in \{1, 2\}$, then all relations are satisfied, and hence $\nu_{t_k} \circ \nu_{t_m} = \nu_{t_m} \circ \nu_{t_k}$ for $k,m \in \{1, 2\}$.

Now, 
\begin{align}
    \nu_{t_k} \circ \nu_{x_m}(t_j) = &\ \nu_{t_k}(a_{mjj}^{-1}(t_j-b_{mj}))=a_{mjj}^{-1}(t_j-b_{mj}), \label{comptxtj3}\\
    \nu_{x_m} \circ \nu_{t_k}(t_j) = &\ \nu_{x_m}(t_j)=a_{mjj}^{-1}(t_j-b_{mj}), \label{comptxtj3.} \\
    \nu_{t_k} \circ \nu_{x_m}(x_i) = &\ \nu_{t_k}(c_{m,i}x_i+q_{m,i}^{(m)})=c_{m,i}a_{ikk}x_i+q_{m,i}^{(m)}, \label{comptxxi3} \quad {\rm and} \\
    \nu_{x_m} \circ \nu_{t_k}(x_i) = &\ \nu_{x_m}(a_{ikk}x_i)=a_{ikk}(c_{m,i}x_i+q_{m,i}^{(m)}), \label{comptxxi3.}
\end{align}

for all $1\leq i,m \leq 3$ and $k,j \in \{1, 2\}$. As it is clear, expressions {\rm (}\ref{comptxtj3}{\rm )} and {\rm (}\ref{comptxtj3.}{\rm )} hold, while relations {\rm (}\ref{comptxxi3}{\rm )} and {\rm (}\ref{comptxxi3.}{\rm )} are also satisfied due to expression {\rm (}\ref{cond4.593}{\rm )}. In this way, $\nu_{t_k} \circ \nu_{x_m} = \nu_{x_m} \circ \nu_{t_k}$, for  $1\leq m \leq 3$ and $k \in \{1, 2\}$.

Finally, 
\begin{align}
    \nu_{x_k} \circ \nu_{x_m}(t_j) = &\ \nu_{x_k}(a_{mjj}^{-1}(t_j-b_{mj}))=a_{mjj}^{-1}a_{kjj}^{-1}(t_j-b_{kj})-a_{mjj}^{-1}b_{jm}, \label{compxxtj3}\\
    \nu_{x_m} \circ \nu_{x_k}(t_j) = &\ \nu_{x_m}(a_{kjj}^{-1}(t_j-b_{kj}))=a_{kjj}^{-1}a_{mjj}^{-1}(t_j-b_{mj})-a_{kjj}^{-1}b_{kj}, \label{compxxtj3.} \\
    \nu_{x_k} \circ \nu_{x_m}(x_i) = &\ \nu_{x_k}(c_{m,i}x_i+q_{m,i}^{(m)})=c_{m,i}(c_{k,i}x_i+q_{k,i}^{(k)})+q_{m,i}^{(m)}, \quad {\rm and} \label{compxxxi3}\\
    \nu_{x_m} \circ \nu_{x_k}(x_i) = &\ \nu_{x_m}(c_{k,i}x_i+q_{k,i}^{(k)})=c_{k,i}(c_{m,i}x_i+q_{m,i}^{(m)})+q_{k,i}^{(k)}, \label{compxxxi3.}
\end{align}

for all $1\leq i,m,k  \leq 3$ and $j \in \{1, 2\}$. Then relations {\rm (}\ref{compxxtj3}{\rm )} and {\rm (}\ref{compxxtj3.}{\rm )} hold because relation {\rm (}\ref{cond4.563}{\rm )} is satisfied. With respect to relations {\rm (}\ref{compxxtj3}{\rm )} and {\rm (}\ref{compxxtj3.}{\rm )}, both are satisfied due to the expressions {\rm (}\ref{cond4.603}{\rm )} and {\rm (}\ref{cond4.613}{\rm )}. This yields that $\nu_{x_k} \circ \nu_{x_m} = \nu_{x_m} \circ \nu_{x_k}$, for  $1\leq k, m \leq 3$.
\end{proof}

We formulate the important result of this section.

\begin{theorem}\label{smoothPBW23}
If a SPBW extension $\sigma(\Bbbk[t_1, t_2])\langle x_1, x_2, x_3\rangle$ satisfies the conditions in Proposition \ref{autoSkewPBW23}, then it is differentially smooth.
\end{theorem}
\begin{proof}
We know that ${\rm GKdim}(\sigma(\Bbbk[t_1, t_2])\langle x_1, x_2, x_3\rangle) =5$ and that we have to consider $\Omega^{1}(\sigma(\Bbbk[t_1, t_2])\langle x_1, x_2, x_3\rangle)$, a free right $\sigma(\Bbbk[t_1, t_2])\langle x_1, x_2, x_3\rangle$-module of rank $5$ with generators $dt_1$, $dt_2, dx_1, dx_2, dx_3$. 

Define a left $\sigma(\Bbbk[t_1, t_2])\langle x_1, x_2, x_3\rangle$-module structure by
    \begin{equation}\label{relrightmod323}
        f dt_i = dt_i \nu_{t_i}(f), \quad f dx_j = dx_j\nu_{x_j}(f),
    \end{equation}

for all $i\in\{1, 2\},\ j \in\{1, 2, 3\}, f \in \sigma(\Bbbk[t_1, t_2])\langle x_1, x_2, x_3\rangle$, where $\nu_{t_i}$, $\nu_{x_j}$ are the algebra automorphisms for $i\in\{1, 2\}, \ j \in\{1, 2, 3\}$  established in Proposition \ref{autoSkewPBW23}. 

The relations in $\Omega^{1}(\sigma(\Bbbk[t_1, t_2])\langle x_1, x_2, x_3\rangle)$ are given by 
    \begin{align}
        t_idt_j = &\ dt_j t_i & t_idx_j = &\ dx_ja_{jii}^{-1}(t_i-b_{ij}), \quad \text{ for all } i\in\{1, 2\},\ j \in\{1, 2, 3\}, \label{reltx5}  \\
        x_idx_i = &\ dx_ix_i,  &  x_idt_j = &\ dt_ja_{ijj}x_i, \quad \text{ for all } i\in\{1, 2\},\ j \in\{1, 2, 3\}, \label{reltx5.} 
         \end{align}

and 
        \begin{align}
        x_idx_j = &\ dx_j(c_{i,j}^{-1}x_i-c_{i,j}^{-1}q_{i,j}^{(j)}), \quad \text{for}\ i < j, \quad {\rm and} \\
        x_idx_j = &\ dx_j(c_{j,i}x_i + q_{j,i}^{(j)}), \quad \text{for}\ i > j.  \label{relxx5} 
    \end{align}

We extend $t_i\mapsto dt_i$ and $x_j\mapsto dx_j$ for $i \in \{1, 2\}$ and $j \in \{1, 2, 3\}$ to a map 
$$
d: \sigma(\Bbbk[t_1, t_2])\langle x_1, x_2, x_3\rangle \to \Omega^{1}(\sigma(\Bbbk[t_1, t_2])\langle x_1, x_2, x_3\rangle)
$$ 

satisfying the Leibniz's rule. This must satisfy the relations given by
    \begin{align*}
        dx_it_j+x_idt_j &\ = a_{ijj}dt_jx_i+a_{ijj}t_jdx_i+b_{ij}dx_i, \quad \text{for}\ i\in\{1, 2\},\ j \in\{1, 2, 3\} \\
         dx_j x_i+x_jdx_1 = &\ c_{i, j}dx_i x_j+c_{i,j}x_jdx_1 + \sum_{k=1}^{3}q_{i,j}^{(k)}dx_k, \quad\ {\rm for}\ i, j, k \in \{1, 2, 3\},\ i<j.
    \end{align*}
    
Define $\Bbbk$-linear maps $$
\partial_{t_i}, \partial_{x_i}: \sigma(\Bbbk[t_1, t_2])\langle x_1, x_2, x_3\rangle \rightarrow \sigma(\Bbbk[t_1, t_2])\langle x_1, x_2, x_3\rangle
$$ 

such that
\begin{align*}
d(a)=dt_1\partial_{t_1}(a)+dt_2\partial_{t_2}(a)+\sum_{i=1}^{3}dx_i\partial_{x_i}(a), \text{ for all } a \in \sigma(\Bbbk[t_1, t_2])\langle x_1, x_2, x_3\rangle.
\end{align*}

These maps are well-defined since $dt_i$, $dx_j$, $i\in\{1, 2\}, j \in\{1, 2, 3\}$ are free generators of the right $\sigma(\Bbbk[t_1, t_2])\langle x_1, x_2, x_3\rangle$-module $\Omega^1(\sigma(\Bbbk[t_1, t_2])\langle x_1, x_2, x_3\rangle)$. With that, $d(a)=0$ if and only if $\partial_{t_i}(a)=\partial_{x_j}(a)=0$, $i\in\{1, 2\}$ and $j \in\{1, 2, 3\}$. Using relations {\rm (}\ref{relrightmod323}{\rm )} and definitions of the maps $\nu_{t_i}$, $\nu_{x_j}$, $i\in\{1, 2\}$ and $j \in\{1, 2, 3\}$, we obtain that
\begin{align}
\partial_{t_1}(t_1^kt_2^{s}x_1^{l_1}x_2^{l_2}x_3^{l_3}) = &\ kt_1^{k-1}t_2^{s}x_1^{l_1}x_2^{l_2}x_3^{l_3}, \\
\partial_{t_2}(t_1^kt_2^{s}x_1^{l_1}x_2^{l_2}x_3^{l_3}) = &\ st_1^kt_2^{s-1}x_1^{l_1}x_2^{l_2}x_3^{l_3}, \notag \\
\partial_{x_1}(t_1^kt_2^{s}x_1^{l_1}x_2^{l_2}x_3^{l_3}) = &\ l_1a_{111}^{-k}a_{122}^{-s}(t_1-b_{11})^k(t_2-b_{12})^sx_1^{l_1-1}x_2^{l_2}x_3^{l_3},  \notag \\
\partial_{x_2}(t_1^kt_2^{s}x_1^{l_1}x_2^{l_2}x_3^{l_3}) = &\ l_2c_{1,2}^{-l_1}a_{211}^{-k}a_{222}^{-s}(t_1-b_{21})^k(t_2-b_{22})^s(x_1-q_{1,2}^{(2)})^{l_1}x_2^{l_2-1}x_3^{l_3}, \quad {\rm and} \notag \\
\partial_{x_3}(t_1^kt_2^{s}x_1^{l_1}x_2^{l_2}x_3^{l_3}) = &\ 
l_3c_{2,3}^{-l_2}c_{1,3}^{-l_1}a_{311}^{-k}a_{322}^{-s}(t_1-b_{31})^k(t_2-b_{32})^s(x_1-q_{1,3}^{(3)})^{l_1} \notag \\
&\ (x_2-q_{2,3}^{(3)})^{l_2}x_3^{l_3-1}. \notag
\end{align}

Then, $d(a)=0$ if and only if $a$ is a scalar multiple of the identity. This shows that $\Omega(\sigma(\Bbbk[t_1, t_2])\langle x_1, x_2, x_3\rangle,d)$ is connected, where 
\[
\Omega(\sigma(\Bbbk[t_1, t_2])\langle x_1, x_2, x_3\rangle) = \bigoplus_{i=0}^{4}\Omega^i(\sigma(\Bbbk[t_1, t_2])\langle x_1, x_2, x_3\rangle).
\]

The universal extension of $d$ to higher forms compatible with {\rm (}\ref{reltx5}{\rm )}, {\rm (}\ref{reltx5.}{\rm )} and {\rm (}\ref{relxx5}{\rm )} gives the following rules for $\Omega^l(\sigma(\Bbbk[t_1, t_2])\langle x_1, x_2, x_3\rangle)$ with $l = 2, 3, 4$:
{\footnotesize{
\begin{align}\label{rel3wedge23}
 dt_2 \wedge dt_1  = &\ - dt_1 \wedge dt_2, \\
 dx_i\wedge dt_j = &\ -a_{ijj}dt_j\wedge dx_i, \quad \text{ for } i\in\{1, 2\}, j \in\{1, 2, 3\}, \notag \\
 dx_k\wedge dx_j = &\ -c_{j,k}dx_j\wedge dx_k, \quad \text{ for } j, k \in\{1, 2, 3\}, j<k \notag \\
dx_j\wedge dt_2\wedge dt_1 = &\ -a_{j11}a_{j22}dt_1\wedge dt_2\wedge dx_j, \text{ for } j \in\{1, 2, 3\}, \notag  \\
dx_k\wedge dx_j\wedge dt_i = &\ -c_{j,k}a_{jii}a_{kii}dt_i\wedge dx_j\wedge dx_k, \text{ for } i\in\{1, 2\}, j,k \in\{1, 2, 3\}, j <k, \notag  \\
dx_3\wedge dx_2\wedge dx_1 = &\ -c_{2,3}c_{1,2}c_{1,3}dx_1\wedge dx_2\wedge dx_3, \quad {\rm and} \notag \\
dx_k\wedge dx_j\wedge dt_2\wedge dt_1 = &\  c_{j, k}a_{1jj}a_{1kk}a_{2jj}a_{2kk}dt_1\wedge dt_2\wedge dx_j \wedge dx_k, \text{ for } j,k \in\{1, 2, 3\}, j <k,\notag \\
dx_3\wedge dx_2\wedge dx_1\wedge dt_i = &\  c_{1, 2}c_{1, 3}c_{2, 3}a_{1ii}a_{2ii}a_{3ii}dt_i\wedge dx_1\wedge dx_2 \wedge dx_3, \text{ for } i \in\{1, 2\},\notag 
\end{align}
}}

Since the automorphisms $\nu_{t_i}$, $\nu_{x_j}$, $i\in\{1, 2\}, j \in\{1, 2, 3\}$ commute with each other, there are no additional relationships to the previous ones, so we write
\begin{align*}
\Omega^{4}(\sigma(\Bbbk[t_1, t_2])\langle x_1, x_2, x_3\rangle) = &\ [dt_1\wedge dt_2\wedge dx_1 \wedge dx_2 \oplus dt_1\wedge dt_2\wedge dx_1 \wedge dx_3 \\
&\ \oplus dt_1\wedge dt_2\wedge dx_2 \wedge dx_3  \oplus dt_1\wedge dx_1\wedge dx_2 \wedge dx_3 \\
&\ \oplus dt_2\wedge dx_1\wedge dx_2 \wedge dx_3] \sigma(\Bbbk[t_1, t_2])\langle x_1, x_2, x_3\rangle.
\end{align*}

Now, 
$$
\Omega^5(\sigma(\Bbbk[t_1, t_2])\langle x_1, x_2, x_3\rangle) = \omega\sigma(\Bbbk[t_1, t_2])\langle x_1, x_2, x_3\rangle\cong \sigma(\Bbbk[t_1, t_2])\langle x_1, x_2, x_3\rangle
$$ 

as a right and left $\sigma(\Bbbk[t_1, t_2])\langle x_1, x_2, x_3\rangle$-module, with 
$$
\omega=dt_1\wedge dt_2 \wedge dx_1 \wedge dx_2 \wedge x_3, \quad {\rm where}\  \nu_{\omega}=\nu_{t_1}\circ\nu_{t_2}\circ\nu_{x_1}\circ\nu_{x_2}\circ\nu_{x_3}.
$$

Then $\omega$ is a volume form of $\sigma(\Bbbk[t_1, t_2])\langle x_1, x_2, x_3\rangle$. From Proposition \ref{BrzezinskiSitarz2017Lemmas2.6and2.7} (2), it follows that $\omega$ is an integral form by setting
\begin{align*}
    &\ \omega_1^1= dt_1, \quad \omega_2^1= dt_2, \quad \omega_3^1= dx_1, \quad \omega_4^1= dx_2, \quad \omega_5^1= dx_3 \\
    &\ \omega_1^2=dt_1\wedge dt_2, \quad \omega_2^2=dt_1\wedge dx_1, \quad \omega_3^2= dt_1\wedge dx_2, \quad  \omega_4^2= dt_1\wedge dx_3, \\
    &\ \omega_5^2= dt_2\wedge dx_1, \quad \omega_6^2= dt_2\wedge dx_2, \quad  \omega_7^2= dt_2\wedge dx_3, \quad  \omega_8^2= dx_1\wedge dx_2,  \\
    &\ \omega_9^2= dx_1\wedge dx_3, \quad \omega_{10}^2= dx_2\wedge dx_3,\\
    &\ \omega_1^3=dt_1\wedge dt_2\wedge dx_1, \quad \omega_2^3= dt_1\wedge dt_2 \wedge dx_2, \quad \omega_3^3= dt_1\wedge dt_2 \wedge dx_3, \\
    &\ \omega_4^3= dt_1\wedge dx_1 \wedge dx_2,  \quad \omega_5^3= dt_1\wedge dx_1 \wedge dx_3,  \quad \omega_6^3= dt_1\wedge dx_2 \wedge dx_3, \\
    &\ \omega_7^3= dt_2\wedge dx_1 \wedge dx_2,  \quad \omega_8^3= dt_2\wedge dx_1 \wedge dx_3,  \quad \omega_9^3= dt_2\wedge dx_2 \wedge dx_3, \\
     &\ \omega_{10}^3= dx_1\wedge dx_2 \wedge dx_3, \\
     &\ \omega_1^4=dt_1\wedge dt_2\wedge dx_1\wedge dx_2, \quad \omega_2^4= dt_1\wedge dt_2\wedge dx_1\wedge dx_3,\\
     &\ \omega_3^4= dt_1\wedge dt_2\wedge dx_2\wedge dx_3, \quad \omega_4^4= dt_1\wedge dx_1\wedge dx_2\wedge dx_3,  \\
     &\ \omega_5^4= dt_2\wedge dx_1\wedge dx_2\wedge dx_3,\\
     &\ \bar{\omega}_1^1=a_{311}^{-1}a_{322}^{-1}c_{1,3}^{-1}c_{2,3}^{-1}dx_3, \quad \bar{\omega}_2^1=-c_{1,2}^{-1}a_{222}^{-1}a_{211}^{-1}dx_2,\quad \bar{\omega}_3^1=a_{122}^{-1}a_{111}^{-1}dx_1, \\
     &\ \bar{\omega}_4^1=-dt_2, \quad \bar{\omega}_5^1=dt_1 \\
    &\ \bar{\omega}_1^2=c_{1,2}^{-1}c_{1,3}^{-1}a_{211}^{-1}a_{311}^{-1}a_{222}^{-1}a_{322}^{-1}dx_2\wedge dx_3, \quad  \bar{\omega}_2^2=-c_{2,3}^{-1}a_{111}^{-1}a_{311}^{-1}a_{122}^{-1}a_{322}^{-1}dx_1\wedge dx_3, \\
    &\ \bar{\omega}_3^2 = a_{111}^{-1}a_{211}^{-1}a_{122}^{-1}a_{222}^{-1}a_{111}^{-1}dx_1\wedge dx_2, \quad \bar{\omega}_4^2 = c_{1,3}^{-1}c_{2,3}^{-1}a_{311}^{-1}dt_2\wedge dx_3, \\
    &\ \bar{\omega}_5^2 = -c_{1,2}^{-1}a_{211}^{-1}dt_2\wedge dx_2, \quad \bar{\omega}_6^2 = -a_{111}^{-1}dt_2\wedge dx_1, \quad \bar{\omega}_7^2 = -c_{1,3}^{-1}c_{2,3}^{-1}a_{322}^{-1}dt_1\wedge dx_3,\\
    &\ \bar{\omega}_8^2 = c_{1,2}^{-1}a_{222}^{-1}dt_1\wedge dx_2, \quad \bar{\omega}_9^2 = -a_{122}^{-1}dt_1\wedge dx_1, \quad \bar{\omega}_{10}^2 = dt_1\wedge dt_2,\\
    &\ \bar{\omega}_1^3=-a_{111}^{-1}a_{211}^{-1}a_{311}^{-1}a_{122}^{-1}a_{222}^{-1}a_{322}^{-1}dx_1\wedge dx_2\wedge dx_3, \\
    &\ \bar{\omega}_2^3=-c_{1,2}^{-1}c_{1,3}^{-1}a_{211}^{-1}a_{311}^{-1}dt_2\wedge dx_2 \wedge dx_3, \quad   \bar{\omega}_3^3=c_{2,3}^{-1}a_{311}^{-1}a_{111}^{-1}dt_2\wedge dx_1\wedge dx_3, \\
   &\ \bar{\omega}_4^3=-a_{211}^{-1}a_{111}^{-1}dt_2 \wedge dx_1 \wedge dx_2, \quad \bar{\omega}_5^3=c_{1,2}^{-1}c_{1,3}^{-1}a_{222}^{-1}a_{322}^{-1}dt_1 \wedge dx_2 \wedge dx_3, \\
   &\ \bar{\omega}_6^3=-c_{2,3}^{-1}a_{122}^{-1}a_{322}^{-1}dt_1 \wedge dx_1 \wedge dx_3, \quad \bar{\omega}_7^3=a_{122}^{-1}a_{222}^{-1}dt_1 \wedge dx_1 \wedge dx_2, \\
   &\ \bar{\omega}_8^3=c_{2,3}^{-1}c_{1,3}^{-1}dt_1 \wedge dt_2 \wedge dx_3, \quad \bar{\omega}_9^3=-c_{1,2}^{-1}dt_1 \wedge dt_2 \wedge dx_2, \quad \bar{\omega}_{10}^3=dt_1 \wedge dt_2 \wedge dx_1, \\
   &\ \bar{\omega}_1^4=a_{311}^{-1}a_{211}^{-1}a_{111}^{-1}dt_2\wedge dx_1 \wedge dx_2 \wedge dx_3, \\
   &\ \bar{\omega}_2^4=-a_{322}^{-1}a_{222}^{-1}a_{122}^{-1}dt_1\wedge dx_1 \wedge dx_2 \wedge dx_3, \\
   &\ \bar{\omega}_3^4=c_{1,2}^{-1}c_{1,3}^{-1}dt_1\wedge dt_2 \wedge dx_2 \wedge dx_3, \quad \bar{\omega}_4^4=-c_{2,3}^{-1}dt_1\wedge dt_2 \wedge dx_1 \wedge dx_3, \\
   &\ \bar{\omega}_5^4=dt_1\wedge dt_2 \wedge dx_1 \wedge dx_2.
\end{align*}

It can be seen that any element $\omega' \in \Omega^{l}(\sigma(\Bbbk[t_1, t_2])\langle x_1, x_2, x_3\rangle)$, $l\in \{1, 2, 3, 4 \}$ can be generated by $\omega_{i}^{l}$, $\bar{\omega}_{i}^{5-l}$, $1\leq i \leq \binom{5}{l}$. Hence, $\sigma(\Bbbk[t_1, t_2])\langle x_1, x_2, x_3 \rangle$ is differentially smooth.
\end{proof}

\subsection{SPBW extensions in \texorpdfstring{$n$}{Lg} indeterminates}\label{SPBWTMTwoIn}

Let $\sigma(\Bbbk[t_1, t_2]) \langle x_1, \ldots, x_n\rangle$. Consider once more again the family of automorphisms of the previous section. We get that
\begin{align}
    x_it_j = &\  a_{ijj}t_jx_i+b_{ij}x_i+ p_i, \quad \text{ for } 1\leq i \leq n \text{ and } j=1, 2, \notag \\
    x_j x_i = &\ c_{i, j}x_i x_j + q_{i,j}^{(0)} + \sum_{k=1}^{n}q_{i,j}^{(k)}x_k, \quad\ {\rm for}\ 1\leq i, j, k \leq n,\ i<j, \label{DSdosnSPBW1}
\end{align}
where $c_{i,j}\in\Bbbk^{\ast}$, $q_{i,j}^{(0)}, q_{i,j}^{(k)}\in \Bbbk$,  for $1\leq i, j, k \leq n$, $i<j$ and $ a_{ijj}\in\Bbbk^{\ast}$, $b_{ij}, p_i \in\Bbbk$ where $1\leq i \leq n$ and $j\in \{1,2 \}$.
\begin{proposition}\label{autoSkewPBW2n} 
Let
\begin{align}
   \nu_{t_i}(t_j) = &\ t_j, & \nu_{t_i}(x_k) &\ = a_{kii}x_k, \label{Autoskew4.42} \\
   \nu_{x_k}(t_i) = &\ a_{kii}^{-1}(t_i-b_{ki}),   &  \nu_{x_k}(x_k) = &\ x_k, \label{Autoskew4.43} \\
   \nu_{x_i}(x_j) = &\ c_{i,j}x_j+q_{i,j}^{(i)}, \quad {\rm for}\ i < j, & \nu_{x_i}(x_j) = &\ c_{j,i}^{-1}x_j-c_{j,i}^{-1}q_{j,i}^{(i)},  \quad {\rm for}\ i > j,  \label{Autoskew4.44}   
\end{align}
Then:
\begin{enumerate}
\item [\rm (1)] Leibniz's rule holds if the following relationships are satisfied:
\begin{align}
    b_{sl}(a_{ill}-1) &\ =b_{il}(a_{sll}-1), \label{cond4.56} \\
    q_{s,i}^{(s)}(a_{ill}-1) &\ =0,\label{cond4.57} \\
    p_i(c_{s,i}-a_{sll}) &\ =b_{il}q_{s,i}^{(s)}, \label{cond4.58} \\
    q_{i,j}^{(s)}(c_{s,j}c_{s,i}-1) &\ =0, \label{cond4.59} \\
    q_{s,j}^{(s)}(c_{i,j}-1) &\ = q_{i,j}^{(i)}(c_{s,j}-1), \label{cond4.60} \\
    q_{s,i}^{(s)}(c_{i,j}-1) &\ = q_{i,j}^{(j)}(c_{s,i}-1), \label{cond4.61} \\
    q_{i,j}^{(k)}(c_{s,j}c_{s,i}-c_{k,s}^{-1}) &\ =0, \quad 1\leq k \leq n, \label{cond4.62} \\
     \sum_{k=1}^{s-1}c_{k,s}^{-1}q_{i,j}^{(k)}q_{k,s}^{(s)}-\sum_{k=s+1}^{n}q_{i,j}^{(k)}q_{k,s}^{(s)} &\ 
     +q_{s,i}^{(s)}q_{s,j}^{(s)}(1-c_{i,j})+(c_{s,j}c_{s,i}-1)q_{i,j}^{(0)}  = 0. \label{cond4.63}
\end{align}
for $1 \leq i,j,s \leq n$, $i<j$, $l \in \{1, 2\}$, having the convention of $c_{r,t}:=c_{t,r}^{-1}$, $q_{r,t}^{(p)}:=-c_{t,r}^{-1}q_{t,r}^{(p)}$, for all $1 \leq r,t,p \leq n$, $t<r$ and  $c_{r,r}=1$ and $q_{r,r}^{(p)}=0$ for all $1 \leq r,p \leq n$.

The maps defined by {\rm (}\ref{Autoskew4.42}{\rm )}, {\rm (}\ref{Autoskew4.43}{\rm )} and {\rm (}\ref{Autoskew4.44}{\rm )} simultaneously extend to algebra automorphisms $\nu_{t_1}, \nu_{t_2}, \nu _{x_i},1 \leq i \leq n$, of $\sigma(\Bbbk[t_1, t_2])\langle x_1, \ldots, x_n\rangle$ when the previous relations are satisfied.
\end{enumerate}
\begin{enumerate}
\item [\rm (2)] If the relations in {\rm (1)} hold, then  
\begin{equation}\label{Eq1skew2n}
   \nu_{t_i} \circ \nu_{t_j} = \nu_{t_j} \circ \nu_{t_i}, \quad \nu_{t_j} \circ \nu_{x_k} = \nu_{x_k} \circ \nu_{t_j}, \quad \nu_{x_k} \circ \nu_{x_l} = \nu_{x_l} \circ \nu_{x_k}, 
\end{equation}

for $i, j = 1, 2$ and $1\le k, l \leq n$.
\end{enumerate}
\end{proposition}
\begin{proof}
For the first assertion, the map $\nu_{t_l}$, $l=1, 2$ can be extended to an algebra homomorphism if and only if the definitions of $\nu_{t_l}(t_j)$,  and $\nu_{t_l}(x_i)$ respect relations  {\rm (}\ref{DSdosnSPBW1}{\rm )}, i.e. 
\begin{equation*}
\nu_{t_l}(x_i)\nu_{t_l}(t_j) - \nu_{t_l}(a_{ijj}t_j+b_{ij})\nu_{t_l}(x_i) = \nu_{t_l}(p_i), \quad \text{for}\ 1\le i \le n,\  j, l \in \{1, 2\},
\end{equation*}

and 
\begin{equation*}
\nu_{t_l}(x_j)\nu_{t_l}(x_i)-c_{i,j}\nu_{t_l}(x_i)\nu_{t_l}(x_j) = q_{i,j}^{(0)}+\sum_{k=1}^{3}q_{i,j}^{(k)}\nu_{t_l}(x_k),
\end{equation*}

for $1\le i, j \le n$, $l \in \{1, 2\}$, and $i < j$. Then
\begin{align}
p_i(a_{ill}-1) = &\ 0, \quad {\rm for}\ 1\le i \le n, \ l \in \{1, 2\} \notag \\
q_{i,j}^{(0)}(a_{jll}a_{ill}-1) = &\ 0, \quad {\rm and} \label{firstsecondrel33n} \\
q_{i,j}^{(k)}(a_{jll}a_{ill}-a_{kll}) = &\ 0, \quad {\rm for}\  1\leq i, j, k \leq n, \ l \in \{1, 2\}. \notag 
\end{align}

Now, it is necessary that $\nu_{x_s}$, $1\leq s \leq n$,  can also be extended to the entire SPBW extension $\sigma(\Bbbk[t_1, t_2])\langle x_1, \ldots, x_n\rangle$. Thus,
\begin{equation}\label{mapvxtn}
\nu_{x_s}(x_i)\nu_{x_s}(t_j)-\nu_{x_s}(a_{ijj}t_j+b_{ij})\nu_{x_s}(x_i) = \nu_{x_s}(p_i), \quad \text{for}\ i\in\{1, \ldots, n\},\ j, l \in \{1, 2\},
\end{equation}

and
\begin{equation}\label{mapvxxnn}
\nu_{x_s}(x_j)\nu_{x_s}(x_i)-c_{i,j}\nu_{x_s}(x_i)\nu_{x_s}(x_j) = q_{i,j}^{(0)}+\sum_{k=1}^{n}q_{i,j}^{(k)}\nu_{x_s}(x_k), 
\end{equation}
For Equation {\rm (}\ref{mapvxtn}{\rm )}, the following three cases are considered:
\begin{itemize}
    \item $s<i$. We obtain 
    \begin{equation*}
        \nu_{x_s}(x_i)\nu_{x_s}(t_j)-\nu_{x_s}(a_{ijj}t_j+b_{ij})\nu_{x_s}(x_i) = \nu_{x_s}(p_i)
    \end{equation*}
    \begin{equation*}
        c_{s,i}(x_i+q_{s,i}^{(s)})a_{sjj}^{-1}(t_j-b_{sj})-a_{ijj}a_{sjj}^{-1}(t_j-b_{sj})c_{s,i}(x_i+q_{s,i}^{(s)})-b_{ij}c_{s,i}(x_i+q_{s,i}^{(s)})=p_i
    \end{equation*}
    From this equation, we obtain the following conditions
    \begin{align*}
        b_{sj}(a_{ijj}-1) &\ =b_{ij}(a_{sjj}-1), \\
        q_{s,i}^{(s)}(a_{ijj}-1) &\ =0, \\
        p_i(c_{s,i}-a_{sjj}) &\ =b_{ij}q_{s,i}^{(s)}.
    \end{align*}
    \item $s=i$. In this case, when the calculations are made, no new conditions are obtained.
    \item $s>i$. We obtain 
    \begin{equation*}
        \nu_{x_s}(x_i)\nu_{x_s}(t_j)-\nu_{x_s}(a_{ijj}t_j+b_{ij})\nu_{x_s}(x_i) = \nu_{x_s}(p_i)
    \end{equation*}
    \begin{equation*}
        c_{i,s}^{-1}(x_i-q_{i,s}^{(s)})a_{sjj}^{-1}(t_j-b_{sj})-a_{ijj}a_{sjj}^{-1}(t_j-b_{sj})c_{i,s}^{-1}(x_i-q_{i,s}^{(s)})-b_{ij}c_{i,s}^{-1}(x_i-q_{i,s}^{(s)})=p_i
    \end{equation*}
    From this equation, we obtain the following new conditions
    \begin{align*}
        q_{i,s}^{(s)}(a_{ijj}-1) &\ =0, \\
        p_i(c_{i,s}a_{sjj}-1) &\ =b_{ij}q_{i,s}^{(s)}.
    \end{align*}
\end{itemize}
By Equation {\rm (}\ref{mapvxxnn}{\rm )}, we have the following three cases:
\begin{itemize}
    \item $s\leq i$. We obtain
    \begin{equation*}
        \nu_{x_s}(x_j)\nu_{x_s}(x_i)-c_{i,j}\nu_{x_s}(x_i)\nu_{x_s}(x_j) = q_{i,j}^{(0)}+\sum_{k=1}^{n}q_{i,j}^{(k)}\nu_{x_s}(x_k)
    \end{equation*}
    From this equation, we obtain the following new conditions
    \begin{align*}
        q_{i,j}^{(s)}(c_{s,j}c_{s,i}-1) &\ =0, \\
        q_{s,j}^{(s)}(c_{i,j}-1) &\ = q_{i,j}^{(i)}(c_{s,j}-1), \\
        q_{s,i}^{(s)}(c_{i,j}-1) &\ = q_{i,j}^{(j)}(c_{s,i}-1), \\
        q_{i,j}^{(k)}(c_{s,j}c_{s,i}-c_{k,s}^{-1}) &\ =0, \quad 1\leq k \leq s-1, \\
        q_{i,j}^{(k)}(c_{s,j}c_{s,i}-c_{s,k}) &\ =0, \quad s+1\leq k \leq n, k\ne i, j, \\
        \sum_{k=1}^{s-1}c_{k,s}^{-1}q_{i,j}^{(k)}q_{k,s}^{(s)}-\sum_{k=s+1}^{n}q_{i,j}^{(k)}q_{k,s}^{(s)} &\ +q_{s,i}^{(s)}q_{s,j}^{(s)}(1-c_{i,j})+(c_{s,j}c_{s,i}-1)q_{i,j}^{(0)}  = 0.
    \end{align*}
    \item $i < s < j$. We obtain
    \begin{equation*}
        \nu_{x_s}(x_j)\nu_{x_s}(x_i)-c_{i,j}\nu_{x_s}(x_i)\nu_{x_s}(x_j) = q_{i,j}^{(0)}+\sum_{k=1}^{n}q_{i,j}^{(k)}\nu_{x_s}(x_k)
    \end{equation*}
    From this equation, we obtain the following new conditions
    \begin{align*}
        q_{i,j}^{(s)}(c_{s,j}-c_{i,s}) &\ =0, \\
        q_{s,j}^{(s)}(c_{i,j}-1) &\ = q_{i,j}^{(i)}(c_{s,j}-1), \\
        q_{i,s}^{(s)}(c_{i,j}-1) &\ = q_{i,j}^{(j)}(c_{i,s}-1), \\
        q_{i,j}^{(k)}(c_{s,j}c_{i,s}^{-1}-c_{k,s}^{-1}) &\ =0, \quad 1\leq k \leq s-1, k \ne i\\
        q_{i,j}^{(k)}(c_{s,j}c_{i,s}^{-1}-c_{s,k}) &\ =0, \quad s+1\leq k \leq n, k\ne j, \\
        \sum_{k=1}^{s-1}c_{k,s}^{-1}q_{i,j}^{(k)}q_{k,s}^{(s)} -  \sum_{k=s+1}^{n}q_{i,j}^{(k)}q_{s,k}^{(s)} &\ + c_{i,s}^{-1}q_{s,j}^{(s)}q_{i,s}^{(s)}(c_{i,j}-1)  +(c_{s,j}c_{i,s}^{-1}-1)q_{i,j}^{(0)}  = 0.
    \end{align*}
    \item $s\geq j$. We obtain
    \begin{equation*}
        \nu_{x_s}(x_j)\nu_{x_s}(x_i)-c_{i,j}\nu_{x_s}(x_i)\nu_{x_s}(x_j) = q_{i,j}^{(0)}+\sum_{k=1}^{n}q_{i,j}^{(k)}\nu_{x_s}(x_k)
    \end{equation*}
    From this equation, we obtain the following new conditions
    \begin{align*}
        q_{i,j}^{(s)}(c_{j,s}c_{i,s}-1) &\ =0, \\
        q_{j,s}^{(s)}(c_{i,j}-1) &\ = q_{i,j}^{(i)}(c_{j,s}-1), \\
        q_{i,j}^{(k)}(c_{j,s}^{-1}c_{i,s}^{-1}-c_{k,s}^{-1}) &\ =0, \quad 1\leq k \leq s-1, k \ne i, j\\
        q_{i,j}^{(k)}(c_{j,s}^{-1}c_{i,s}^{-1}-c_{s,k}) &\ =0, \quad s+1\leq k \leq n,  \\
        \sum_{k=1}^{s-1}c_{k,s}^{-1}q_{i,j}^{(k)}q_{k,s}^{(s)} - \sum_{k=s+1}^{n}q_{i,j}^{(k)}q_{s,k}^{(s)}  &\ + c_{j,s}^{-1}c_{i,s}^{-1}q_{j,s}^{(s)}q_{i,s}^{(s)}(1-c_{i,j})  +(c_{j,s}^{-1}c_{i,s}^{-1}-1)q_{i,j}^{(0)}  = 0.
    \end{align*}
\end{itemize}
Putting together all the conditions for $s$, we obtain the restrictions for the extension of the automorphisms.

For the second assertion, it is enough to prove it for the generators $t_j$, $x_i$, $1\leq i \leq n$ and $j \in \{1, 2\}$. Note that 
\begin{align}
    \nu_{t_k} \circ \nu_{t_m}(t_j) = &\ \nu_{t_k}(t_j)=t_j, \label{comptttj}\\
    \nu_{t_m} \circ \nu_{t_k}(t_j) = &\ \nu_{t_m}(t_j)=t_j, \label{comptttj.} \\
    \nu_{t_k} \circ \nu_{t_m}(x_i) = &\ \nu_{t_k}(a_{imm}x_i)=a_{imm}a_{ikk}x_i, \label{compttxi}\\
    \nu_{t_m} \circ \nu_{t_k}(x_i) = &\ \nu_{t_m}(a_{ikk}x_i)=a_{imm}a_{ikk}x_i. \label{compttxi.}
\end{align}

for all $1\leq i \leq n$ and $k,m,j \in \{1, 2\}$, so all relations are satisfied. It yields that $\nu_{t_k} \circ \nu_{t_m} = \nu_{t_m} \circ \nu_{t_k}$ for $k,m \in \{1, 2\}$.

Now, 
\begin{align}
    \nu_{t_k} \circ \nu_{x_m}(t_j) = &\ \nu_{t_k}(a_{mjj}^{-1}(t_j-b_{mj}))=a_{mjj}^{-1}(t_j-b_{mj}), \label{comptxtj}\\
    \nu_{x_m} \circ \nu_{t_k}(t_j) = &\ \nu_{x_m}(t_j)=a_{mjj}^{-1}(t_j-b_{mj}), \label{comptxtj.} \\
    \nu_{t_k} \circ \nu_{x_m}(x_i) = &\ \nu_{t_k}(c_{m,i}x_i+q_{m,i}^{(m)})=c_{m,i}a_{ikk}x_i+q_{m,i}^{(m)}, \label{comptxxi}\\
    \nu_{x_m} \circ \nu_{t_k}(x_i) = &\ \nu_{x_m}(a_{ikk}x_i)=a_{ikk}(c_{m,i}x_i+q_{m,i}^{(m)}), \label{comptxxi.}
\end{align}

for all $1\leq i,m \leq n$ and $k,j \in \{1, 2\}$. In this way, expressions {\rm (}\ref{comptxtj}{\rm )} and {\rm (}\ref{comptxtj.}{\rm )} hold. With respect to the relations {\rm (}\ref{comptxxi}{\rm )} and {\rm (}\ref{comptxxi.}{\rm )}, both also hold since {\rm (}\ref{cond4.59}{\rm )} is satisfied. So $\nu_{t_k} \circ \nu_{x_m} = \nu_{x_m} \circ \nu_{t_k}$, for  $1\leq m \leq n$ and $k \in \{1, 2\}$.

Finally, 
\begin{align}
    \nu_{x_k} \circ \nu_{x_m}(t_j) = &\ \nu_{x_k}(a_{mjj}^{-1}(t_j-b_{mj}))=a_{mjj}^{-1}a_{kjj}^{-1}(t_j-b_{kj})-a_{mjj}^{-1}b_{jm}, \label{compxxtj}\\
    \nu_{x_m} \circ \nu_{x_k}(t_j) = &\ \nu_{x_m}(a_{kjj}^{-1}(t_j-b_{kj}))=a_{kjj}^{-1}a_{mjj}^{-1}(t_j-b_{mj})-a_{kjj}^{-1}b_{kj}, \label{compxxtj.} \\
    \nu_{x_k} \circ \nu_{x_m}(x_i) = &\ \nu_{x_k}(c_{m,i}x_i+q_{m,i}^{(m)})=c_{m,i}(c_{k,i}x_i+q_{k,i}^{(k)})+q_{m,i}^{(m)}, \label{compxxxi}\\
    \nu_{x_m} \circ \nu_{x_k}(x_i) = &\ \nu_{x_m}(c_{k,i}x_i+q_{k,i}^{(k)})=c_{k,i}(c_{m,i}x_i+q_{m,i}^{(m)})+q_{k,i}^{(k)}, \label{compxxxi.}
\end{align}

for all $1\leq i,m,k  \leq n$ and $j \in \{1, 2\}$. Relations {\rm (}\ref{compxxtj}{\rm )} and {\rm (}\ref{compxxtj.}{\rm )} hold due to expression {\rm (}\ref{cond4.56}{\rm )}. Relations {\rm (}\ref{compxxtj}{\rm )} and {\rm (}\ref{compxxtj.}{\rm )} work, since we have the expressions {\rm (}\ref{cond4.60}{\rm )} and {\rm (}\ref{cond4.61}{\rm )}. So $\nu_{x_k} \circ \nu_{x_m} = \nu_{x_m} \circ \nu_{x_k}$, for  $1\leq k, m \leq n$.
\end{proof}
\begin{theorem}\label{smoothPBW2n}
If a SPBW extension $\sigma(\Bbbk[t_1, t_2])\langle x_1, \ldots, x_n\rangle$ satisfies the conditions in Proposition \ref{autoSkewPBW2n}, then it is differentially smooth.
\end{theorem}
\begin{proof}
It is clear that ${\rm GKdim}(\sigma(\Bbbk[t_1, t_2])\langle x_1, \ldots, x_n\rangle) = n+2$. We know that we have to consider $\Omega^{1}(\sigma(\Bbbk[t_1, t_2])\langle x_1, \ldots, x_n\rangle)$, a free right $\sigma(\Bbbk[t_1, t_2])\langle x_1, \ldots, x_n\rangle$-module of rank $n+2$ with generators $dt_1$, $dt_2, dx_j$, $1\leq j \leq n$. Define a left $\sigma(\Bbbk[t_1, t_2])\langle x_1, \ldots, x_n\rangle$-module structure by
    \begin{equation}\label{relrightmod32n}
        adt_i = dt_i \nu_{t_i}(a), \quad adx_j = dx_j\nu_{x_j}(a),  
    \end{equation}

for all $i\in\{1, 2\},\ j \in\{1, \ldots, n\},  a\in \sigma(\Bbbk[t_1, t_2])\langle x_1, \ldots, x_n\rangle$, where $\nu_{t_i}$, $\nu_{x_j}$, $i\in\{1, 2\}, j \in\{1, \ldots, n\}$ are the algebra automorphisms established in Proposition \ref{autoSkewPBW2n}. Notice that the relations in $\Omega^{1}(\sigma(\Bbbk[t_1, t_2])\langle x_1, \ldots, x_n\rangle)$ are given by
    \begin{align}
        t_idt_j = &\ dt_j t_i & t_idx_j = &\ dx_ja_{jii}^{-1}(t_i-b_{ij}), \quad \text{ for all } i\in\{1, 2\}, \ j \in\{1, \ldots, n\}, \label{reltxn}  \\
        x_idx_i = &\ dx_ix_i,  &  x_idt_j = &\ dt_ja_{ijj}x_i, \quad \text{ for all } i\in\{1, 2\}, \ j \in\{1, \ldots, n\}, \label{reltxn.} 
         \end{align}

and 
        \begin{align}
        x_idx_j = &\ dx_j(c_{i,j}^{-1}x_i-c_{i,j}^{-1}q_{i,j}^{(j)}), \quad \text{for}\ 1\leq i < j \leq n, \quad {\rm and} \\
        x_idx_j = &\ dx_j(c_{j,i}x_i + q_{j,i}^{(j)}), \quad \text{for}\ 1\leq i < j \leq n.  \label{relxxn} 
    \end{align}

We extend $t_i\mapsto dt_i$, $x_j\mapsto dx_j$, $i \in \{1, 2\}$, $j \in \{1, \ldots, n\}$ to a map 
$$
d: \sigma(\Bbbk[t_1, t_2])\langle x_1, \ldots, x_n\rangle \to \Omega^{1}(\sigma(\Bbbk[t_1, t_2])\langle x_1, \ldots, x_n\rangle)
$$ 

satisfying the Leibniz's rule. This is possible if we guarantee its compatibility with the non-trivial relations {\rm (}\ref{DSdosnSPBW1}{\rm )}, i.e.
\begin{align*}
        dx_it_j+x_idt_j &\ = a_{ijj}dt_jx_i+a_{ijj}t_jdx_i+b_{ij}dx_i, \quad \text{for}\ i\in\{1, 2\}, \ j \in\{1, \ldots, n\} \\
         dx_j x_i+x_jdx_i = &\ c_{i, j}dx_i x_j+c_{i,j}x_jdx_i + \sum_{k=1}^{n}q_{i,j}^{(k)}dx_k, \ {\rm for}\ i, j, k \in \{1, \ldots, n\}, \ i < j.
    \end{align*}
    
Define $\Bbbk$-linear maps $$
\partial_{t_i}, \partial_{x_j}: \sigma(\Bbbk[t_1, t_2])\langle x_1, \ldots, x_n\rangle \rightarrow \sigma(\Bbbk[t_1, t_2])\langle x_1, \ldots, x_n\rangle
$$ 

such that
\begin{align*}
d(a)=dt_1\partial_{t_1}(a)+dt_2\partial_{t_2}(a)+\sum_{i=1}^{n}dx_i\partial_{x_i}(a), \text{ for all } a \in \sigma(\Bbbk[t_1, t_2])\langle x_1, \ldots, x_n\rangle.
\end{align*}

These maps are well-defined since $dt_i$, $dx_j$, $i\in\{1, 2\}, j \in\{1, \ldots, n\}$ are free generators of the right $\sigma(\Bbbk[t_1, t_2])\langle x_1, \ldots, x_n\rangle$-module $\Omega^1(\sigma(\Bbbk[t_1, t_2])\langle x_1, \ldots, x_n\rangle)$. Then $d(a)=0$ if and only if $\partial_{t_i}(a)=\partial_{x_j}(a)=0$ for $i\in\{1, 2\}, j \in\{1, \ldots, n\}$. Using relations {\rm (}\ref{relrightmod32n}{\rm )} and definitions of the maps $\nu_{t_i}$, $\nu_{x_j}$, $i\in\{1, 2\}, j \in\{1, \ldots, n\}$, we obtain that
\begin{align}
\partial_{t_1}(t_1^kt_2^{s}x_1^{l_1}\cdots x_n^{l_n}) = &\ kt_1^{k-1}t_2^{s}x_1^{l_1}\cdots x_n^{l_n}, \\
\partial_{t_2}(t_1^kt_2^{s}x_1^{l_1}\cdots x_n^{l_n}) = &\ st_1^kt_2^{s-1}x_1^{l_1}\cdots x_n^{l_n}, \notag \\
\partial_{x_j}(t_1^kt_2^{s}x_1^{l_1}\cdots x_n^{l_n}) = &\ l_ja_{j11}^{-k}a_{j22}^{-s}(t_1-b_{j1})^k(t_2-b_{j2})^s \\ \notag
&\ \prod_{s=1}^{j-1}c_{s,j}^{-l_s}(x_s-q_{s,j}^{(j)})^{l_s} x_j^{l_j-1}x_{j+1}^{l_{j+1}}\cdots x_n^{l_n}, \quad 1 \leq j \leq n. \notag 
\end{align}

Since $d(a)=0$ if and only if $a$ is a scalar multiple of the identity, it follows that $\Omega(\sigma(\Bbbk[t_1, t_2])\langle x_1, \ldots, x_n\rangle,d)$ is connected, where 
\[
\Omega(\sigma(\Bbbk[t_1, t_2])\langle x_1, \ldots, x_n\rangle) = \bigoplus_{i=0}^{n+1}\Omega^i(\sigma(\Bbbk[t_1, t_2])\langle x_1, \ldots, x_n\rangle).
\]

The universal extension of $d$ to higher forms compatible with {\rm (}\ref{reltxn}{\rm )}, {\rm (}\ref{reltxn.}{\rm )} and {\rm (}\ref{relxxn}{\rm )} gives the following rules for $\Omega^l(\sigma(\Bbbk[t_1, t_2])\langle x_1, \ldots, x_n\rangle)$ $(l = 2, \ldots, n+1)$:
\begin{align}
 dx_{q(1)}\wedge \dotsb \wedge &\ dx_{q(s)}\wedge dt_2 \wedge dt_1 \wedge dx_{q(s+1)}\wedge \dotsb \wedge dx_{q(l)} \\
 &\ = (-1)^{s+1}\prod_{\substack{r=1,\\r \ne s_1, s_2}}^{s} a_{q(r)11}^{-1}a_{q(r)22}^{-1} dt_1 \wedge dt_2 \wedge \bigwedge_{\substack{k=1,\\k \ne s_1, s_2}}^{l}dx_{q(k)},\label{relasins1s2n1}  \\
 dx_{q(1)}\wedge \dotsb \wedge &\ dx_{q(s)}\wedge dt_1 \wedge dt_2 \wedge dx_{q(s+1)}\wedge \dotsb \wedge dx_{q(l)} \\
 &\ = (-1)^{s}\prod_{\substack{r=1,\\r \ne s_1, s_2}}^{s} a_{q(r)11}^{-1}a_{q(r)22}^{-1} dt_1 \wedge dt_2 \wedge \bigwedge_{\substack{k=1,\\k \ne s_1, s_2}}^{l}dx_{q(k)}, \label{relasins1s2n2}  \\
 dx_{q(1)}\wedge \dotsb \wedge &\ dx_{q(s)}\wedge dt_i  \wedge dx_{q(s+1)} \wedge \dotsb \wedge dx_{q(l)} \\
 &\ = (-1)^{s}\prod_{\substack{r=1,\\r \ne s_1}}^{s} a_{q(r)ii}^{-1} dt_i \wedge \bigwedge_{\substack{k=1,\\k \ne s_1}}^{l}dx_{q(k)}, \label{relasins1s2n3}  \\
   \bigwedge_{k=1}^{l}dx_{q(k)} &\ = (-1)^{\sharp}\prod_{r,s\in P}c_{r,s}^{-1}\bigwedge_{k=1}^ldx_{p(k)},
\end{align}

where $s_1, s_2 \in \{1,\ldots,l\}$ do not appear in expression {\rm (}\ref{relasins1s2n1}{\rm )}, {\rm (}\ref{relasins1s2n2}{\rm )} or {\rm (}\ref{relasins1s2n3}{\rm )}. Besides, 
$$
q:\{1,\ldots,l\}\rightarrow \{1,\ldots,n\}
$$ 

is an injective map, and 
$$
p:\{1,\ldots,l\}\rightarrow \text{Im}(q)
$$ 

is an increasing injective map, $\sharp$ is the number of $2$-permutations needed to transform $q$ into $p$, and $P := \{(s, t) \in \{1, \ldots, l\} \times \{1, \ldots, l\} \mid q(s) >  q(t)\}$.

Since the automorphisms $\nu_{t_i}$, $\nu_{x_j}$, $i\in\{1, 2\}, j \in\{1, \ldots, n\}$ commute with each other, there are no additional relationships to the previous ones, so 
{\small{
\begin{align*}
\Omega^{n+1} (\sigma(\Bbbk[t_1, t_2])\langle x_1,\ldots, x_n\rangle) = &\ \left[\bigoplus_{r=2}^{n-1}dt_1\wedge dt_2 \wedge dx_1 \wedge \cdots \wedge dx_{r-1}\wedge dx_{r+1} \wedge \cdots \wedge dx_{n} \right. \\
&\   \oplus dt_1 \wedge dx_1 \wedge \cdots\wedge dx_{n} \\
&\ \left. \oplus dt_2 \wedge dx_1 \wedge \cdots \wedge dx_{n}\right]\sigma(\Bbbk[t_1, t_2])\langle x_1,\ldots, x_n\rangle.
\end{align*}
}}

Now, 
$$
\Omega^{n+2}(\sigma(\Bbbk[t_1, t_2])\langle x_1, \ldots, x_n\rangle) = \omega\sigma(\Bbbk[t_1, t_2])\langle x_1, \ldots, x_n\rangle\cong \sigma(\Bbbk[t_1, t_2])\langle x_1, \ldots, x_n\rangle
$$ 

as a right and left $\sigma(\Bbbk[t_1, t_2])\langle x_1, \ldots, x_n\rangle$-module, with $\omega=dt_1\wedge dt_2 \wedge dx_1 \wedge \cdots \wedge x_n$, where $\nu_{\omega}=\nu_{t_1}\circ\nu_{t_2}\circ\nu_{x_1}\circ\cdots\circ\nu_{x_n}$, this means that $\omega$ is a volume form of $\sigma(\Bbbk[t_1, t_2])\langle x_1, \ldots, x_n\rangle$. In order to make the calculations easier, we consider the following notation:
\[
t_1 = x_{-1}, \ t_2= x_0, \ c_{-1,0} = -1, \  c_{-1,i} = a_{i11}, \ c_{0,i} = a_{i22}, \quad {\rm for}\ 1\le i \le n.
\]

From Proposition \ref{BrzezinskiSitarz2017Lemmas2.6and2.7} (2) we get that $\omega$ is an integral form by setting
\begin{align*}
    \omega_i^j = &\ \bigwedge_{k=-1}^{j-2}dx_{p_{i,j}(k)}, \text{ for } 1\leq i \leq \binom{n+2}{j}, \quad {\rm and} \\
    \bar{\omega}_i^{n+2-j} = &\ (-1)^{\sharp_{i,j}}\prod_{r,s\in P_{i,j}}c_{r,s}^{-1}\bigwedge_{k=j-1}^{n}dx_{\bar{p}_{i,j}(k)}, \text{ for } 1\leq i \leq \binom{n+2}{j},
\end{align*}

for $1\leq j \leq n+2$, where 
\begin{align*}
    p_{i,j}:\{-1,\ldots,j-2\}\rightarrow &\ \{-1,\ldots,n\}, \quad {\rm and} \\
\bar{p}_{i,j}:\{j-1,\ldots,n\}\rightarrow &\ (\text{Im}(p_{i,j}))^c
\end{align*}

(the symbol $\square^c$ denotes the complement of the set $\square$), are increasing injective maps, and $\sharp_{i,j}$ is the number of $2$-permutation needed to transform 
\[
\left\{\bar{p}_{i,j}(j-1),\ldots, \bar{p}_{i,j}(n), p_{i,j}(-1), \ldots, p_{i,j}(j-2)\right\} \ {\rm into\ the\ set} \ \{-1, \ldots, n\},
\]

and 
\[
P_{i,j} :=\{(s, t) \in \{-1, \ldots, j-2\}\times\{j-1, \ldots, n\} \mid p_{i,j}(s)< \bar{p}_{i,j}(t)\}.
\]

Consider $\omega' \in \Omega^j(\sigma(\Bbbk[t_1, t_2])\langle x_1,\ldots, x_n\rangle)$, that is,  
\begin{align*}
\omega' =\sum_{i=1}^{\binom{n+2}{j}}\bigwedge_{k=-1}^{j-2}dx_{p_{i,j}(k)}\alpha_i, \quad {\rm with} \ \alpha_i \in \Bbbk.
\end{align*}

Then
\begin{align*}
 \sum_{i=1}^{\binom{n+2}{j}}\omega_{i}^{j}\pi_{\omega}(\bar{\omega}_i^{n+2-j}\wedge \omega') = &\ \sum_{i=1}^{\binom{n+2}{j}}\left[\bigwedge_{k=-1}^{j-2}dx_{p_i(k)}\right] \cdot  \pi_{\omega} \left[(-1)^{\sharp_{i,j}} \square^{*} \wedge \omega'\right] \\
 = &\ \displaystyle \sum_{i=1}^{\binom{n+2}{j}}\bigwedge_{k=-1}^{j-2}dx_{p_{i,j}(k)}\alpha_i =  \omega',
\end{align*}

where 
\begin{align*}
    \square^{*} := &\ \prod_{r,s\in P_{i,j}}c_{r,s}^{-1} \bigwedge_{k=j}^{n}dx_{\bar{p}_{i,j}(k)}.
\end{align*}

Therefore, $\sigma(\Bbbk[t_1, t_2])\langle x_1, \ldots, x_n \rangle$ is differentially smooth.
\end{proof}

\section{Future work}\label{FutureworkDSSPBW}

As expected, a natural task is to investigate the differential smoothness of SPBW extensions over commutative polynomial rings on three and more indeterminates. With this aim, we recall briefly some interesting facts on automorphisms of these polynomial rings. We follow Shestakov and Umirbaev's presentation \cite[p. 197]{ShestakovUmirbaev2003}.

For $\Bbbk[X] = \Bbbk[x_1, \dotsc, x_n]$, an automorphism $\tau \in {\rm Aut}(\Bbbk[X])$ is called {\em elementary} if it has a form
\[
\tau (x_1, \dotsc, x_{i-1}, x_i, x_{i+1}, x_n) \mapsto (x_1, \dotsc, x_{i-1}, ax_i + f, x_{i+1}, \dotsc, x_n),
\]

where $0\neq a\in \Bbbk, f\in \Bbbk[x_1, \dotsc, x_{i-1}, x_{i+1}, \dotsc x_n]$. The subgroup of ${\rm Aut}(\Bbbk[X])$ generated by all the elementary automorphisms is called {\em tame subgroup}, and the elements from this subgroup are called {\em tame automorphisms} of $\Bbbk[X]$. Non-tame automorphisms of $\Bbbk[X]$ are called {\em wild}.

In the literature it has been shown that the automorphisms of polynomial rings and free associative algebras in two indeterminates are tame (e.g. \cite{McKayWang1988, VandenEssen2000}). Nevertheless, in the case of three or more indeterminates the similar question was open and known as \textquotedblleft The generation gap problem\textquotedblright\ \cite{VandenEssen2000} or \textquotedblleft Tame generators problem\textquotedblright\ \cite{VandenEssen2000}. The general belief was that the answer is negative, and the best known counterexample is the following automorphism $\sigma \in {\rm Aut}(\Bbbk[x, y, z])$, constructed by Nagata \cite{Nagata1972}:
\begin{align*}
    \sigma(x) = &\ x + (x^2 - yz)z, \\
    \sigma(y) = &\ y + 2(x^2 - yz)x + (x^2 - yz)^2 z, \quad {\rm and} \\
    \sigma(z) = &\ z.
\end{align*}

Note that Nagata automorphism is {\em stably tame}; that is, it becomes tame after adding new variables. Shestakov and Umirbaev \cite{ShestakovUmirbaev2003} gave a negative answer to the above question; in particular, the Nagata automorphism $\sigma$ is wild.

Since that the study of ${\rm Aut}(\Bbbk[x, y, z])$ and ${\rm Aut}(\Bbbk[x_1, \dotsc, x_n])$ requires greater mathematical techniques that have not been considered at the time when this paper was written, the study of the differential smoothness of SPBW extensions over these polynomial rings will be one of our next tasks.

On the other hand, since Artamonov \cite{Artamonov2015}, Venegas \cite{Venegas2015} and the second author \cite{ReyesSuarez2017} presented some results concerning automorphisms and derivations of SPBW extensions, it is natural to study relationships between this kind of morphisms and those adequate to characterize the smoothness of these extensions. This will also be our topic of interest in the immediate future. 


\section{Declarations}\label{}

This work was supported by Faculty of Science, Universidad Nacional de Colombia - Sede Bogot\'a, Colombia (Grant number 53880).

All authors declare that they have no conflicts of interest.


\begin{thebibliography}{60}

\bibitem{AbdiTalebi2023} M. Abdi and Y. Talebi. On the diameter of the zero-divisor graph over skew PBW extensions. {\em J. Algebra Appl.} {\bf 23}(5) (2024) 2450089.

\bibitem{AkalanMarubayashi2016} E. Akalan and H. Marubayashi. Multiplicative Ideal Theory in Noncommutative Rings. In: S. Chapman, M. Fontana, A.  Geroldinger, B. Olberding, eds. {\em Multiplicative Ideal Theory and Factorization Theory. Commutative and Non-commutative Perspectives}. Springer Proceedings in Mathematics \& Statistics, Vol 170. Springer Cham (2016) pp. 1--22.


\bibitem{Apel1988} J. Apel. Gr\"obnerbasen in nichtkommutativen Algebren und ihre Anwendung. Doctoral Thesis. University of Leipzig, Leipzig, Germany (1988).

\bibitem{Artamonov2015} V. Artamonov. Derivations of skew PBW extensions. {\em Commun. Math. Stat.} {\bf 3}(4) (2015) 449--457.


\bibitem{Bavula2020} V. V. Bavula. Generalized Weyl algebras and diskew polynomial rings. {\em J. Algebra Appl.} {\bf 19}(10) (2020) 2050194.

\bibitem{Bavula2023} V. V. Bavula. Description of bi-quadratic algebras on 3 generators with {P}{B}{W} basis. {\em J. Algebra} {\bf 631} (2023) 695--730.

\bibitem{BeggsBrzezinski2005} E. J. Beggs and T. Brzezi{\'n}ski. The Serre spectral sequence of a noncommutative fibration for de Rham cohomology. {\em Acta Math.} {\bf 195}(2) (2005) 155--196.

\bibitem{BellSmith1990} A. D. Bell and S. P. Smith. Some 3-dimensional skew polynomial ring. University of Wisconsin, Milwaukee. Preprint (1990).

\bibitem{BellGoodearl1988} A. Bell and K. Goodearl. Uniform Rank Over Differential Operator Rings and Poincar{\'e}-Birkhoff-Witt Extensions. {\em Pacific J. Math.} {\bf 131}(1) (1988) 13--37.

\bibitem{BrownGoodearl2002} K. Brown and K. R. Goodearl. Lectures on {A}lgebraic {Q}uantum {G}roups. Advanced Courses in Mathematics - CRM Barcelona. {\em Birk\"auser Basel} (2002).

\bibitem{Brzezinski2008} T. Brzezi{\'n}ski. Non-commutative connections of the second kind. {\em J. Algebra Appl.} {\bf 7}(5) (2008) 557--573.

\bibitem{Brzezinski2011} T. Brzezi{\'n}ski. Divergences on Projective Modules and noncommutative Integrals. {\em Int. J. Geom. Methods Mod. Phys.} {\bf 8}(4) (2011) 885--896.

\bibitem{Brzezinski2014} T. Brzezi{\'n}ski. On the Smoothness of the Noncommutative Pillow and Quantum Teardrops. {\em SIGMA Symmetry Integrability Geom. Methods Appl.} {\bf 10}(015) (2014) 1--8.

\bibitem{Brzezinski2015} T. Brzezi{\'n}ski. Differential smoothness of affine Hopf algebras of Gelfand-Kirillov of dimension two. {\em Colloq. Math.} {\bf 139}(1) (2015) 111--119.

\bibitem{Brzezinski2016} T. Brzezi\'nski. Noncommutative Differential Geometry of Generalized Weyl Algebras. {\em SIGMA Symmetry Integrability Geom. Methods Appl.} {\bf 12}(059) (2016) 1--18.

\bibitem{Brzezinsketal2016Springer} T. Brzezi\'nski, N. Ciccoli, L. Dabrowski, and A. Sitarz. Twisted {R}eality {C}ondition for {D}irac {O}perators. {\em Math. Phys. Anal. Geom.} {\bf 19}(3) (2016) 1--11.

\bibitem{BDR92} T. Brzezi{\'n}ski, H. Dabrowski, and J. Rembieli\'nski. On the quantum differential calculus and the quantum holomorphicity. {\em J. Math. Phys.} {\bf 33}(1) (1992) 19--24.

\bibitem{BrzezinskiElKaoutitLomp2010} T. Brzezi{\'n}ski, L. El. Kaoutit and C. Lomp. Noncommutative integral forms and twisted multi-derivations. {\em J. Noncommut. Geom.} {\bf 4}(2) (2010) 281--312.

\bibitem{BrzezinskiLomp2018} T. Brzezi{\'n}ski and C. Lomp. Differential smoothness of skew polynomial rings. {\em J. Pure Appl. Algebra} {\bf 222}(9) (2017) 2413--2426.

\bibitem{BrzezinskiSitarz2017} T. Brzezi\'nski and A. Sitarz. Smooth geometry of the noncommutative pillow, cones and lens spaces. {\em J. Noncommut. Geom.} {\bf 11}(2) (2017) 413--449.

\bibitem{BrzezinskiSzymanski2016} T. Brzezinski and Szymanski. The $C^{*}$-algebras of quantum lens and weighted projective spaces. {\em J. Noncommut. Geom.} {\bf 12}(1) (2018) 195--215.

\bibitem{BuesoTorrecillasVerschoren2003} J. Bueso, J. G\'omez-Torrecillas and A. Verschoren. {\em Algorithmic Methods in Non-Commutative Algebra: Applications to Quantum Groups}. Mathematical Modelling: Theory and Applications. Springer Dordrecht (2003).

\bibitem{Connes1985} A. Connes. Non-commutative differential geometry. {\em Publ. Math. Inst. Hautes Etudes Sci.} {\bf 62} (1985) 41--144.

\bibitem{Connes1994} A. Connes. Noncommutative Geometry. Academic Press, New York (1994).

\bibitem{ChaconReyes2024} A. Chac\'on and A. Reyes. Noncommutative scheme theory and the Serre-Artin-Zhang-Verevkin theorem for semi-graded rings. {\em J. Noncommut. Geom.} {\bf 19}(2) (2024) 495--532.

\bibitem{CuntzQuillen1995} J. Cuntz and D. Quillen. Algebra extensions and nonsingularity. {\em J. Amer. Math. Soc.} {\bf 8}(2) (1995) 251--289.


\bibitem{DuboisViolette1988} M. Dubois-Violette. D\'erivations et calcul diff\'erentiel non commutatif. {\em C. R. Acad. Sci. Paris, Ser. I} {\bf 307} (1988) 403--408.

\bibitem{DuboisVioletteKernerMadore1990} M. Dubois-Violette, R. Kerner and J. Madore. Noncommutative differential geometry of matrix algebras, {\em J. Math. Phys.} {\bf 31}(2) (1990) 316--322.

\bibitem{Fajardoetal2020} W. Fajardo, C. Gallego, O.  Lezama, A. Reyes, H. Su\'arez, and H. Venegas. {\em Skew PBW Extensions: Ring and Module-theoretic properties, Matrix and Gr\"obner Methods, and Applications}. Algebra and Applications {\bf 28}. Springer Cham (2020).

\bibitem{Fajardoetal2024} W. Fajardo, O. Lezama, C. Payares, A. Reyes and C. Rodr\'iguez. Introduction to Algebraic Analysis on Ore Extensions. In A. Martsinkovsky, editor, {\em Functor Categories, Model Theory, Algebraic Analysis and Constructive Methods FCMTCCT2 2022, Almer\'ia, Spain, July 11--15, Invited and Selected Contributions}, volume 450 of Springer Proceedings in Mathematics \& Statistics, pages 45--116, Springer, Cham, Switzerland.


\bibitem{GallegoLezama2010} C. Gallego and O. Lezama. Gr{\"o}bner {B}ases for {I}deals of $\sigma$-{P}{B}{W} {E}xtensions. {\em Comm. Algebra} {\bf 39}(1) (2010) 50--75.

\bibitem{GelfandKirillov1966} I. M. Gelfand and A. A. Kirillov. On fields connected with the enveloping algebras of {L}ie algebras. {\em Dokl. Akad. Nauk} {\bf 167}(3) (1966) 503--505.

\bibitem{GelfandKirillov1966b} I. M. Gelfand and A. A. Kirillov. Sur les corps li\'es aux algèbres enveloppantes des algèbres de Lie. {\em Publ. Math. IHES} {\bf 31} (1966) 5--19.

\bibitem{Giachettaetal2005} G. Giachetta, L. Mangiarotti and G. Sardanashvily. {\em Geometric and {A}lgebraic {T}opological {M}ethods in {Q}uantum {M}echanics}. World Scientific Publishing, Singapore (2005).

\bibitem{Giesbrechtetal2002} M. W. Giesbrecht, G. J. Reid and Y. Zhang. Non-commutative Gr\"obner bases in Poincar\'e-Birkhoff-Witt extensions, In: V. G. Ganzha, E. W. Mayr and E.V. Vorozhtsov, eds. {\em Proceedings of the Fifth International Workshop on Computer Algebra in Scientific Computing, CASC 2002, Yalta, Ukraine, Technical University of Munich} (2002) pp. 97--106.

\bibitem{Giesbrechtetal2014} M. W. Giesbrecht, G. Labahn and Y. Zhang. Computing {P}opov {F}orms of {M}atrices Over {P}{B}{W} {E}xtensions. In: R. Feng, Ws. Lee and Y. Sato, eds. {\em 9th Asian Symposium} ({\em ASCM2009}), {\em Fukuoka, December 2009, 10th Asian Symposium} ({\em ASCM2012}), {\em Beijing, October 2012, Contributed Papers and Invited Talks}. Computer Mathematics. Springer - Verlag, Berlin, Heidelberg (2014) pp. 61--65.

\bibitem{GomezTorrecillas2014} J. G. G\'omez Torrecillas. Basic Module Theory over Non-commutative Rings with Computational Aspects of Operator Algebras. In: Barkatou, M.,  Cluzeau, T., Regensburger, G., Rosenkranz, M., eds. {\em Algebraic and Algorithmic Aspects of Differential and Integral Operators. AADIOS 2012.} Lecture Notes in Computer Science, Vol. 8372, Springer, Berlin, Heidelberg (2014), pp. 23--82.

\bibitem{Goodearl1990} K. R. Goodearl. Classical {L}ocalizability in {S}olvable {E}nveloping {A}lgebras and {P}oincar\'e-{B}irkhoff-{W}itt {E}xtensions. {\em J. Algebra} {\bf 132}(1) (1990) 243--262.

\bibitem{GoodearlLetzter1994} K. R. Goodearl and E. S. Letzter. {\em Prime {I}deals in {S}kew and $q$-{S}kew {P}olynomial {R}ings}. Mem. Amer. Math. Soc. 521, Vol. 109. Providence, Rhode Island: American Mathematical Society (1994).

\bibitem{GoodearlWarfield2004} K. R. Goodearl and R. B. Warfield Jr. {\em An Introduction to Noncommutative Noetherian Rings}. Algebra and Applications Vol. 61. Cambridge University Press (2004).

\bibitem{Grothendieck1964} A. Grothendieck. \'El\'ements de g\'eom\'etrie alg\'ebrique. {\em Publ. Math. Inst. Hautes Etudes Sci.} {\bf 20} (1964) 5--251.

\bibitem{Hamidizadehetal2020} M. Hamidizadeh, E. Hashemi and A. Reyes. A classification of ring elements in skew PBW extensions over compatible rings. {\em Int. Electron. J. Algebra} {\bf 28}(1) (2020) 75--97.

\bibitem{HashemiKhalilnezhadAlhevaz2017} E. Hashemi, K. Khalilnezhad and A. Alhevaz. $(\Sigma, \Delta)$-Compatible skew PBW extension ring. {\em Kyungpook Math. J.} {\bf 57}(3) (2017) 401--417.

\bibitem{HigueraReyes2023} S. Higuera and A. Reyes. On weak annihilators and nilpotent associated primes of skew PBW extensions. {\em Comm. Algebra} {\bf 51}(11) (2023) 4839--4861.


\bibitem{IsaevPyatovRittenberg2001} A. P. Isaev, P. N. Pyatov and V. Rittenberg. Diffusion algebras. {\em J. Phys. A} {\bf 34}(29) (2001) 5815--5834.



\bibitem{Jordan2000} D. A. Jordan. Down-Up Algebras and Ambiskew Polynomial Rings. {\em J. Algebra} 228(1) (2000) 311--346.

\bibitem{JordanWells1996} D. Jordan and I. Wells. Invariants for automorphisms of certain iterated skew polynomial rings {\em Proc. Edinb. Math. Soc. (2)} {\bf 39}(3) (1996) 461--472.



\bibitem{KandryWeispfenning1990} A. Kandri-Rody and V. Weispfenning. Non-commutative Gr\"obner Bases in Algebras of Solvable Type. {\em J. Symbolic Comput.} {\bf 9}(1) (1990) 1--26.

\bibitem{Karacuha2015} S. Kara\c cuha. Aspects of Noncommutative Differential Geometry. PhD Thesis. Universidade do Porto, Porto, Portugal (2015).

\bibitem{KaracuhaLomp2014} S. Kara\c cuha and C. Lomp, Integral calculus on quantum exterior algebras, {\em Int. J. Geom. Methods Mod. Phys.} {\bf 11}(04) (2014) 1450026.

\bibitem{Krahmer2012} U. Kr\"ahmer. On the Hochschild (co)homology of quantum homogeneous spaces. {\em Israel J. Math.} {\bf 189}(1) (2012) 237--266.

\bibitem{KrauseLenagan2000} G. R. Krause and T. H. Lenagan, {\em Growth of Algebras and Gelfand–Kirillov Dimension. Revised Edition}, Graduate Studies in Mathematics 22 (American Mathematical Society, 2000).





\bibitem{Levandovskyy2005} V. Levandovskyy. Non-commutative computer algebra for polynomial algebras: Gr\"obner bases, applications and implementation. Doctoral Thesis. Universit\"at Kaiserslautern. Kaiserslautern, Germany.

\bibitem{Lezama2020} O. Lezama. Computation of point modules of finitely semi-graded rings. {\em Comm. Algebra} {\bf 48}(2) (2020) 866--878.

\bibitem{Lezama2021} O. Lezama. Some Open Problems in the Context of Skew PBW Extensions and Semi-graded Rings, {\em Commun. Math. Stat.} {\bf 9}(3) (2021) 347--378.

\bibitem{LezamaAcostaReyes2015} O. Lezama, J. P. Acosta and A. Reyes. Prime ideals of skew PBW extensions. {\em Rev. Un. Mat. Argentina} {\bf 56}(2) (2015) 39--55.

\bibitem{LezamaGomez2019} O. Lezama and J. G\'omez. Koszulity and Point Modules of Finitely Semi-Graded Rings and Algebras. {\em Symmetry} {\bf 11}(7) (2019) 881.

\bibitem{LezamaLatorre2017} O. Lezama and E. Latorre. Non-commutative algebraic geometry of semi-graded rings. {\em Internat. J. Algebra Comput.} {\bf 27}(4) (2017) 361--389.

\bibitem{LezamaReyes2014} O. Lezama and A. Reyes. Some Homological Properties of Skew PBW Extensions. {\em Comm. Algebra} {\bf 42}(3) (2014) 1200--1230.

\bibitem{Li2002} H. Li. Noncommutative Gr\"obner Bases and Filtered-Graded Transfer. Lecture Notes in Math, Vol. 1795. Heidelberg: Springer.

\bibitem{Manin1997} Y. Manin. {\em Gauge Field Theory and Complex Geometry. Second Edition}. Grundlehren der mathematischen Wissenschaften, Vol. 289. Springer Berlin, Heidelberg (1997).

\bibitem{MarubayashiZhang1996} H. Marubayashi and Y. Zhang. Maximality of {P}{B}{W} extensions of orders. {\em Comm. Algebra} {\bf 24}(4) (1996) 1377--1388.

\bibitem{Matczuk1988} J. Matczuk. The Gelfand-Kirillov Dimension of Poincar\'e-Birkhoff-Witt Extensions. In: F. Van Oystaeyen and L. Le Bruyn, eds. {\em Perspectives in Ring Theory}. Nato Science Series C ASIC, Vol 233. Springer Dordrecht (1988) pp. 221 -- 226.

\bibitem{McConnellRobson2001} J. C. McConnell and J. C. Robson. {\em Noncommutative Noetherian Rings}. Graduate Studies in Mathematics Vol. 30. American Mathematical Society (2001).

\bibitem{McKayWang1988} J. McKay and S. Wang. An Elementary proof of the automorphism theorem for the polynomial ring in two variables. {\em J. Pure Appl. Algebra} {\bf 52}(1-2) (1988) 91--102.

\bibitem{Nagata1972} M. Nagata. On the automorphism group of $\Bbbk[x, y]$. {\em Lectures in Math.}, Kyoto Univ., Kinokuniya, Tokyo (1972).

\bibitem{NinoRamirezReyes2020} A. Ni\~no, M. C. Ram\'irez and A. Reyes. Associated prime ideals over skew PBW extensions. {\em Comm. Algebra} {\bf 48}(12) (2020) 5038--5055.

\bibitem{NinoReyes2024} A. Ni\~no and A. Reyes. On centralizers and pseudo-multidegree functions for non-commutative rings having PBW bases. {\em J. Algebra Appl.} (2024). \url{https://doi.org/10.1142/S0219498825501099}


\bibitem{Ore1931} O. Ore. Linear Equations in Non-commutative Fields. {\em Ann. of Math. (2)} {\bf 32}(3) (1931) 463--477.

\bibitem{Ore1933} O. Ore, Theory of Non-Commutative Polynomials. {\em Ann. of Math. (2)} {\bf 34}(3) (1933) 480--508.

\bibitem{PyatovTwarock2002} P. N. Pyatov and R. Twarock. Construction of diffusion algebras. {\em J. Math. Phys.} {\bf 43}(6) (2002) 3268--3279.


\bibitem{RedmanPhDThesis1996} I. T. Redman. The Noncommutative Algebraic Geometry of Some Skew Polynomial Rings. PhD Thesis. University of Wisconsin-Milwaukee, Milwaukee, United States (1996).

\bibitem{Redman1999} I. T. Redman. The homogenization of the three dimensional skew polynomial algebras of type I. {\em Comm. Algebra} {\bf 27}(11) (1999) 5587--5602.

\bibitem{Reyes2013} A. Reyes. Gelfand-Kirillov Dimension of Skew PBW Extensions. {\em Rev. Colombiana Mat.} {\bf 47}(1) (2013) 95--111.

\bibitem{ReyesRodriguez2021} A. Reyes and C. Rodr\'iguez. The McCoy Condition on Skew Poincar\'e-Birkhoff-Witt Extensions. {\em Commun. Math. Stat.} {\bf 9}(1) (2021) 1--21.

\bibitem{ReyesSarmiento2022} A. Reyes and C. Sarmiento. On the differential smoothness of 3-dimensional skew polynomial algebras and diffusion algebras. {\em Internat. J. Algebra and Comput.} {\bf 32}(3) (2022) 529--559.

\bibitem{ReyesSuarez2016} A. Reyes and H. Su\'arez. Some Remarks About the Cyclic Homology of Skew PBW Extensions. {\em Ciencia en Desarrollo} {\bf 7}(2) (2016) 99--107.

\bibitem{ReyesSuarez2017} A. Reyes and H. Su\'arez. $\sigma$-{P}{B}{W} Extensions of Skew {A}rmendariz Rings. {\em Adv. Appl. Clifford Algebr.} {\bf 27}(4) (2017) 3197--3224.

\bibitem{ReyesSuarez20173D} A. Reyes and H. Su\'arez. PBW Bases for Some 3-Dimensional Skew Polynomial Algebras. {\em Far East J. Math. Sci.} {\bf 101}(6) (2017) 1207--1228.

\bibitem{ReyesSuarez2020} A. Reyes and H. Su\'arez. Skew Poincar\'e-Birkhoff–Witt extensions over weak compatible rings. {\em J. Algebra Appl.} {\bf 19}(12) (2020) 2050225.

\bibitem{Rosenberg1995} A. Rosenberg. {\em Noncommutative Algebraic Geometry and Representations of Quantized Algebras}. Mathematics and Its Applications, Vol. 330. Kluwer Academic Publishers (1995).

\bibitem{RubianoReyes2024DSBiquadraticAlgebras} A. Rubiano and A. Reyes. Smooth geometry of bi-quadratic algebras on three generators with PBW basis. (2024). \url{https://arxiv.org/abs/2408.16648}

\bibitem{RubianoReyes2024DSDoubleOreExtensions} A. Rubiano and A. Reyes. Smooth geometry of double extension regular algebras of type (14641). (2024). \url{https://arxiv.org/abs/2409.10264}

\bibitem{RubianoReyes2024DSSPBWKt3n} A. Rubiano and A. Reyes. On the differential smoothness of skew PBW extensions over commutative polynomial rings II. Preprint (2024).

\bibitem{Schelter1986} W. F. Schelter. Smooth {A}lgebras. {\em J. Algebra.} {\bf 103}(2) (1986) 677--685.

\bibitem{SeilerBook2010} W. M. Seiler. {\em Involution. {T}he {F}ormal {T}heory of {D}ifferential {E}quations and its {A}pplications in {C}omputer {A}lgebra}. Algorithms and Computation in Mathematics (AACIM) Vol. 24 (Springer Berlin, Heidelberg 2010).

\bibitem{ShestakovUmirbaev2003} I. P. Shestakov and U. U. Umirbaev. The Tame and the Wild Automorphisms of Polynomial Rings in Three Variables. {\em J. Amer. Math. Soc.} {\bf 17} (1) (2003) 197--227.





\bibitem{SuarezReyesSuarez2023} H. Su\'arez, A. Reyes and Y. Su\'arez. Homogenized skew PBW extensions. {\em Arab. J. Math.} {\bf 12}(1) (2023) 247--263.

\bibitem{StaffordZhang1994} J. T. Stafford and J. J. Zhang. Homological Properties of (Graded) Noetherian PI Rings. {\em J. Algebra} {\bf 168}(3) (1994) 988--1026.

\bibitem{Tumwesigyeetal2020} A. B. Tumwesigye, J. Richter and S. Silvestrov. Centralizers in PBW extensions. In: S. Silvestrov, A. Malyarenko and M. Ran\v{c}i\'c, eds. {\em Algebraic Structures and Applications. SPAS 2017, V\"aster\aa s and Stockholm, Sweden, October 4--6}. Springer, Proceedings in Mathematics \& Statistics, Vol. 317, Springer Cham (2020) pp. 469--490.

\bibitem{VandenBergh1998} M. Van den Bergh. A Relation Between Hochschild Homology and Cohomology for Gorenstein Rings. {\em Proc. Amer. Math. Soc.} {\bf 126}(5) (1998) 1345--1348.

\bibitem{VandenEssen2000} A. Van den Essen. {\em Polynomial {A}utomorphisms and the {J}acobian {C}onjecture}. Progress in Mathematics, Vol. 190. Birk\"auser-Verlag, Basel (2000).

\bibitem{Venegas2015} C. Venegas. Automorphisms for Skew PBW Extensions and Skew Quantum Polynomial Rings. {\em Comm. Algebra} {\bf 43}(5) (2015) 1877--1897.

\bibitem{Woronowicz1987} S. L. Woronowicz. Twisted $SU(2)$ {G}roup. An {E}xample of a {N}oncommutative {D}ifferential {C}alculus. {\em Publ. Res. Inst. Math. Sci.} {\bf 23}(1) (1987) 117--181.

\bibitem{ZhangZhang2008} J. J. Zhang and J. Zhang. Double Ore extensions. {\em J. Pure Appl. Algebra} {\bf 212}(12) (2008) 2668--2690.

\bibitem{ZhangZhang2009} J. J. Zhang and J. Zhang. Double extension regular algebras of type (14641) {\em J. Algebra} {\bf 322}(2) (2009) 373--409.

\end{thebibliography}
\end{document}